\documentclass[10pt]{article}

\usepackage{amssymb}
\usepackage{amsmath,amsfonts,mathtools}
\usepackage{amsthm}
\usepackage{latexsym}
\usepackage{mathrsfs}
\usepackage{color}
\usepackage{float}
\usepackage{enumitem}
\usepackage{algorithm,algorithmic}
\usepackage{soul}
\allowdisplaybreaks[3]

\usepackage[a4paper]{geometry}
\newcommand{\beq}{\begin{equation}}
\newcommand{\eeq}{\end{equation}}
\newcommand{\beqa}{\begin{eqnarray}}
\newcommand{\eeqa}{\end{eqnarray}}
\newcommand{\beqas}{\begin{eqnarray*}}
\newcommand{\eeqas}{\end{eqnarray*}}
\newcommand{\ba}{\begin{array}}
\newcommand{\ea}{\end{array}}
\newcommand{\bi}{\begin{itemize}}
\newcommand{\ei}{\end{itemize}}
\newcommand{\nn}{\nonumber}

\DeclareMathOperator*{\argmin}{argmin}  

\newcommand{\mcX}{{\mathcal X}}
\newcommand{\mcY}{{\mathcal Y}}
\newcommand{\mcL}{{\mathcal L}}

\newcommand{\prox}{\mathrm{prox}}
\newcommand{\dom}{\mathrm{dom}}
\newcommand{\dist}{\mathrm{dist}}

\newtheorem{lemma}{Lemma}
\newtheorem{thm}{Theorem}

\newtheorem{defi}{Definition}

\newtheorem{assumption}{Assumption}

\newtheorem{rem}{Remark}

\newcounter{spb}
\def\bfx{{\rm \bf x}}
\def\bfy{{\rm \bf y}}
\def\bfz{{\rm \bf z}}

\def\cL{{\cal L}}

\def\cO{{\cal O}}

\def\tlambda{{\tilde \lambda}}
\def\cI{{\mathscr I}}

\def\tg{{\tilde g}}

\def\tL{{\widetilde L}}

\def\tx{{\hat x}}
\def\ty{{\hat y}}

\def\bR{{\mathbb{R}}}

\def\tf{{\tilde f}}

\def\xe{{x_\epsilon}}
\def\ye{{y_\epsilon}}

\def\bbK{\mathbb{K}}
\def\Lag{{\widehat \cL}}
\def\tcL{{\widetilde \cL}}

\def\bh{{h}}
\def\bH{{H}}

\def\h{\bar h}

\def\tL{{\widetilde L}}
\def\bk{{ k}}

\definecolor{ngreen}{RGB}{38,217,169}

\title{Solving bilevel optimization via sequential minimax optimization\thanks{This work was partially supported by the Air Force Office of Scientific Research under Award FA9550-24-1-0343, the Office of Naval Research under Award N00014-24-1-2702, and the National Science Foundation under Awards 2211491 and 2435911. It was primarily conducted during Sanyou Mei's Ph.D. studies at the University of Minnesota.}}
\author{
Zhaosong Lu
\thanks{
Department of Industrial and Systems Engineering, University of Minnesota, USA (email: {\tt zhaosong@umn.edu}).}
\and
Sanyou Mei
\thanks{Department of Industrial Engineering and Decision Analytics, Hong Kong University of Science and Technology, Hong Kong, China (email: {\tt symei@ust.hk}).}
}

\date{May 12, 2024 (Revised: November 5, 2025)}

\begin{document}
\maketitle

\begin{abstract}
In this paper we propose a sequential minimax optimization (SMO) method for solving a class of constrained bilevel optimization problems in which the lower-level part is a possibly nonsmooth convex optimization problem, while the upper-level part is a possibly nonconvex optimization problem. Specifically, SMO applies a first-order method to solve a sequence of minimax subproblems, which are obtained by employing a hybrid of modified augmented Lagrangian and penalty schemes on the bilevel optimization problems. Under suitable assumptions, we establish an \emph{operation complexity} of $\mathcal{O}(\varepsilon^{-7}\log\varepsilon^{-1})$ and $\mathcal{O}(\varepsilon^{-6}\log\varepsilon^{-1})$, measured in terms of fundamental operations, for SMO in finding an $\varepsilon$-KKT (Karush-Kuhn-Tucker) solution of the bilevel optimization problems with merely convex and strongly convex lower-level objective functions, respectively. 
The latter result improves the previous best-known operation complexity by a factor of $\varepsilon^{-1}$. Preliminary numerical results demonstrate significantly superior computational performance compared to the recently developed first-order penalty method.
\end{abstract}


\medskip


\section{Introduction} \label{introduction}

Bilevel optimization (BO) is a two-level hierarchical optimization, which is typically in the form of 
 \begin{equation}\label{prob}
\begin{array}{rl}
f^*=\min & f(x,y)\\
\mbox{s.t.}& y\in\argmin\limits_z\{\tf(x,z)|\tg(x,z)\leq0\}.\footnote{}
\end{array}
\end{equation}
\footnotetext{For ease of reading, throughout this paper the tilde symbol is particularly used for the functions related to the lower-level optimization problem. Besides,  ``$\argmin$'' denotes the set of optimal solutions of the associated problem.} BO has widely been used in many areas, including adversarial training \cite{Madry18,mirrlees1999theory,szegedy2013intriguing},   continual learning \cite{lopez2017gradient}, hyperparameter tuning \cite{bennett2008bilevel,franceschi2018bilevel,okuno2021lp}, image reconstruction \cite{crockett2022bilevel}, meta-learning \cite{bertinetto2018meta,ji2020convergence,rajeswaran2019meta}, neural architecture search \cite{feurer2019hyperparameter,liu2018darts}, reinforcement learning \cite{hong2023two,konda1999actor}, and Stackelberg games \cite{von2010market}. More applications about it can be found in \cite{bard2013practical,colson2007overview,dempe2002foundations,dempe2015bilevel,
dempe2020bilevel,shimizu2012nondifferentiable} and the references therein. 
Theoretical properties including optimality conditions of \eqref{prob} have been extensively
studied in the literature (e.g., see \cite{dempe2020bilevel,dempe2013bilevel,ma2021combined,vicente1994bilevel,ye2020constraint}).  

Numerous methods have been developed for solving some special cases of \eqref{prob}. For example,  constraint-based methods \cite{hansen1992new,shi2005extended}, deterministic gradient-based methods \cite{franceschi2017forward,franceschi2018bilevel,grazzi2020iteration,hu2023improved,maclaurin2015gradient,
pedregosa2016hyperparameter,rajeswaran2019meta}, and stochastic gradient-based methods \cite{chen2022single,guo2021randomized,hong2023two,huang2021biadam,
huang2025efficiently,ji2021bilevel,khanduri2021near,kwon2023fully,li2022fully,yang2021provably} 
were proposed for solving  \eqref{prob} with $\tg\equiv 0$,  $f$,  $\tf$ being smooth, and $\tf$ being \emph{strongly convex} with respect to $y$. For a similar case as this but with $\tf$ being \emph{convex} with respect to $y$, a zeroth-order method was recently proposed in \cite{chen2023bilevel}, and also numerical methods were developed in \cite{li2023novel,sow2022primal,liu2022bome} by solving \eqref{prob} as a single or sequential smooth constrained optimization problems. Besides, when all the functions in \eqref{prob} are smooth and $\tf,\,\tg$ are \emph{convex} with respect to $y$, gradient-type methods were proposed by solving  a mathematical program with equilibrium constraints resulting from replacing the lower-level optimization problem of \eqref{prob} by its first-order optimality conditions (e.g., see \cite{allende2013solving,luo1996mathematical,outrata1998nonsmooth}).  Furthermore, a single-loop gradient method was recently introduced in \cite{yao2024constrained} based on a novel reformulation of the bilevel optimization problem as a single-level smooth optimization problem using Moreau envelope of the Lagrangian function of the lower-level problem.  Recently, difference-of-convex (DC) algorithms were developed in \cite{ye2023difference} for solving \eqref{prob} with $f$ being a DC function, and $\tf$, $\tg$ being convex functions. Lately, a practically efficient multi-stage gradient descent and ascent algorithm was developed in \cite{wang2023effective} for \eqref{prob} with $\tg\equiv 0$, $f$ being convex and Lipschitz continuous, and $\tf$ being strongly convex and Lipschitz smooth via solving the aforementioned minimax reformulation of \eqref{prob}.  In addition, penalty methods were proposed in \cite{ishizuka1992double,lu2024first-bilevel,shen2023penalty} for solving \eqref{prob}. Notably, the paper \cite{lu2024first-bilevel} demonstrates for the first time that BO can be approximately solved as minimax optimization. Specifically, it reformulates BO as minimax optimization by a novel double penalty scheme and proposes a first-order method with complexity guarantees for BO via solving a single minimax problem. In addition, a novel single-loop Hessian-free algorithm based on a doubly regularized gap function was proposed in \cite{yao2024overcoming} for solving \eqref{prob}. More discussion on algorithmic development for BO can be found in  \cite{bard2013practical,colson2007overview,dempe2020bilevel,liu2021investigating,
sinha2017review,vicente1994bilevel}) and the references therein.

In this paper, we consider problem \eqref{prob} under similar assumptions as in \cite{lu2024first-bilevel}.  Specifically, we assume that problem \eqref{prob} has at least one optimal solution and satisfies the following assumptions.
\begin{assumption}\label{a1}
\begin{enumerate}[label=(\roman*)]
\item $f(x,y)=f_1(x,y)+f_2(x)$ and $\tf(x,y)=\tf_1(x,y)+\tf_2(y)$ are respectively $L_f$- and $L_\tf$-Lipschitz continuous on $\mcX \times \mcY$ with $\mcX:=\dom\,f_2$ and $\mcY:=\dom\,\tf_2$, where $\tf_1(x,\cdot)$ is $\sigma$-strongly-convex for any given $x\in\mcX$ for some $\sigma\geq0$,\footnote{$\tf_1(x,\cdot)$ is either merely convex for any given $x\in\mcX$ if $\sigma=0$ or strongly convex with  parameter $\sigma$ if $\sigma>0$. }  $f_1$, $\tf_1$ are respectively $L_{\nabla f_1}$- and $L_{\nabla \tf_1}$-smooth on $\mcX \times \mcY$, and $f_2: \bR^n \to \bR\cup \{+\infty\}$  and $\tf_2:\bR^m \to \bR\cup \{+\infty\}$ are proper closed convex functions.
\item The proximal operators associated with $f_2$ and $\tf_2$ can be exactly evaluated. 
\item $\tg:\bR^n\times\bR^m\to\bR^l$ is $L_\tg$-Lipschitz continuous and $L_{\nabla \tg}$-smooth on $\mcX\times\mcY$, and $\tg_i(x,\cdot)$ is convex for all $x\in\mcX$ and $i=1,2,\dots, l$. 
\item The sets $\mcX$ and $\mcY$ (namely, $\dom\,f_2$ and $\dom\,\tf_2$) are compact.
\end{enumerate}
\end{assumption}

Due to the sophisticated structure described in Assumption \ref{a1}, existing methods except the first-order penalty method \cite{lu2024first-bilevel} are generally not applicable to problem \eqref{prob}. In particular, instead of solving  \eqref{prob} directly,  the latter method applies a first-order method \cite{lu2024first} to solve an approximate counterpart of  \eqref{prob} given by a single minimax problem 
\beq \label{penalty-prob}
\min_{x,y}\max_z f(x,y)+\rho\big(\tf(x,y)+\mu\left\|[\tg(x,y)]_+\right\|^2-\tf(x,z)-\mu\left\|[\tg(x,z)]_+\right\|^2\big)
\eeq 
with a suitable choice of penalty parameters $\rho, \mu>0$. Notice that the minimax problem \eqref{penalty-prob} can be obtained from \eqref{prob} by performing two steps: (i) apply the classical quadratic penalty scheme to the lower-level problem of \eqref{penalty-prob} and approximate  \eqref{prob} by a simpler BO problem,  $\min_{x,y}\{f(x,y)|
 y\in\argmin_z \tf(x,z)+\mu\left\|[\tg(x,z)]_+\right\|^2\}$, which can be viewed as
\beq \label{penalty-prob1}
\min_{x,y}\big\{f(x,y)\big|\tf(x,y)+\mu\left\|[\tg(x,y)]_+\right\|^2 \leq \min\limits_{z}\tf(x,z)+\mu\left\|[\tg(x,z)]_+\right\|^2\big\},
\eeq
and (ii) apply a penalty method to \eqref{penalty-prob1} to obtain the minimax problem \eqref{penalty-prob}. While this method enjoys complexity guarantees for finding an approximate KKT (Karush-Kuhn-Tucker) solution of \eqref{prob}, it may suffer from practical inefficiency issues. Specifically, the penalty parameters are pre-chosen to achieve a desired operational complexity and may be overly large in practice. Additionally, the classical quadratic penalty scheme is used to obtain the minimax problem \eqref{penalty-prob}, and its associated penalty parameter could be much larger than the one associated with an augmented Lagrangian scheme.

To address the aforementioned issues, in this paper we propose a novel sequential minimax optimization (SMO) method to solve problem \eqref{prob}, which substantially outperforms the first-order penalty method \cite{lu2024first-bilevel} as observed in our numerical experiment.  Specifically, instead of using the classical quadratic penalty scheme for the lower-level problem of \eqref{prob}, we propose a new augmented Lagrangian scheme by replacing the quadratic penalty function $\tf(x,y)+\mu\left\|[\tg(x,y)]_+\right\|^2$ with a  \emph{modified} augmented Lagrangian function $\tf(x,z)+\frac{1}{2\rho\mu}\left(\|[\lambda+\mu\tg(x,z)]_+\|^2-\|\lambda\|^2\right)$ for $\lambda\in\bR^l_+$ and $\mu>0$.\footnote{The standard augmented Lagrangian function associated with the lower-level problem of \eqref{prob} is $\tf(x,z)+\frac{1}{2\mu}\left(\|[\lambda+\mu\tg(x,z)]_+\|^2-\|\lambda\|^2\right)$.  Therefore, $\tf(x,z)+\frac{1}{2\rho\mu}\left(\|[\lambda+\mu\tg(x,z)]_+\|^2-\|\lambda\|^2\right)$ can be viewed as a modified augmented Lagrangian function. Its advantage over the standard augmented Lagrangian function will be discussed in Section \ref{sec:SMO}.} Performing such a replacement in \eqref{penalty-prob} results in a new minimax problem
\beq \label{AL-minimax-subprob}
\min_{x,y}\max_z f(x,y)+\rho\Big(\tf(x,y)+\frac{1}{2\rho\mu}\|[\lambda+\mu\tg(x,y)]_+\|^2-\tf(x,z)-\frac{1}{2\rho\mu}\|[\lambda+\mu\tg(x,z)]_+\|^2\Big).
\eeq 
Our SMO method solves a sequence of minimax subproblems in the form of \eqref{AL-minimax-subprob}. Specifically, let $\{\rho_k\}$, $\{\mu_k\}$, $(x^0,y^0,z^0,\lambda^0)$ be given. At each iteration $k \geq 0$, SMO finds an approximate solution $(x^{k+1},y^{k+1},z^{k+1})$ of \eqref{AL-minimax-subprob} with $\lambda=\lambda^k$, $\rho=\rho_k$ and $\mu=\mu_k$, starting at $(x^k,y^k,z^k)$, and then updates $\lambda^{k+1}$ according to
$\lambda^{k+1}=[\lambda^k+\mu_k\tg(x^{k+1},z^{k+1})]_+$. The resulting SMO enjoys the following notable features. 
\bi
\item It uses only the first-order information of the problem. Specifically, its fundamental operations consist only of gradient evaluations of $\tg$ and the smooth component of $f$ and $\tf$ and also proximal operator evaluations of the nonsmooth component of $f$ and $\tf$ (see Algorithm \ref{AL-alg}).
\item It has theoretical guarantees on operation complexity, which is measured by the aforementioned fundamental operations, for finding an $\varepsilon$-KKT solution of \eqref{prob}. Specifically,  it enjoys an operation complexity of $\cO(\varepsilon^{-7}\log \varepsilon^{-1})$ when the lower objective function $\tf_1(x,\cdot)$ is merely convex (see Theorem \ref{complexity}).  Moreover, it enjoys an operation complexity of $\cO(\varepsilon^{-6}\log \varepsilon^{-1})$ when $\tf_1(x,\cdot)$ is strongly convex, which \emph{ improves} the previous best-known operation complexity \cite[Theorem 5]{lu2024first-bilevel} by a factor of $\varepsilon^{-1}$ (see Theorem \ref{complexity-str}).
\item It demonstrates significantly superior computational performance compared to the first-order penalty method  \cite{lu2024first-bilevel} (see Section \ref{sec:exp}).
\ei

The rest of this paper is organized as follows. In Subsection \ref{notation}, we introduce some notation and terminology. 
In Section \ref{sec:SMO}, we propose a sequential minimax optimization method for solving \eqref{prob} and study its complexity. Preliminary numerical results and the proofs of the main results are  presented in Sections \ref{sec:exp} and \ref{sec:proof}, respectively.

 \subsection{Notation and terminology}  \label{notation}
The following notation will be used throughout this paper. Let $\bR^n$ denote the Euclidean space of dimension $n$ and $\bR^n_+$ denote the nonnegative orthant in $\bR^n$. The standard inner product, $l_1$-norm and Euclidean norm are denoted by $\langle\cdot,\cdot\rangle$, $\|\cdot\|_1$ and $\|\cdot\|$, respectively. For any $v\in\bR^n$, let $v_+$ denote the nonnegative part of $v$, that is, $(v_+)_i=\max\{v_i,0\}$ for all $i$. For any two vectors $u$ and $v$, $(u;v)$ denotes the vector resulting from stacking $v$ under $u$. Given a point $x$ and a closed set $S$ in $\bR^n$, let $\dist(x,S)=\min_{x'\in S} \|x'-x\|$ and $\cI_S$ denote the indicator function associated with $S$.

A function or mapping $\phi$ is said to be \emph{$L_{\phi}$-Lipschitz continuous} on a set $S$ if $\|\phi(x)-\phi(x')\| \leq L_{\phi} \|x-x'\|$ for all $x,x'\in S$. In addition, it is said to be \emph{$L_{\nabla\phi}$-smooth} on $S$ if $\|\nabla\phi(x)-\nabla\phi(x')\| \leq L_{\nabla\phi} \|x-x'\|$ for all $x,x'\in S$. For a closed convex function $p:\bR^n\to \bR\cup\{+\infty\}$, associated with $p$ is denoted by  
$\prox_p$,  that is,
\[
\prox_p(x) = \argmin_{x'\in\bR^n} \left\{ \frac{1}{2}\|x' - x\|^2 + p(x') \right\} \quad \forall x \in \bR^n.
\]
Given that evaluation of $\prox_{\gamma p}(x)$ is often as cheap as $\prox_p(x)$, we count the evaluation of $\prox_{\gamma p}(x)$ as one evaluation of proximal operator of $p$ for any $\gamma>0$ and $x\in\bR^n$. 

For a lower semicontinuous function $\phi:\bR^n\to \bR\cup\{\infty\}$, its \emph{domain} is the set $\dom\, \phi := \{x| \phi(x)<\infty\}$. The \emph{upper subderivative} of $\phi$ at $x\in \dom\, \phi$ in a direction $d\in\bR^n$ is defined by
\[
\phi'(x;d) = \limsup\limits_{x' \stackrel{\phi}{\to} x,\, t \downarrow 0} \inf_{d' \to d} \frac{\phi(x'+td')-\phi(x')}{t},
\] 
where $t\downarrow 0$ means both $t > 0$ and $t\to 0$, and $x' \stackrel{\phi}{\to} x$ means both $x' \to x$ and $\phi(x')\to \phi(x)$. The \emph{subdifferential} of $\phi$ at $x\in \dom\, \phi$ is the set 
\[
\partial \phi(x) = \{s\in\bR^n\big| s^T d \leq \phi'(x; d) \ \ \forall d\in\bR^n\}.
\]
We use $\partial_{x_i} \phi(x)$ to denote the subdifferential with respect to $x_i$.  
In addition, for an upper semicontinuous function $\phi$, its subdifferential is defined as $\partial \phi=-\partial (-\phi)$. If $\phi$ is locally Lipschitz continuous, the above definition of subdifferential coincides with the Clarke subdifferential. Besides, if $\phi$ is convex, it coincides with the ordinary subdifferential for convex functions. Also, if $\phi$ is continuously differentiable at $x$ , we simply have $\partial \phi(x) = \{\nabla \phi(x)\}$, where $\nabla \phi(x)$ is the gradient of $\phi$ at $x$. In addition, it is not hard to verify that $\partial (\phi_1+\phi_2)(x)=\nabla \phi_1(x)+\partial \phi_2(x)$ if $\phi_1$ is continuously differentiable at $x$ and $\phi_2$ is lower or upper semicontinuous at $x$. See \cite{clarke1990optimization,ward1987nonsmooth} for more details.

Finally, we introduce an (approximate) primal-dual stationary point (e.g., see \cite{dai2024optimality,dai2020optimality,kong2021accelerated}) for a general minimax problem
\begin{equation}\label{eg}
\min_{x}\max_{y}\Psi(x,y),
\end{equation}
where $\Psi(\cdot,y): \bR^n \to \bR \cup\{+\infty\}$ is a lower semicontinuous function, and $\Psi(x,\cdot): \bR^m \to \bR \cup\{-\infty\}$ is an upper semicontinuous function.
\begin{defi}  \label{def1}
A point $(x,y)$ is said to be a primal-dual stationary point of the minimax problem \eqref{eg} if 
\[
0 \in \partial_x\Psi(x,y), \quad 0\in\partial_y\Psi(x,y).
\]
In addition, for any $\epsilon>0$, a point $(\xe,\ye)$ is said to be an $\epsilon$-primal-dual stationary point of the minimax problem \eqref{eg} if
\begin{equation*}
\dist\left(0,\partial_x\Psi(\xe,\ye)\right)\leq\epsilon,\quad\dist\left(0,\partial_y\Psi(\xe,\ye)\right)\leq\epsilon.
\end{equation*}
\end{defi}
 
\section{A sequential minimax optimization method for problem \eqref{prob}} \label{sec:SMO}

As discussed in Section \ref{introduction}, the first-order penalty method \cite{lu2024first-bilevel} may suffer practical inefficiency issues due to possibly overly large penalty parameters. To address these issues, in this section we propose a sequential minimax optimization (SMO) method for finding an approximate KKT solution of \eqref{prob}, which substantially outperforms the first-order penalty method \cite{lu2024first-bilevel} as observed in our numerical experiment. 

To motivate our SMO method,  we first apply a \emph{modified} augmented Lagrangian (AL) scheme to migrate the constraint $\tg(x,y)\leq0$ of the lower-level problem of \eqref{prob} to its objective function and obtain an approximation to  \eqref{prob} given by
\begin{equation}\label{prob-AL-ref}
\begin{array}{rl}
\min & f(x,y)\\
\mbox{s.t.}& y\in\argmin\limits_z\{\tf(x,z)+\frac{1}{2\rho\mu}\left(\|[\lambda+\mu\tg(x,z)]_+\|^2-\|\lambda\|^2\right)\}
\end{array}
\end{equation}
for some $\rho, \mu>0$ and $\lambda\in\bR^l_+$.  By applying a penalty scheme, problem \eqref{prob-AL-ref} can be approximated by 
{\small
\[
\min\limits_{x,y} f(x,y)+\rho\left(\tf(x,y)+\frac{1}{2\rho\mu}\left(\|[\lambda+\mu\tg(x,y)]_+\|^2-\|\lambda\|^2\right)
-\min_z\Big\{\tf(x,z)+\frac{1}{2\rho\mu}\left(\|[\lambda+\mu\tg(x,z)]_+\|^2-\|\lambda\|^2\right)\Big\}\right),
\]
}
which is equivalent to the minimax problem
\beq\label{AL-sub0}
\min_{x,y}\max_{z}\mcL(x,y,z,\lambda;\rho,\mu),
\eeq
 where 
\begin{align}
\mcL(x,y,z,\lambda;\rho,\mu) :=&f(x,y)+\rho\tf(x,y)+\frac{1}{2\mu}\|[\lambda+\mu\tg(x,y)]_+\|^2-\rho\tf(x,z)-\frac{1}{2\mu}\|[\lambda+\mu\tg(x,z)]_+\|^2. \label{Lag}
\end{align}

As observed above, the function $\tf(x,z)+\frac{1}{2\rho\mu}\left(\|[\lambda+\mu\tg(x,z)]_+\|^2-\|\lambda\|^2\right)$, which serves as a modified AL function for the lower-level problem of \eqref{prob}, induces the minimax problem \eqref{AL-sub0}, where the penalty parameters $\rho$ and $\mu$ are separately associated with the lower-level objective and constraint functions, respectively. This separation plays a crucial role in designing a practically efficient SMO method that achieves the desired operation complexity. In contrast,  the standard AL function $\tf(x,z)+\frac{1}{2\mu}\left(\|[\lambda+\mu\tg(x,z)]_+\|^2-\|\lambda\|^2\right)$ for the lower-level problem of \eqref{prob} leads to a minimax problem that lacks this property and over-penalizes the constraints. As a result, the corresponding first-order method based on such a minimax problem cannot guarantee finding an $\varepsilon$-KKT solution of problem \eqref{prob}, since the condition $|\tilde f(x^k,y^{k})-\tilde f^*(x^{k})|\leq\epsilon$ may fail to hold  for the generated solution sequence when $k$ is sufficiently large, which is required by the definition of an $\varepsilon$-KKT point (see Definition~\ref{def2}). Furthermore, under Assumption \ref{a1}, one can observe that $\mcL$ possesses the following desirable structure.
\bi
\item For any given $\rho,\mu>0$ and $\lambda\in\bR_+^l$, $\mcL$ is the sum of smooth function $h(x,y,z)$ with Lipschitz continuous gradient and possibly nonsmooth function $p(x,y)-q(z)$ with exactly computable proximal operator, where
\begin{align*}
&h(x,y,z)=f_1(x,y)+\rho\tf_1(x,y)+\frac{1}{2\mu}\|[\lambda+\mu\tg(x,y)]_+\|^2-\rho\tf_1(x,z)-\frac{1}{2\mu}\|[\lambda+\mu\tg(x,z)]_+\|^2,\\
&p(x,y)=f_2(x)+\rho\tf_2(y), \quad q(z)=\rho\tf_2(z).
\end{align*}
\item $\mcL$ is nonconvex in $(x,y)$ but $\rho\sigma$-strongly-concave in $z$.
\ei

Thanks to the above nice structure of $\mcL$, an approximate primal-dual stationary point of problem \eqref{AL-sub0} can be suitably found by Algorithm \ref{mmax-alg2} (see Appendix \ref{appendix-B}).
Additionally, recall from the above discussion that the minimax problem provides an approximation to the bilevel optimization problem \eqref{prob}. Based on these observations, we propose solving problem \eqref{prob} by iteratively solving a sequence of minimax subproblems in the form of \eqref{AL-sub0}, similar to the classical AL method for constrained nonlinear optimization.

Specifically, let $\{(\rho_k,\mu_k)\}$ be a sequence of penalty parameters. Given the current iterate $(x^k,y^k,z^k,\lambda^k)$, our method calls Algorithm 2 or 3 (see Appendix \ref{appendix-A}), depending on $\sigma=0$ or $\sigma>0$, to obtain an approximate solution $y_{\rm init}^k$ of $\min_z\tL(x^{k},z,\lambda^k;\rho_k,\mu_k)$, where
\beq\label{tLag}
\tcL(x,z,\lambda;\rho,\mu):=\tf(x,z)+\frac{1}{2\rho\mu}\|[\lambda+\mu \tg(x,z)]_+\|^2.
\eeq
It then uses $(x^k,y^k,y_{\rm init}^k)$ as the initial point and calls Algorithm \ref{mmax-alg2} (see Appendix \ref{appendix-B}) with properly chosen parameters to find an approximate primal-dual stationary point $(x^{k+1},y^{k+1},z^{k+1})$ of $\min_{x,y}\max_z\mcL(x,y,\\ z,\lambda^k;\rho_k,\mu_k)$. Subsequently, it updates $\lambda^{k+1}$ in a standard manner.

Let $\tg_{\rm hi}=\max\{\|\tg(x,y)\| |(x,y)\in\mcX\times\mcY\}$. We now present our method for solving problem \eqref{prob} below.

\begin{algorithm}[H]
\caption{A sequential minimax optimization (SMO) method for \eqref{prob}}\label{AL-alg}
\begin{algorithmic}[1]
\REQUIRE $\varepsilon,\tau\in(0,1)$, $\epsilon_0\in(\tau\varepsilon,1]$, $x^0\in\mcX$, $z^0\in\mcY$, $\epsilon_k=\epsilon_0\tau^k$, $\rho_k=\epsilon_k^{-1}$, $\mu_k=\epsilon_k^{-3}$,  and $\lambda^0\in\bR_+^l$.
\FOR{$k=0,1\dots$}
\STATE Call Algorithm \ref{opt} (see Appendix \ref{appendix-A}) if $\sigma=0$ or Algorithm \ref{opt-str} (see Appendix \ref{appendix-A})  if $\sigma>0$ 
with $\Psi(\cdot)\leftarrow\tcL(x^k,\cdot,\lambda^k;\rho_k,\mu_k)$, $\tilde\epsilon\leftarrow\epsilon_k$, $\sigma_\phi\leftarrow\sigma$, $L_{\nabla\phi}\leftarrow\tL_k$, $\tilde x^0\leftarrow y^k$ to find an approximate solution $y_{\rm init}^k$ of $\min_z\tL(x^{k},z,\lambda^k;\rho_k,\mu_k)$ such that
\beq\label{y-nf}
\tcL(x^k,y_{\rm init}^k,\lambda^k;\rho_k,\mu_k)-\min_z\tcL(x^k,z,\lambda^k;\rho_k,\mu_k)\leq\epsilon_k,
\eeq
where $\tcL$ is given in \eqref{tLag} and
\beq\label{tLk}
\tL_k= L_{\nabla\tf_1}+\rho_k^{-1}(\mu_kL_\tg^2+\mu_k\tg_{\rm hi}L_{\nabla \tg}+\|\lambda^k\|L_{\nabla\tg}).
\eeq
\STATE Call Algorithm \ref{mmax-alg2} (see Appendix \ref{appendix-B}) with $\epsilon\leftarrow\epsilon_k$, $\hat x^0\leftarrow (x^k,y^k_{\rm init})$, $\hat y^0\leftarrow z^k$, $L_{\nabla h}\leftarrow L_k$, and $\hat\epsilon_0\leftarrow\epsilon_k/(2\sqrt{\mu_k})$ if $\sigma=0$ and $\hat\epsilon_0\leftarrow\epsilon_k/2$ if $\sigma>0$ to find an $\epsilon_k$-primal-dual stationary point $(x^{k+1},y^{k+1},z^{k+1})$ of 
\begin{equation}\label{AL-sub}
\min_{x,y}\max_z\mcL(x,y,z,\lambda^k;\rho_k,\mu_k),
\end{equation}
where
\beq\label{Lk}
L_k= L_{\nabla f_1}+2\rho_kL_{\nabla\tf_1}+2\mu_kL_\tg^2+2\mu_k\tg_{\rm hi}L_{\nabla \tg}+2\|\lambda^k\|L_{\nabla\tg}.
\eeq
\STATE Set $\lambda^{k+1}=[\lambda^k+\mu_k\tg(x^{k+1},z^{k+1})]_+$.
\STATE If $\epsilon_k\leq\varepsilon$, terminate the algorithm and output $(x^{k+1},y^{k+1})$.
\ENDFOR
\end{algorithmic}
\end{algorithm}

We next provide some remarks regarding the well-definedness of Algorithm~\ref{AL-alg}.

\begin{rem}
Notice that $\tL(x^k,y,\lambda^k;\rho_k,\mu_k)=\phi(y)+\tf_2(y)$ with {\small $\phi(y)=\tf_1(x^k,y)+\|[\lambda^k+\mu_k\tg(x^k,y)]_+\|^2/(2\rho_k\mu_k)$}. By Assumption \ref{a1} and \eqref{tghi}, one can see that $\phi$ is $\tL_k$-smooth and $\sigma$-strongly-convex on $\dom\,P$ and the proximal operator of $\tf_2$ can be exactly evaluated. It then follows from this and Theorems \ref{thm-opt} and \ref{thm-opt-str} (see Appendix \ref{appendix-A}) that $y_{\rm init}^k$ satisfying \eqref{y-nf} can be successfully found in step 2 of Algorithm~\ref{AL-alg} by applying Algorithm \ref{opt} or \ref{opt-str} to the problem $\min_z\tL(x^{k},z,\lambda^k;\rho_k,\mu_k)$. In addition,  by Theorem \ref{mmax-thm} (see Appendix \ref{appendix-B}), one can see that an $\epsilon_k$-primal-dual stationary point of \eqref{AL-sub} can be successfully found in step 3 of Algorithm~\ref{AL-alg} by applying Algorithm \ref{mmax-alg2} to problem \eqref{AL-sub}. Consequently, Algorithm~\ref{AL-alg} is well-defined.
\end{rem}

\subsection{Complexity results for Algorithm \ref{AL-alg}}\label{sec:thm}
In this subsection we study \emph{iteration and operation complexity} for Algorithm \ref{AL-alg}. In particular, in order to characterize the approximate solution found by Algorithm~\ref{AL-alg}, we first introduce a notion called an $\varepsilon$-KKT solution of problem \eqref{prob}. Then we establish iteration and operation complexity of Algorithm \ref{AL-alg} for finding an $\cO(\varepsilon)$-KKT solution of \eqref{prob}.

For notational convenience, we define
\beq
\tf^*(x):=\min_z\{\tf(x,z)|\tg(x,z)\leq0\}.\label{tfstarx}
\eeq
Observe that problem \eqref{prob} can be equivalently reformulated as
\beq\label{prob-ref}
\min_{x,y}\{f(x,y)|\tf(x,y)\leq \tf^*(x), \  \tg(x,y)\leq0\}.
\eeq
The Lagrangian function associated with \eqref{prob-ref} is given by
\begin{equation}\label{def-L}
\Lag(x,y,\rho,\lambda_\bfy)=f(x,y)+\rho(\tf(x,y)-\tf^*(x))+\langle\lambda_\bfy,\tg(x,y)\rangle.
\end{equation}
In the same spirit of classical constrained optimization, one would naturally be interested in a KKT solution $(x,y)$ of \eqref{prob-ref}, namely, $(x,y)$ satisfies 
\beq \label{constr-cond}
\tf(x,y)\leq \tf^*(x), \quad \tg(x,y)\leq0, \quad \rho(\tf(x,y)- \tf^*(x))=0, \quad \langle \lambda_\bfy, \tg(x,y) \rangle=0,
\eeq
and moreover $(x,y)$ is a stationary point of the problem
\beq \label{min-prob1}
\min_{x',y'}\Lag(x',y',\rho,\lambda_\bfy)
\eeq 
for some $\rho\geq 0$ and $\lambda_\bfy\in\bR^l_+$. Yet, due to the sophisticated problem structure, characterizing a stationary point of  \eqref{min-prob1} is generally difficult. On the other hand, notice from Lemma \ref{dual-bnd} and
\eqref{def-L} that problem \eqref{min-prob1} is equivalent to the minimax problem
\[
\min_{x',y',\lambda_\bfz'}\max_{z'} \big\{f(x',y')+\rho\big(\tf(x',y')-\tf(x',z')-\langle\lambda_\bfz',\tg(x',z')\rangle\big)+\langle\lambda_\bfy,\tg(x',y')\rangle+\cI_{\bR^l_+}(\lambda_\bfz')\big\},\footnote{$\cI_{\bR^l_+}(\cdot)$ denotes the indicator function associated with the set $\bR^l_+$.}
\]
whose stationary point $(x,y,\lambda_\bfz,z)$, according to Definition \ref{def1} and Assumption \ref{a1}, satisfies 
\begin{align}
&0\in\partial f(x,y)+\rho\partial\tf(x,y) - \rho(\nabla_x\tf(x,z)+\nabla_x\tg(x,z)\lambda_\bfz;0)+\nabla \tg(x,y)\lambda_\bfy, \label{kkt-c1}\\ 
&0\in\rho(\partial_z\tf(x,z)+\nabla_z\tg(x,z)\lambda_\bfz),  \label{kkt-c2} \\
& \lambda_\bfz \in \bR^l_+, \quad \tg(x,z) \leq 0, \quad \langle \lambda_\bfz, \tg(x,z)\rangle=0.\footnote{} \label{kkt-c3}
\end{align}
\footnotetext{The relations in \eqref{kkt-c3} are equivalent to $0 \in -\tg(x,z)+ \partial \cI_{\bR^l_+}(\lambda_\bfz)$.}
Based on this observation, the equivalence of \eqref{prob} and \eqref{prob-ref}, and also the fact that \eqref{constr-cond} is equivalent to 
\beq \label{constr-cond-new}
\tf(x,y)=\tf^*(x), \quad \tg(x,y)\leq0, \quad \langle \lambda_\bfy, \tg(x,y) \rangle=0,
\eeq
we are instead interested in a (weak) KKT solution of problem \eqref{prob} and its inexact counterpart that are defined below.

\begin{defi}[{\bf KKT solution and $\epsilon$-KKT solution}] \label{def2}
The pair $(x,y)$ is said to be a KKT solution of problem \eqref{prob} if there exists $(z,\rho,\lambda_\bfy,\lambda_\bfz)\in\bR^m\times\bR_+\times\bR^l_+\times\bR^l_+$ such that 
\eqref{kkt-c1}-\eqref{constr-cond-new} hold. In addition, for any $\varepsilon>0$, $(x,y)$ is said to be an $\varepsilon$-KKT solution of problem \eqref{prob} if there exists $(z,\rho,\lambda_\bfy,\lambda_\bfz)\in\bR^m\times\bR_+\times\bR^l_+\times\bR^l_+$ such that
\begin{align*}
&\dist\big(0,\partial f(x,y)+\rho\partial\tf(x,y) - \rho\big(\nabla_x\tf(x,z)+\nabla_x\tg(x,z)\lambda_\bfz; 0\big)+\nabla \tg(x,y)\lambda_\bfy\big)\leq\varepsilon,\\
& \dist\big(0,\rho(\partial_z\tf(x,z)+\nabla_z\tg(x,z)\lambda_\bfz)\big)\leq\varepsilon, \\
& \|[\tg(x,z)]_+\| \leq \varepsilon, \quad |\langle \lambda_\bfz, \tg(x,z)\rangle| \leq\varepsilon, \\ 
&|\tf(x,y)-\tf^*(x)|\leq\varepsilon, \quad \|[\tg(x,y)]_+\|\leq\varepsilon, \quad |\langle \lambda_\bfy, \tg(x,y) \rangle|\leq\varepsilon,
\end{align*}
where $\tf^*$ is defined in \eqref{tfstarx}.
\end{defi}

The notions of KKT solution and $\epsilon$-KKT solution were initially introduced in \cite[Section 3]{lu2024first-bilevel}. Notably, it was demonstrated in \cite[Theorem 2]{lu2024first-bilevel} that under suitable assumptions, an $\epsilon$-KKT solution $(x,y)$ of problem \eqref{prob}, with conditions such as $\tg=0$, $f_2=0$, $\tf_2=0$, $\tf_1$ being twice differentiable, and $\tf_1(x,\cdot)$ being strongly convex, implies that $x$ is an $\cO(\epsilon)$-hypergradient-based stationary point of \eqref{prob}.

The notions of KKT solution and $\epsilon$-KKT solution were initially introduced in \cite[Section 3]{lu2024first-bilevel}. Notably, it was demonstrated in \cite[Theorem 2]{lu2024first-bilevel} that under suitable assumptions, an $\epsilon$-KKT solution $(x,y)$ of problem \eqref{prob}, with conditions such as $\tg=0$, $f_2=0$, $\tf_2=0$, $\tf_1$ being twice differentiable, and $\tf_1(x,\cdot)$ being strongly convex, implies that $x$ is an $\cO(\epsilon)$-hypergradient-based stationary point of \eqref{prob}.

We next study iteration and operation complexity for Algorithm \ref{AL-alg}. To proceed, recall that $\mcX=\dom\;f_2$ and $\mcY=\dom\;\tf_2$.  We define
\begin{align}
&\tf^*_{\rm hi}:=\sup\{\tf^*(x)|x\in\mcX\},\label{def-tFx}\\
&D_\bfx\coloneqq \max\{\|u-v\|\big|u,v\in\mcX\},\quad D_\bfy\coloneqq\max\{\|u-v\|\big|u,v\in\mcY\},\label{DxDy}\\
&f_{\rm hi}:=\max\{f(x,y)|(x,y)\in\mcX\times\mcY\},\quad f_{\rm low}:=\min\{f(x,y)|(x,y)\in\mcX\times\mcY\},\label{fhi}\\
&\tf_{\rm low}:=\min\{\tf(x,z)|(x,z)\in\mcX\times\mcY\},\quad\tg_{\rm hi}:=\max\{\|\tg(x,y)\|\big|(x,y)\in\mcX\times\mcY\},\label{tfhi}\\
&K:=\left\lceil(\log\varepsilon-\log\epsilon_0)/\log\tau\right\rceil_+,\quad\bbK:=\{0,1,\dots, K+1\},\quad\bbK-1=\{k-1| k\in\bbK\}.\label{def-K}
\end{align}
It then follows from Assumption \ref{a1}(iii) that 
\beq\label{tghi}
\|\nabla \tg(x,y)\|\leq L_\tg \qquad \forall (x,y)\in\mcX\times\mcY.
\eeq
In addition, by Assumption \ref{a1} and the compactness of $\mcX$ and $\mcY$, one can observe that $D_\bfx$, $D_\bfy$, $f_{\rm hi}$, $f_{\rm low}$, $\tf_{\rm low}$ and $\tg_{\rm hi}$ are finite. 
Moreover,  $\tf^*_{\rm hi}$ is also finite (see Lemma \ref{dual-bnd}(ii) in Section \ref{sec:proof}).

The following assumption will be used to establish complexity of Algorithm \ref{AL-alg}. 

\begin{assumption}[{\bf Slater's condition}]\label{a2}
There exists $\hat z_x\in\mcY$ for each $x\in\mcX$ such that $\tg_i(x,\hat z_x)<0$ for all $i=1,2,\dots, l$ and $G:=\inf\{-\tg_i(x,\hat z_x)|x\in\mcX,\ i=1,\dots, l\}>0$.\footnote{If Assumption \ref{a2} fails to hold, one may instead consider the perturbed counterpart of \eqref{prob} with $\tg(x,z)$ replaced by $\tg(x,z)-\epsilon$ for some suitable $\epsilon>0$, which clearly satisfies Assumption \ref{a2}.}
\end{assumption}

We are now ready to present an \emph{iteration and operation complexity} of Algorithm~\ref{AL-alg}, measured by the amount of evaluations of $\nabla f_1$, $\nabla\tf_1$, $\nabla\tg$ and proximal operators of $f_2$ and $\tf_2$, for finding an $\cO(\varepsilon)$-KKT solution of \eqref{prob}, whose proofs are deferred to Section \ref{sec:proof}.

\begin{thm}[{\bf iteration and operation complexity of Algorithm \ref{AL-alg} for problem  \eqref{prob} with $\sigma=0$}]\label{complexity}
Suppose that Assumptions \ref{a1} and \ref{a2} hold with $\sigma=0$, i.e., $\tf_1(x,\cdot)$ being convex but not strongly convex for any given $x\in\dom\,f_2$. Let $\{(x^k,y^k,z^k,\lambda^k)\}_{k\in\bbK}$ be generated by Algorithm \ref{AL-alg}, $f^*$, $\tf^*_{\rm hi}$, $D_\bfx$, $D_\bfy$, $f_{\rm hi}$, $f_{\rm low}$, $\tf_{\rm low}$, $\tg_{\rm hi}$ and $K$ be defined in \eqref{prob}, \eqref{def-tFx}, \eqref{DxDy}, \eqref{fhi}, \eqref{tfhi} and \eqref{def-K}, $L_f$, $L_\tf$, $L_{\nabla f_1}$, $L_{\nabla\tf_1}$, $L_\tg$, $L_{\nabla\tg}$ and $G$ be given in Assumptions \ref{a1} and \ref{a2}, and $\varepsilon$, $\epsilon_0$, $\tau$, $\mu_K$, $\rho_K$ and $\lambda_0$ be given in Algorithm \ref{AL-alg}. Let
\begin{align}
&\vartheta=\frac{1}{2}\|\lambda^0\|^2+\frac{\tf^*_{\rm hi}-\tf_{\rm low}}{1-\tau^4}+\frac{D_\bfy\epsilon_0}{1-\tau^3},\label{ht}\\
&L= L_{\nabla f_1}+2L_{\nabla\tf_1}+2L_\tg^2+2\tg_{\rm hi}L_{\nabla \tg}+2\sqrt{2\vartheta}L_{\nabla\tg},\quad\tL=L_{\nabla\tf_1}+L_\tg^2+\tg_{\rm hi}L_{\nabla \tg}+\sqrt{2\vartheta}L_{\nabla\tg},\label{hL}\\
&\alpha=\min\Big\{1, \sqrt{4/(D_\bfy L)}\Big\},\quad \delta= (2+\alpha^{-1})L (D_\bfx^2+D_\bfy^2)+\max\{1/D_\bfy,L/4\}D_\bfy^2,\label{ho}\\
&M=16\max\left\{1/(4L_\tg^2),2/(\alpha L_\tg^2)\right\}\left[(3L+1/(2D_\bfy))^2/\min\{2L_\tg^2,1/(2D_\bfy)\}+ 3L+1/(2D_\bfy)\right]^2\nn\\
&\qquad\, \times\Big(\delta+2\alpha^{-1}\big(f^*-f_{\rm low}+\tf^*_{\rm hi}-\tf_{\rm low}+L_\tf D_\bfy+3\vartheta+\tg_{\rm hi}^2+D_\bfy/4+L (D_\bfx^2+D_\bfy^2)\big)\Big),\label{hM}\\
&T= \big\lceil16\left(f_{\rm hi}-f_{\rm low}+1+D_\bfy/4\right) L+8(1+4D_\bfy^2L^2)\big\rceil_+, \label{hT}\\
&\lambda_\bfy^{K+1}=[\lambda^K+\mu_K\tg(x^{K+1},y^{K+1})]_+,\quad\lambda_\bfz^{K+1}=\rho_K^{-1}[\lambda^K+\mu_K\tg(x^{K+1},z^{K+1})]_+.\label{lam}\end{align}
Suppose that $\varepsilon^{-2}-8\tau^{-3}G^{-2}\vartheta\geq0$. 
Then the following statements hold.
\begin{enumerate}[label=(\roman*)]
\item Algorithm \ref{AL-alg} terminates after $K+1$ outer iterations and outputs an approximate point $(x^{K+1},y^{K+1})$ of \eqref{prob} satisfying
{\small
\begin{align}
&\dist\Big(0,\partial f(x^{K+1},y^{K+1})+\rho_K\partial\tf(x^{K+1},y^{K+1}) - \rho_K\big(\nabla_x\tf(x^{K+1},z^{K+1})+\nabla_x\tg(x^{K+1},z^{K+1})\lambda_\bfz^{K+1}; 0\big)\nn\\
&\qquad+\nabla \tg(x^{K+1},y^{K+1})\lambda_\bfy^{K+1}\Big)\leq\varepsilon,\label{kkt1}\\
& \dist\Big(0,\rho_K\big(\partial_z\tf(x^{K+1},z^{K+1})+\nabla_z\tg(x^{K+1},z^{K+1})\lambda_\bfz^{K+1}\big)\Big)\leq\varepsilon,\label{kkt2} \\
& \|[\tg(x^{K+1},z^{K+1})]_+\|\leq2\varepsilon^2 G^{-1}(\epsilon_0+L_\tf)D_\bfy,\label{kkt3}\\ 
&|\langle\lambda^{K+1}_\bfz,\tg(x^{K+1},z^{K+1})\rangle|\leq2\varepsilon^2 G^{-1}(\epsilon_0+L_\tf)D_\bfy \max\{\|\lambda^0\|,\ 2G^{-1}(\epsilon_0+L_\tf)D_\bfy\}, \label{kkt4}\\
& \|[\tg(x^{K+1},y^{K+1})]_+\|\leq2\varepsilon^2 G^{-1}(\epsilon_0+L_f+L_\tf)D_\bfy  , \label{kkt5}\\ 
&|\langle\lambda^{K+1}_\bfy,\tg(x^{K+1},z^{K+1})\rangle|\leq2\varepsilon G^{-1}(\epsilon_0+L_f+L_\tf)D_\bfy \max\{\|\lambda^0\|,\ 2G^{-1}(\epsilon_0+L_f+L_\tf)D_\bfy\},\label{kkt6}\\
&|\tf(x^{K+1},y^{K+1})-\tf^*(x^{K+1})|\leq\max\Big\{2\varepsilon^2 G^{-2}L_\tf(\epsilon_0+L_f+L_\tf)D_\bfy^2,\ \varepsilon^3\max\{\|\lambda^0\|,\ 2G^{-1}(\epsilon_0+L_\tf)D_\bfy\}/2\nn\\
&\qquad\qquad\qquad\qquad\qquad\qquad\qquad\qquad+\varepsilon\left(f_{\rm hi}-f_{\rm low}+1+D_\bfy/4+L_\tg^{-2}/4+2D_\bfy^2L\right)\Big\}.\label{kkt7}
\end{align}
}
\item The total number of evaluations of $\nabla f_1$, $\nabla\tf_1$, $\nabla \tg$ and proximal operators of $f_2$ and $\tf_2$ performed in Algorithm \ref{AL-alg} is no more than $N$, respectively, where
\begin{align}
N=&\left(\left\lceil96\sqrt{2}\left(1+\left(12L+2/D_y\right)/L_\tg^2\right)\right\rceil+2\right)\max\left\{2,\sqrt{D_yL}\right\}T(1-\tau^7)^{-1}\nn\\
&\ \ \ \times(\tau\varepsilon)^{-7}\left(56K\log(1/\tau)+56\log(1/\epsilon_0)+2(\log M)_++2+2\log(2T) \right)\nn\\
&\ \ \ +(\tau\varepsilon)^{-3/2}(1-\tau^{3/2})^{-1}D_\bfy\sqrt{2\tL}+K.\label{def-N}
\end{align}
\end{enumerate}
\end{thm}

\begin{rem}
One can observe from Theorem \ref{complexity} that Algorithm \ref{AL-alg} enjoys an iteration complexity of $\cO(\log\varepsilon^{-1})$ and an operation complexity of $\cO(\varepsilon^{-7}\log\varepsilon^{-1})$, measured by the amount of evaluations of $\nabla f_1$, $\nabla \tf_1$, $\nabla \tg$ and proximal operators of $f_2$ and $\tf_2$, for finding an $\cO(\varepsilon)$-KKT solution $(x^{K+1},y^{K+1})$ of \eqref{prob} satisfying 
\begin{align}
&\dist\Big(0,\partial f(x^{K+1},y^{K+1})+\rho_K\partial\tf(x^{K+1},y^{K+1}) +\nabla \tg(x^{K+1},y^{K+1})\lambda_\bfy^{K+1} \nonumber \\
&\qquad - \rho_K\big(\nabla_x\tf(x^{K+1},z^{K+1})+\nabla_x\tg(x^{K+1},z^{K+1})\tilde\lambda_\bfz^{K+1}; 0\big)\Big)\leq\varepsilon,  \label{cond1}\\
& \dist\Big(0,\rho_K\big(\partial_z\tf(x^{K+1},z^{K+1})+\nabla_z\tg(x^{K+1},z^{K+1})\lambda_\bfz^{K+1}\big)\Big)\leq\varepsilon,  \label{cond2} \\
& \|[\tg(x^{K+1},z^{K+1})]_+\|=\cO(\varepsilon^2),\quad|\langle\lambda^{K+1}_\bfz,\tg(x^{K+1},z^{K+1})\rangle|=\cO(\varepsilon^2),  \label{cond3}\\
& \|[\tg(x^{K+1},y^{K+1})]_+\|=\cO(\varepsilon^2),\quad |\langle\lambda^{K+1}_\bfy,\tg(x^{K+1},z^{K+1})\rangle|=\cO(\varepsilon), \label{cond4} \\
&|\tf(x^{K+1},y^{K+1})-\tf^*(x^{K+1})|=\cO(\varepsilon),  \label{cond5}
\end{align}
where $\tf^*$ is defined in \eqref{tfstarx}, $\rho_K=(\epsilon_0\tau^K)^{-1}$, and $\lambda_\bfy^{K+1},\lambda_\bfz^{K+1}\in\bR_+^l$ are given in \eqref{lam}.
\end{rem}

\begin{thm}[{\bf iteration and operation complexity of Algorithm \ref{AL-alg} for problem  \eqref{prob}  with $\sigma>0$}]\label{complexity-str}
Suppose that Assumptions \ref{a1} and \ref{a2} hold with $\sigma>0$, i.e., $\tf_1(x,\cdot)$ being strongly convex with parameter $\sigma$ for any given $x\in\dom\,f_2$. Let $\{(x^k,y^k,z^k,\lambda^k)\}_{k\in\bbK}$ be generated by Algorithm \ref{AL-alg}, $f^*$, $\tf^*_{\rm hi}$, $D_\bfx$, $D_\bfy$, $\tf_{\rm low}$, $f_{\rm low}$, $f_{\rm hi}$, $\tg_{\rm hi}$, $K$, $\vartheta$, $L$, $\tL$, $\lambda_\bfy^{K+1}$ and $\lambda_\bfz^{K+1}$ be defined in \eqref{prob}, \eqref{def-tFx}, \eqref{DxDy}, \eqref{fhi}, \eqref{tfhi}, \eqref{def-K}, \eqref{ht}, \eqref{hL} and \eqref{lam}, $\sigma$, $L_f$, $L_\tf$, $L_{\nabla f_1}$, $L_{\nabla\tf_1}$, $L_\tg$, $L_{\nabla\tg}$ and $G$ be given in Assumptions \ref{a1} and \ref{a2}, and $\varepsilon$, $\epsilon_0$, $\tau$, $\mu_K$, $\rho_K$ and $\lambda_0$ be given in Algorithm \ref{AL-alg}. Let
\begin{align}
&\tilde\alpha=\min\left\{1, \sqrt{8\sigma/L}\right\},\quad \tilde\delta= (2+\tilde\alpha^{-1})(D_\bfx^2+D_\bfy^2)L+\max\{1/D_\bfy,L/4\}D_\bfy^2,\label{ho-str}\\
&\widetilde M=16\max\left\{1/(4L_\tg^2),2/(\tilde\alpha L_\tg^2)\right\}\left[9L^2/\min\{2L_\tg^2,\sigma\}+ 3L\right]^2\nn\\
&\qquad\, \times\Big(\tilde\delta+2\tilde\alpha^{-1}\big(f^*-f_{\rm low}+\tf^*_{\rm hi}-\tf_{\rm low}+L_\tf D_\bfy+3\vartheta+\tg_{\rm hi}^2+L(D_\bfx^2+D_\bfy^2)\big)\Big),\label{hM-str}\\
&\widetilde T= \big\lceil16\left(f_{\rm hi}-f_{\rm low}+1\right) L+8(1+\sigma^{-2}L^2)\big\rceil_+. \label{hT-str}
\end{align}
Suppose that $\varepsilon^{-2}-8\tau^{-3}G^{-2}\vartheta\geq0$. 
Then the following statements hold.
\begin{enumerate}[label=(\roman*)]
\item Algorithm \ref{AL-alg} terminates after $K+1$ outer iterations and outputs an approximate point $(x^{K+1},y^{K+1})$ of problem \eqref{prob} satisfying \eqref{kkt1}-\eqref{kkt6} and
 {\small
\begin{align}
|\tf(x^{K+1},y^{K+1})-\tf^*(x^{K+1})|\leq\max\Big\{&2\varepsilon^2 G^{-2}L_\tf(\epsilon_0+L_f+L_\tf)D_\bfy^2,\ \varepsilon^3\max\{\|\lambda^0\|,\ 2G^{-1}(\epsilon_0+L_\tf)D_\bfy\}/2\nn\\
&+\varepsilon\left(f_{\rm hi}-f_{\rm low}+1+L_\tg^{-2}/4+\sigma^{-2}L/2\right)\Big\}.\label{kkt7-str}\end{align}
}
\item The total number of evaluations of $\nabla f_1$, $\nabla\tf_1$, $\nabla \tg$ and proximal operators of $f_2$ and $\tf_2$ performed in Algorithm \ref{AL-alg} is no more than $\widetilde N$, respectively, where
\begin{align}
\widetilde N=&3397\max\left\{2,\sqrt{L/(2\sigma)}\right\}\widetilde T(1-\tau^6)^{-1}\nn\\
&\ \times(\tau\varepsilon)^{-6}\left(38K\log(1/\tau)+38\log(1/\epsilon_0)+2(\log \widetilde M)_++2+2\log(2\widetilde T) \right)\nn\\
&\ +2(\tau\varepsilon)^{-1}(1-\tau)\left\lceil\sqrt{\tL/\sigma}+1\right\rceil\max\left\{1,\left\lceil2\log(2\tL D_\bfy^2)+6K\log(1/\tau)-6\log\epsilon_0\right\rceil\right\}+K.\label{def-N-str}
\end{align}
\end{enumerate}
\end{thm}

\begin{rem}
One can observe from Theorem \ref{complexity-str} that Algorithm \ref{AL-alg} enjoys an iteration complexity of $\cO(\log\varepsilon^{-1})$ and an operation complexity of $\cO(\varepsilon^{-6}\log\varepsilon^{-1})$, measured by the amount of evaluations of $\nabla f_1$, $\nabla \tf_1$, $\nabla \tg$ and proximal operators of $f_2$ and $\tf_2$, for finding an $\cO(\varepsilon)$-KKT solution $(x^{K+1},y^{K+1})$ of \eqref{prob} satisfying \eqref{cond1}-\eqref{cond5}. Such an operation complexity significantly improves the previous best-known operation complexity $\cO(\varepsilon^{-7}\log\varepsilon^{-1})$ established in \cite[Theorem 5]{lu2024first-bilevel} by a factor of $\varepsilon^{-1}$.
\end{rem}

\section{Numerical results}\label{sec:exp}
In this section, we conduct some preliminary experiments to test the performance of our SMO method (Algorithm \ref{AL-alg}), and compare it with a first-order penalty (FOP) method  (\cite[Algorithm 4]{lu2024first-bilevel}). Both algorithms are coded in Matlab and all the computations are performed on a laptop with a 2.30 GHz Intel i9-9880H 8-core processor and 16 GB of RAM.

\subsection{Constrained bilevel linear optimization}
In this subsection, we consider constrained bilevel linear optimization in the form of
\begin{equation}\label{linear}
\begin{array}{rl}
\min & c^Tx+d^Ty+\cI_{[-1,1]^n}(x)\\ [4pt]
\mbox{s.t.}& y\in\argmin\limits_{z}\left\{\tilde d^Tz+\cI_{[-1,1]^m}(z)\big|\widetilde Ax+\widetilde Bz-\tilde b\leq0\right\}, 
\end{array}
\end{equation}
where $c\in\bR^n$, $d, \tilde d\in\bR^m$,  $\tilde b\in\bR^l$, $\widetilde A\in\bR^{l\times n}$, $\widetilde B\in\bR^{l\times m}$, and $\cI_{[-1,1]^n}(\cdot)$ and $\cI_{[-1,1]^m}(\cdot)$ are the indicator functions of $[-1,1]^n$ and $[-1,1]^m$ respectively.

For each triple $(n,m,l)$, we randomly generate $10$ instances of problem \eqref{linear}. Specifically,  we first randomly generate $c$ and $d$ with all the entries independently chosen from the standard normal distribution. We then randomly generate $\widetilde A$ and $\widetilde B$ with all the entries independently chosen from a normal distribution with mean $0$ and standard deviation $0.01$. In addition, we randomly generate $\hat y\in[-1,1]^m$ with all the entries independently chosen from a normal distribution with mean $0$ and standard deviation $0.1$ and then projected to $[-1,1]^m$ and choose $\tilde d$ and $\tilde b$ such that  $\hat y$ is an optimal solution of the lower-level optimization of \eqref{linear} with $x=0$.

Notice that \eqref{linear} is a special case of \eqref{prob} with $f(x,y)= c^Tx+d^Ty+\cI_{[-1,1]^n}(x)$, $\tf(x,z)=\tilde d^Tz+\cI_{[-1,1]^m}(z)$ and $\tg(x,z)=\widetilde Ax+\widetilde Bz-\tilde b$.   
We now apply SMO and FOP to solve \eqref{linear}.  In particular, we choose $0$ as the initial point for both methods. In addition, we set $(\varepsilon,\epsilon_0,\tau)=(10^{-2},1,0.8)$ for SMO.  To enhance the efficiency of FOP, we adopt a dynamic updating scheme on its penalty and tolerance parameters. Specifically, we set $\rho_k=5^{k-1}$, $\varepsilon_k=\rho_k^{-1}$ and 
$x_{-1}=0$ for \cite[Algorithm 2]{lu2024first-bilevel}. For each $k >1$, let $(x^{k-1}, y^{k-1})$ be the output of \cite[Algorithm 2]{lu2024first-bilevel} with $(\varepsilon, \rho)=(\varepsilon_{k-1}, \rho_{k-1})$.  We run \cite[Algorithm 2]{lu2024first-bilevel} with $(\varepsilon, \rho)=(\varepsilon_k, \rho_k)$ and $(x^{k-1},\tilde y^{k-1})$ as the initial point to generate $(x^k,y^k)$, where $\tilde y^{k-1}\in\argmin_z\tf(x^{k-1},z)$ is found by CVX \cite{grant2014cvx}. We terminate both algorithms once $\epsilon_{\bk} \leq 10^{-2}$ for SMO, $\varepsilon_{\bk}\leq10^{-2}$ for FOP, and 
$(x^{\bk},y^{\bk})$ satisfies
\begin{equation*}
\|[\tg(x^{\bk},y^{\bk})]_+\|\leq10^{-2},\quad \tf(x^{\bk},y^{\bk})-\tf^*(x^{\bk})\leq10^{-2}
\end{equation*}
for some $\bk$, and output $(x^{\bk},y^{\bk})$ as an approximate solution of\eqref{linear}, where $\tf^*$ is defined in \eqref{tfstarx} and the value $\tf^*(x^{\bk})$ is computed by CVX \cite{grant2014cvx}.

The computational results of SMO and FOP for problem \eqref{linear} with the instances randomly generated above are presented in Table \ref{t-linear}. In detail, the values of $n$, $m$ and $l$ are listed in the first three columns. For each triple $(n, m, l)$, the average initial objective value $f(x^0,\hat y)$ with $\hat y$ being generated above,\footnote{Note that $(x^0,y_{\rm init}^0)$ may not be a feasible point of \eqref{linear}. Nevertheless, $(x^0,\hat y)$ is a feasible point of \eqref{linear} due to $x^0=0$ and the particular way for generating instances of \eqref{linear}. Besides, \eqref{linear} can be viewed as an implicit optimization problem in terms of the variable $x$. It is thus reasonable to use $f(x^0,\hat y)$ as the initial objective value for the purpose of comparison.} the average final objective value $f(x^{\bk},y^{\bk})$ and the average CPU time (in seconds) over $10$ random instances are given in the rest of the columns. One can observe that both SMO and FOP find an approximate solution with much lower  objective value than the initial objective value. Moreover, SMO outputs an approximate solution with a similar objective value as FOP, while SMO significantly outperforms FOP in terms of average CPU time.

\begin{table}[H]
\centering
\resizebox{0.95\linewidth}{!}{
\begin{tabular}{ccc||c||cc||cc}
\hline
 & & &Initial objective value&\multicolumn{2}{c||}{Final objective value}&\multicolumn{2}{c}{CPU time (seconds)}\\
$n$ & $m$&$l$ & &SMO&FOP&SMO&FOP\\\hline
$100$&$100$&$5$&$0.22$&$-75.78$&$-75.75$&$7.4$&$22.1$\\
$200$&$200$&$10$&$-0.38$&$-154.20$&$-154.18$&$12.5$&$107.9$\\
$300$&$300$&$15$&$-0.11$&$-246.70$&$-246.66$&$24.1$&$267.8$\\
$400$&$400$&$20$&$0.34$&$-305.97$&$-305.89$&$37.0$&$561.6$\\
$500$&$500$&$25$&$0.96$&$-394.74$&$-394.70$&$63.8$&$719.8$\\ \hline
\end{tabular}
}
\caption{Numerical results for problem \eqref{linear}}
\label{t-linear}
\end{table}

\subsection{Constrained bilevel optimization with quadratic upper level and linear lower level}
In this subsection, we consider constrained bilevel optimization with quadratic upper level and linear lower level in the form of
\begin{equation}\label{Qlinear}
\begin{array}{rl}
\min & x^TAx+x^TBy+y^TCy+c^Tx+d^Ty+\cI_{[-1,1]^n}(x)\\ [4pt]
\mbox{s.t.}& y\in\argmin\limits_{z}\left\{\tilde d^Tz+\cI_{[-1,1]^m}(z)\big|\widetilde Ax+\widetilde Bz-\tilde b\leq0\right\}, 
\end{array}
\end{equation}
where $A\in\bR^{n\times n}$, $B\in\bR^{n\times m}$, $C\in\bR^{m\times m}$, $c\in\bR^n$, $d, \tilde d\in\bR^m$,  $\tilde b\in\bR^l$, $\widetilde A\in\bR^{l\times n}$, $\widetilde B\in\bR^{l\times m}$, and $\cI_{[-1,1]^n}(\cdot)$ and $\cI_{[-1,1]^m}(\cdot)$ are the indicator functions of $[-1,1]^n$ and $[-1,1]^m$ respectively.

For each triple $(n,m,l)$, we randomly generate $10$ instances of problem \eqref{Qlinear}. Specifically,  we first randomly generate $A$, $B$, $C$, $c$ and $d$ with all the entries independently chosen from a normal distribution with mean $0$ and standard deviation $0.1$. We then randomly generate $\widetilde A$ and $\widetilde B$ with all the entries independently chosen from a normal distribution with mean $0$ and standard deviation $0.01$. In addition, we randomly generate $\hat y\in[-1,1]^m$ with all the entries independently chosen from a normal distribution with mean $0$ and standard deviation $0.1$ and then projected to $[-1,1]^m$ and choose $\tilde d$ and $\tilde b$ such that  $\hat y$ is an optimal solution of the lower-level optimization of \eqref{Qlinear} with $x=0$.

Notice that \eqref{Qlinear} is a special case of \eqref{prob} with $f(x,y)= x^TAx+x^TBy+y^TCy+c^Tx+d^Ty+\cI_{[-1,1]^n}(x)$, $\tf(x,z)=\tilde d^Tz+\cI_{[-1,1]^m}(z)$ and $\tg(x,z)=\widetilde Ax+\widetilde Bz-\tilde b$.   
We now apply our SMO method (Algorithm \ref{AL-alg}) to solve \eqref{Qlinear}.\footnote{Clearly, problem \eqref{Qlinear} is more sophisticated than \eqref{linear}. As shown in Table \ref{t-linear}, problem \eqref{linear} already poses significant challenges for FOP \cite{lu2024first-bilevel} when the dimension $n$ is relatively large. Therefore, we do not apply FOP to solve \eqref{Qlinear}.} Specifically, we choose $0$ as the initial point and set the parameters of SMO as $(\varepsilon,\epsilon_0,\tau)=(10^{-2},1,0.8)$. We terminate SMO once $\epsilon_{\bk} \leq 10^{-2}$ and 
$(x^{\bk},y^{\bk})$ satisfies
\begin{equation*}
\|[\tg(x^{\bk},y^{\bk})]_+\|\leq10^{-2},\quad \tf(x^{\bk},y^{\bk})-\tf^*(x^{\bk})\leq10^{-2}
\end{equation*}
for some $\bk$ and output $(x^{\bk},y^{\bk})$ as an approximate solution of \eqref{Qlinear}, where $\tf^*$ is defined in \eqref{tfstarx} and the value $\tf^*(x^{\bk})$ is computed by CVX \cite{grant2014cvx}.

The computational results of SMO for problem \eqref{Qlinear} with the instances randomly generated above are presented in Table \ref{t-Qlinear}. In detail, the values of $n$, $m$ and $l$ are listed in the first three columns. For each triple $(n, m, l)$, the average initial objective value $f(x^0,\hat y)$ with $\hat y$ being generated above and the average final objective value $f(x^{\bk},y^{\bk})$ over $10$ random instances are given in the rest of the columns. One can see that the approximate solution $(x^{\bk},y^{\bk})$ found by SMO significantly reduces the initial objective value.

\begin{table}[H]
\centering
\begin{tabular}{ccc||cc}
\hline
$n$ & $m$&$l$ &Initial objective value&Final objective value\\\hline
$100$&$100$&$5$&$-0.04$&$-95.70$\\
$200$&$200$&$10$&$0.03$&$-275.34$\\
$300$&$300$&$15$&$0.15$&$-487.64$\\
$400$&$400$&$20$&$0.20$&$-749.02$\\
$500$&$500$&$25$&$0.13$&$-1085.57$\\\hline
\end{tabular}
\caption{Numerical results for problem \eqref{Qlinear}}
\label{t-Qlinear}
\end{table}

\subsection{Hyperparameter tuning for support vector machine}

In this subsection, we consider a hyperparameter tuning model for support vector machine (SVM):
\begin{align}
\min_{c,w,b,\xi}\,\, & \frac{1}{m}\sum_{n<i\leq n+m}\ell\left(\hat y_i,w^T\hat x_i+b\right)+\cI_{[0,10]^n}(c)\label{prob-SVM}\\
\mbox{s.t.}\,\, & (w,b,\xi)\in\argmin_{\tilde w, \tilde b, \tilde\xi}\Big\{\sum_{1\leq i\leq n}\ell\left(\hat y_i,\tilde w^T\hat x_i+\tilde b\right)+c^T\tilde \xi+\cI_{[-1,1]^{q+1}}(\tilde w,\tilde b)+\cI_{[0,20]^n}(\tilde\xi) \Big|\nn \\
&\qquad\qquad\qquad\qquad\quad\,\,\hat y_i(\tilde w^T\hat x_i+\tilde b)\geq1-\tilde\xi_i,\;i=1,\dots,n\Big\},\nn
\end{align}
where $\{(\hat x_i,\hat y_i)\}_{1\leq i\leq n}$ is the training set, $\{(\hat x_i,\hat y_i)\}_{n< i\leq n+m}$ is the validation set, $c\in\bR^n$, $w\in\bR^q$, $b\in\bR$,  $\xi\in\bR^n$, $\ell(u,v)=\log(1+e^{-uv})$ is the binomial deviance loss function \cite{hastie2009elements}, and $\cI_{[0,10]^{n}}$, $\cI_{[-1,1]^{q+1}}$, and $\cI_{[0,20]^{n}}$ are the indicator functions of $[0,10]^n$, $[-1,1]^{q+1}$ and $[0,20]^{n}$, respectively. Specifically, at the lower level of \eqref{prob-SVM},  we train a linear SVM on the training set with a decision hyperplane of the form $\{x\in\bR^q:w^Tx+b=0\}$, where the binomial deviance is used as the loss function, and a slack variable $\xi$ and penalty parameter $c$ are introduced to handle non-separable datasets.\footnote{The vector $c$ is introduced to assign weights to individual data points in the weighted SVM formulation. This technique has been studied in the literature to capture the relative importance of data points in the training set (see \cite{shen2013cu, xanthopoulos2014weighted, yang2007weighted}).}
At the upper level of \eqref{prob-SVM}, we minimize the validation loss to select the hyperparameter $c$ and the corresponding $(w,b,\xi)$. Similar bilevel SVM models have been widely studied in the literature (e.g., \cite{bennett2008bilevel,okuno2021lp,ye2023difference}).

In our experiments, we solve problem \eqref{prob-SVM} on five datasets from the LIBSVM repository \cite{chang2011libsvm}. For each dataset, one-fourth of the samples are randomly selected as the validation set, and the remainder are used as the training set. Notice that \eqref{prob-SVM} is a special case of \eqref{prob} with $x=c$, $y=(w,b,\xi)$, $z=(\tilde w,\tilde b,\tilde\xi)$,
\begin{align*}
&f(x,y)=\frac{1}{m}\sum_{n<i\leq n+m}\ell\left(\hat y_i,w^T\hat x_i+b\right)+\cI_{[0,10]^n}(c),\\
&\tf(x,z)=\sum_{1\leq i \leq n}\ell\left(\hat y_i,w^T\hat x_i+b\right)+c^T\tilde \xi+\cI_{[-1,1]^{q+1}}(\tilde w,\tilde b)+\cI_{[0,20]^n}(\tilde\xi),\\
& \tg_i(x,z)=1-\tilde\xi_i-\hat y_i(\tilde w^T\hat x_i+\tilde b), i=1\dots,n.
\end{align*}
As a result, both SMO (Algorithm \ref{AL-alg}) and FOP (\cite[Algorithm 2]{lu2024first-bilevel}) are suitable for solving \eqref{prob-SVM}. We now apply both methods to solve \eqref{prob-SVM}, starting  from $x^0=0$ and $y^0$ with the entries independently drawn from the uniform distribution on $[0,1]$. In addition, we set $(\varepsilon,\epsilon_0,\tau)=(10^{-2},1,0.9)$ for SMO.  To enhance the efficiency of FOP, we adopt a dynamic updating scheme on its penalty and tolerance parameters. Specifically, we set $\rho_k=5^{k-1}$, $\varepsilon_k=\rho_k^{-1}$ for \cite[Algorithm 2]{lu2024first-bilevel}. For each $k >1$, let $(x^{k-1}, y^{k-1})$ be the output of \cite[Algorithm 2]{lu2024first-bilevel} with $(\varepsilon, \rho)=(\varepsilon_{k-1}, \rho_{k-1})$.  We run \cite[Algorithm 2]{lu2024first-bilevel} with $(\varepsilon, \rho)=(\varepsilon_k, \rho_k)$ and $(x^{k-1},\tilde y^{k-1})$ as the initial point to generate $(x^k,y^k)$, where $\tilde y^{k-1}\in\argmin_z\tf(x^{k-1},z)$ is found by CVX \cite{grant2014cvx}.
We terminate both algorithms once $\epsilon_{\bk} \leq 10^{-2}$ for SMO, $\varepsilon_{\bk}\leq10^{-2}$ for FOP, and 
$(x^{\bk},y^{\bk})$ satisfies
\begin{equation*}
\|[\tg(x^{\bk},y^{\bk})]_+\|\leq10^{-2},\quad \tf(x^{\bk},y^{\bk})-\tf^*(x^{\bk})\leq10^{-2}
\end{equation*}
for some $\bk$, and output $(x^{\bk},y^{\bk})$ as an approximate solution of \eqref{prob-SVM}, where $\tf^*$ is defined in \eqref{tfstarx} and the value $\tf^*(x^{\bk})$ is computed by CVX \cite{grant2014cvx}.

The computational results of SMO and FOP for problem \eqref{prob-SVM} are presented in Table \ref{t-SVM}. In detail, the names of the datasets are listed in the first column. For each dataset, the initial objective value $f(x^0,y_0^*)$, where $y_0^*$ is the lower-level optimal solution with $x=x^0$ computed by CVX \cite{grant2014cvx}, the final objective value $f(x^{\bk},y^{\bk})$, validation accuracy, and the CPU time (in seconds) are given in the rest of the columns. It can be observed that both SMO and FOP find approximate solutions with objective values substantially lower than the initial value, and the resulting support vector machine achieves good validation accuracy. Moreover, SMO produces an approximate solution with an objective value similar to that of FOP but significantly outperforms FOP in terms of CPU time.

\begin{table}[H]
\centering
\resizebox{0.95\linewidth}{!}{
\begin{tabular}{c||c||cc||cc||cc}
\hline
Dataset &Initial objective value&\multicolumn{2}{c||}{Final objective value}&\multicolumn{2}{c||}{Validation accuracy}&\multicolumn{2}{c}{CPU time (seconds)}\\
 & &SMO&FOP&SMO&FOP &SMO &FOP\\\hline
breast-cancer\_scale&$2.894$&$0.336$&$0.352$&$100.0\%$& $98.5\%$ &$175.8$&$1537.9$\\
heart\_scale&$1.359$&$0.556$&$0.551$&$75.9\%$ &$72.4\%$ &$402.6$&$3335.5$\\
ionosphere\_scale&$1.063$&$0.443$&$0.447$& $88.9\%$& $86.7\%$&$792.5$&$7520.7$\\
german.numer\_scale&$1.609$&$0.544$&$0.562$&$87.7\%$ &$84.2\%$ &$278.6$&$2453.0$\\
australian\_scale&$2.565$&$0.654$&$0.678$&$72.5\%$ &$72.5\%$ &$224.6$&$2138.0$\\ \hline
\end{tabular}
}
\caption{Numerical results for problem \eqref{prob-SVM}}
\label{t-SVM}
\end{table}

\section{Proof of main results }\label{sec:proof}
In this section,  we provide a proof of our main results presented in Subsection \ref{sec:thm}, which are particularly Theorems \ref{complexity} and \ref{complexity-str}. 

It should be noted that while the first-order penalty method (FOP) (\cite[Algorithm 4]{lu2024first-bilevel}) and the SMO method (Algorithm \ref{AL-alg}) both require solving minimax subproblems, they differ substantially in several aspects: (i) FOP solves a single minimax subproblem with fixed penalty parameters and without a warm-start strategy, whereas SMO solves a sequence of minimax subproblems with dynamically updated penalty parameters and a tailored warm-start strategy; (ii) The minimax subproblem in FOP arises from a quadratic penalty scheme applied to the lower-level problem, while in SMO it is derived from a modified augmented Lagrangian scheme; (iii) FOP is designed for problems with a merely convex lower-level objective and cannot exploit strong convexity,  whereas SMO can leverage it to achieve stronger complexity results. As a result, the convergence proofs for FOP and SMO are fundamentally different. Moreover, due to SMO’s more intricate structure, establishing its convergence is significantly more challenging.

To proceed, one can observe from \eqref{tLag} and \eqref{tfstarx} that
\beq\label{p-ineq}
\min_z\tcL(x,z,\lambda;\rho,\mu) \leq \tf^*(x) +\frac{\|\lambda\|^2}{2\rho\mu}\qquad\forall x\in\mcX, \lambda\in\bR_+^l, \rho,  \mu>0,
\eeq
which will be frequently used later. 

We now introduce several technical lemmas that will be utilized to prove Theorems \ref{complexity} and \ref{complexity-str} subsequently. The following lemma presents several properties of the lower-level problem of \eqref{prob}, whose proof can be found in \cite{lu2024first-bilevel}.

\begin{lemma}[{\bf \cite[Lemma 3]{lu2024first-bilevel}}] \label{dual-bnd}
Suppose that Assumptions~\ref{a1} and \ref{a2} hold. Let $\tf^*$, $\tf^*_{\rm hi}$, $D_\bfy$,  $L_\tf$ and $G$ be given in \eqref{tfstarx}, \eqref{def-tFx}, \eqref{DxDy}, and Assumptions~\ref{a1} and \ref{a2}, respectively. Then the following statements hold.
\begin{enumerate}[label=(\roman*)]
\item $\lambda^*\geq0$ and $\|\lambda^*\|\leq G^{-1}L_{\tf} D_\bfy$ for all $\lambda^*\in\Lambda^*(x)$ and $x\in\mcX$, where $\Lambda^*(x)$ denotes the set of optimal Lagrangian multipliers of problem 
\eqref{tfstarx} for any $x\in\mcX$.
\item The function $\tf^*$ is Lipschitz continuous on $\mcX$ and $\tf^*_{\rm hi}$ is finite.
\item It holds that
\[
\tf^*(x)=\max_{\lambda}\min_{z}\tf(x,z)+\langle\lambda,\tg(x,z)\rangle-\cI_{\bR_+^l}(\lambda) \qquad \forall x\in\mcX,
\]
where $\cI_{\bR_+^l}(\cdot)$ is the indicator function associated with $\bR_+^l$.
\end{enumerate}
\end{lemma}

The next lemma provides an upper bound on $\|\lambda^k\|$ for all $0\leq k\in\bbK-1$.

\begin{lemma}\label{l-ycnstr}
Suppose that Assumption \ref{a1} holds. Let $\bbK$ and $\vartheta$ be defined in  \eqref{def-K} and \eqref{ht}, $\mu_k$ and $\rho_k$ be given in Algorithm \ref{AL-alg}, and $\{\lambda^k\}_{k\in\bbK}$ be generated by Algorithm \ref{AL-alg}. Then we have
\beq\label{ly-cnstr}
\|\lambda^k\|^2\leq 2\rho_k\mu_k\vartheta \qquad \forall 0\leq k\in\bbK-1. 
\eeq
\end{lemma}

\begin{proof}
One can observe from \eqref{def-tFx}, \eqref{tfhi} and Algorithm \ref{AL-alg} that $\tf^*_{\rm hi} \geq \tf_{\rm low}$ and $\mu_0\geq\rho_0\geq1>\tau>0$, which together with \eqref{ht} imply that \eqref{ly-cnstr} holds for $k=0$. It remains to show that \eqref{ly-cnstr} holds for all $1\leq k\in\bbK-1$.

Since $(x^{t+1},y^{t+1},z^{t+1})$ is an $\epsilon_t$-primal-dual stationary point of \eqref{AL-sub} for all $0\leq t\in\bbK-1$, it follows from Definition \ref{def1} that there exists some $u\in\partial_z\mcL(x^{t+1},y^{t+1},z^{t+1},\lambda^t;\rho_t,\mu_t)$ with $\|u\|\leq\epsilon_t$. Notice from \eqref{Lag} and \eqref{tLag} that $\partial_z\mcL(x^{t+1},y^{t+1},z^{t+1},\lambda^t;\rho_t,\mu_t)=-\rho_t\partial_z\tcL(x^{t+1},z^{t+1},\lambda^t;\rho_t,\mu_t)$. Hence, $-\rho_t^{-1}u\in\partial_z\tcL(x^{t+1},z^{t+1},\lambda^t;\rho_t,\mu_t)$. Also, observe from \eqref{tLag} and Assumption \ref{a1} that $\tcL(x^{t+1},\cdot,\lambda^t;\rho_t,\mu_t)$ is convex. Using this, \eqref{DxDy}, $-\rho_t^{-1}u\in\partial_z\tcL(x^{t+1},z^{t+1},\lambda^t;\rho_t,\mu_t)$ and $\|u\|\leq\epsilon_t$, we obtain
\begin{align*}
\tcL(x^{t+1},z,\lambda^t;\rho_t,\mu_t)\geq&\ \tcL(x^{t+1},z^{t+1},\lambda^t;\rho_t,\mu_t)+\langle-\rho_t^{-1}u,z-z^{t+1}\rangle.\\
\geq&\ \tcL(x^{t+1},z^{t+1},\lambda^t;\rho_t,\mu_t)-\rho_t^{-1}D_\bfy\epsilon_t\qquad \forall z\in\mcY,
\end{align*}
which implies that
\beq\label{y-gap}
\min_z\tcL(x^{t+1},z,\lambda^t;\rho_t,\mu_t)\geq\tcL(x^{t+1},z^{t+1},\lambda^t;\rho_t,\mu_t)-\rho_t^{-1}D_\bfy\epsilon_t.
\eeq
By this, \eqref{tLag} and \eqref{p-ineq}, one has
\begin{align*}
\tf^*(x^{t+1}) & \overset{\eqref{p-ineq}}{\geq}\min_z\tcL(x^{t+1},z,\lambda^t;\rho_t,\mu_t)-\frac{\|\lambda^t\|^2}{2\rho_t\mu_t}\\
&\overset{\eqref{tLag}\eqref{y-gap}}{\geq} \tf(x^{t+1},z^{t+1})+\frac{1}{2\rho_t\mu_t}\left(\|[\lambda^t+\mu_t\tg(x^{t+1},z^{t+1})]_+\|^2-\|\lambda^t\|^2\right)-\rho_t^{-1}D_\bfy\epsilon_t\\
& = \tf(x^{t+1},z^{t+1})+\frac{1}{2\rho_t\mu_t}\left(\|\lambda^{t+1}\|^2-\|\lambda^t\|^2\right)-\rho_t^{-1}D_\bfy\epsilon_t,
\end{align*}
where the equality follows from the relation $\lambda^{t+1}=[\lambda^t+\mu_t \tg(x^{t+1},z^{t+1})]_+$ (see Algorithm \ref{AL-alg}). Using this inequality, \eqref{def-tFx},  \eqref{tfhi} and $\epsilon_t\leq\epsilon_0$ (see Algorithm \ref{AL-alg}), we have
\begin{equation*}
\|\lambda^{t+1}\|^2-\|\lambda^t\|^2\leq2\rho_t\mu_t(\tf^*(x^{t+1})-\tf(x^{t+1},y^{t+1}))+2\mu_tD_\bfy\epsilon_t\leq2\rho_t\mu_t(\tf^*_{\rm hi}-\tf_{\rm low})+2\mu_tD_\bfy\epsilon_0.
\end{equation*}
Summing up this inequality for $t=0,\dots,k-1$ with $1 \leq  k\in\bbK-1$ yields
\beq\label{l1-sum}
\|\lambda^{k}\|^2\leq \|\lambda^0\|^2+2(\tf^*_{\rm hi}-\tf_{\rm low})\sum_{t=0}^{k-1}\rho_t\mu_t+2D_\bfy\epsilon_0\sum_{t=0}^{k-1}\mu_t.
\eeq
Recall from Algorithm \ref{AL-alg} that $\epsilon_t=\epsilon_0\tau^t$, $\mu_t=\epsilon_t^{-3}$ and $\rho_t=\epsilon_t^{-1}$. It is not hard to verify that $\sum_{t=0}^{k-1}\rho_t\mu_t\leq\rho_{k-1}\mu_{k-1}/(1-\tau^4)$ and $\sum_{t=0}^{k-1}\mu_t\leq\mu_{k-1}/(1-\tau^3)$. Using these, \eqref{l1-sum}, $\rho_k>\rho_{k-1}\geq1$ and $\mu_{k}>\mu_{k-1}\geq1$ (see Algorithm \ref{AL-alg}), we obtain that for all $1 \leq k\in\bbK-1$,
\begin{align*}
\rho_k^{-1}\mu_k^{-1}\|\lambda^k\|^2&\leq\rho_k^{-1}\mu_k^{-1}\left( \|\lambda^0\|^2+\frac{2\rho_{k-1}\mu_{k-1}(\tf^*_{\rm hi}-\tf_{\rm low})}{1-\tau^4}+\frac{2\mu_{k-1}D_\bfy\epsilon_0}{1-\tau^3}\right)\\
&\leq\|\lambda^0\|^2+\frac{2(\tf^*_{\rm hi}-\tf_{\rm low})}{1-\tau^4}+\frac{2D_\bfy\epsilon_0}{1-\tau^3}\overset{\eqref{ht}}{=}2\vartheta.
\end{align*}
It implies that the conclusion of this lemma holds.
\end{proof}

The following lemma provides an upper bound on $\|[\tg(x^{k+1},z^{k+1})]_+\|$ and $\|[\tg(x^{k+1},y^{k+1})]_+\|$.

\begin{lemma}
Suppose that Assumptions \ref{a1} and \ref{a2} hold. Let $D_\bfy$, $\bbK$ and $\vartheta$ be defined in \eqref{DxDy}, \eqref{def-K} and \eqref{ht}, $L_f$, $L_\tf$ and $G$ be given in Assumptions \ref{a1} and \ref{a2}, and $\epsilon_0$, $\rho_k$ and $\mu_k$ be given in Algorithm \ref{AL-alg}. Suppose that $(x^{k+1},y^{k+1},z^{k+1}, \lambda^{k+1})$ is generated by Algorithm \ref{AL-alg} for some $0\leq k\in\bbK-1$ with 
\beq\label{muk-bnd}
\rho_k^{-1}\mu_k\geq8G^{-2}\vartheta.
\eeq
Then we have
\begin{align}
&\|[\tg(x^{k+1},z^{k+1})]_+\|\leq\mu_k^{-1}\|\lambda^{k+1}\|\leq2 \mu_k^{-1}G^{-1}(\epsilon_0+\rho_kL_\tf)D_\bfy,\label{y-cnstr}\\
&\|[\tg(x^{k+1},y^{k+1})]_+\|\leq\mu_k^{-1}\|[\lambda^k+\mu_k\tg(x^{k+1},y^{k+1})]_+\|\leq2 \mu_k^{-1}G^{-1}(\epsilon_0+L_f+\rho_kL_\tf)D_\bfy.\label{x-cnstr}
\end{align}
\end{lemma}

\begin{proof}
Suppose that $(x^{k+1},y^{k+1},z^{k+1}, \lambda^{k+1})$ is generated by Algorithm \ref{AL-alg} for some $0\leq k\in\bbK-1$ satisfying \eqref{muk-bnd}.  Notice that $(x^{k+1}, y^{k+1},z^{k+1})$ is an $\epsilon_k$-primal-dual stationary point of \eqref{AL-sub}. It then follows from \eqref{Lag}, Definition \ref{def1} and Assumption \ref{a1} that
\begin{align}
&\dist\big(0,\nabla_yf(x^{k+1},y^{k+1})+\rho_k\partial_y\tf(x^{k+1},y^{k+1})+\nabla_y\tg(x^{k+1},y^{k+1})[\lambda^k+\mu_k\tg(x^{k+1},y^{k+1})]_+\big)\leq\epsilon_k, \label{x-stat}\\
&\dist\big(0,-\rho_k\partial_z\tf(x^{k+1},z^{k+1})-\nabla_z \tg(x^{k+1},z^{k+1})[\lambda^k+\mu_k\tg(x^{k+1},z^{k+1})]_+\big)\leq\epsilon_k.\label{y-stat}
\end{align}

We first show that \eqref{y-cnstr} holds. Notice from Algorithm \ref{AL-alg} that  $\lambda^{k+1}=[\lambda^k+\mu_k\tg(x^{k+1},z^{k+1})]_+$. Hence, it follows from \eqref{y-stat} that there exists some $u\in\partial_z\tf(x^{k+1},z^{k+1})$ such that
\beq\label{y-res}
\|\rho_ku+\nabla_z \tg(x^{k+1},z^{k+1})\lambda^{k+1}\|\leq\epsilon_k.
\eeq
By Assumption \ref{a2}, there exists some $\hat z^{k+1}\in\mcY$ such that $-\tg_i(x^{k+1},\hat z^{k+1})\geq G$ for all $i$. Observe that $\langle\lambda^{k+1},\lambda^k+\mu_k\tg(x^{k+1},z^{k+1})\rangle=\|[\lambda^k+\mu_k\tg(x^{k+1},z^{k+1})]_+\|^2\geq0$, which implies that 
\beq \label{complim-ineq}
-\langle\lambda^{k+1},\mu_k^{-1}\lambda^k\rangle \leq \langle\lambda^{k+1},\tg(x^{k+1},z^{k+1})\rangle.
\eeq
 Using these, \eqref{y-res}, $\lambda^{k+1}\geq 0$ and $u\in\partial_z\tf(x^{k+1},z^{k+1})$, we have
\begin{align}
&\ \rho_k\tf(x^{k+1},z^{k+1})-\rho_k\tf(x^{k+1},\hat z^{k+1})+G\|\lambda^{k+1}\|_1-\langle\lambda^{k+1},\mu_k^{-1}\lambda^k\rangle\nn\\
&\ \leq \rho_k\tf(x^{k+1},z^{k+1})-\rho_k\tf(x^{k+1},\hat z^{k+1})+\langle\lambda^{k+1},-\tg(x^{k+1},\hat z^{k+1})-\mu_k^{-1}\lambda^k\rangle\nn \\
& \overset{\eqref{complim-ineq}}\leq \rho_k\tf(x^{k+1},z^{k+1})-\rho_k\tf(x^{k+1},\hat z^{k+1})+\langle\lambda^{k+1},\tg(x^{k+1},z^{k+1})-\tg(x^{k+1},\hat z^{k+1}))\rangle\nn \\
&\ \leq \langle \rho_ku, z^{k+1}-\hat z^{k+1}\rangle+\langle\nabla_z\tg(x^{k+1},z^{k+1})\lambda^{k+1},z^{k+1}-\hat z^{k+1}\rangle\nn \\
&\ = \langle \rho_ku+\nabla_z \tg(x^{k+1},z^{k+1})\lambda^{k+1},z^{k+1}-\hat z^{k+1}\rangle\leq D_\bfy\epsilon_k, \label{bound-ineq}
\end{align}
where the first inequality is due to $\lambda^{k+1}\geq 0$ and $-\tg_i(x^{k+1},\hat z^{k+1})\geq G$ for all $i$, the third inequality follows from  $u\in\partial_z \tf(x^{k+1}, z^{k+1})$, $\lambda^{k+1}\geq 0$ and the convexity of $\tf(x^{k+1},\cdot)$  and $\tg_i(x^{k+1},\cdot)$ for all $i$, and the last inequality is due to \eqref{DxDy}, \eqref{y-res} and $z^{k+1}
, \hat z^{k+1}\in\mcY$.

In view of \eqref{DxDy}, \eqref{bound-ineq}, $z^{k+1}
, \hat z^{k+1}\in\mcY$, and the Lipschitz continuity of $\tf$, one has
\begin{align}
 D_\bfy\epsilon_k+\rho_kL_\tf D_\bfy & \overset{\eqref{DxDy}}{\geq} D_\bfy\epsilon_k+\rho_kL_\tf\|z^{k+1}-\hat z^{k+1}\|  \geq  D_\bfy\epsilon_k+\rho_k(\tf(x^{k+1},\hat z^{k+1})-\tf(x^{k+1},z^{k+1}))\nn\\
&\overset{\eqref{bound-ineq}}{\geq} G\|\lambda^{k+1}\|_1-\langle\lambda^{k+1},\mu_k^{-1}\lambda^k\rangle\geq (G-\mu_k^{-1}\|\lambda^k\|)\|\lambda^{k+1}\|,\label{bnd-ineq}
\end{align}
where the first inequality is due to \eqref{DxDy} and $z^{k+1}
, \hat z^{k+1}\in\mcY$, the second inequality follows from $L_\tf$-Lipschitz continuity of $\tf$, and the last inequality is due to $\|\lambda^{k+1}\|_1\geq\|\lambda^{k+1}\|$. In addition, it follows from \eqref{ly-cnstr} and \eqref{muk-bnd} that
\[
G-\mu_k^{-1}\|\lambda^k\|\overset{\eqref{ly-cnstr}}{\geq}G-\sqrt{2\rho_k\mu_k^{-1}\vartheta}\ \overset{\eqref{muk-bnd}}{\geq}G/2,
\]
which together with \eqref{bnd-ineq} yields
\[
\|\lambda^{k+1}\|\leq 2G^{-1}(\epsilon_k+\rho_kL_\tf)D_\bfy.
\]
The statement \eqref{y-cnstr} then follows from this, $\epsilon_k\leq \epsilon_0$, and
\[
\|[\tg(x^{k+1},z^{k+1})]_+\|\leq\mu_k^{-1}\|[\lambda^k+\mu_k \tg(x^{k+1},z^{k+1})]_+\|=\mu_k^{-1}\|\lambda^{k+1}\|.
\]

We next show that \eqref{x-cnstr} holds. Indeed, let $\tlambda^{k+1}=[\lambda^k+\mu_k\tg(x^{k+1},y^{k+1})]_+$. It then follows from \eqref{x-stat} that
\begin{align*}
\dist\big(0,\nabla_y f(x^{k+1},y^{k+1})+\rho_k\partial_y\tf(x^{k+1},y^{k+1})+\nabla_y \tg(x^{k+1},y^{k+1})\tlambda^{k+1}\big)\leq\epsilon_k.
\end{align*}
Hence, there exists some $v\in\rho_k^{-1}\nabla_yf(x^{k+1},y^{k+1})+\partial_y\tf(x^{k+1},y^{k+1})$ such that
\[
\|\rho_kv+\nabla_y \tg(x^{k+1},y^{k+1})\tlambda^{k+1}\|\leq\epsilon_k.
\]
The rest of the proof of \eqref{x-cnstr} is similar to the one of \eqref{y-cnstr} with $u$, $z^{k+1}$ and $\lambda^{k+1}$ being replaced with $v$, $y^{k+1}$ and $\tlambda^{k+1}$ respectively and thus omitted.
\end{proof}

The next lemma provides an upper bound on the quantities associated with an approximate KKT solution  
$(x^{k+1},y^{k+1})$.

\begin{lemma}\label{l-subdcnstr}
Suppose that Assumptions \ref{a1} and \ref{a2} hold.  Let $D_\bfy$, $\bbK$ and $\vartheta$ be defined in \eqref{DxDy}, \eqref{def-K} and \eqref{ht}, $L_f$, $L_\tf$ and $G$ be given in Assumptions \ref{a1} and \ref{a2}, and $\epsilon_0$, $\tau$, $\rho_k$ and $\mu_k$ be given in Algorithm \ref{AL-alg}.
Suppose that $(x^{k+1},y^{k+1},z^{k+1},\lambda^{k+1})$ is generated by Algorithm \ref{AL-alg} for  some $0\leq k\in\bbK-1$ with 
\beq\label{muk-1}
\rho_k^{-1}\mu_k\geq8\tau^{-2}G^{-2}\vartheta.
\eeq
Let 
\beq \label{def-tlx} 
 \lambda^{k+1}_\bfy=[\lambda^k+\mu_k\tg(x^{k+1},y^{k+1})]_+,\qquad\lambda^{k+1}_\bfz=\rho_k^{-1}[\lambda^k+\mu_k\tg(x^{k+1},z^{k+1})]_+.
\eeq
Then we have
\begin{align}
&\dist\Big(0,\partial f(x^{k+1},y^{k+1})+\rho_k\partial\tf(x^{k+1},y^{k+1}) - \rho_k\big(\nabla_x\tf(x^{k+1},z^{k+1})+\nabla_x\tg(x^{k+1},z^{k+1})\lambda_\bfz^{k+1}; 0\big)\nn\\
&\qquad+\nabla \tg(x^{k+1},y^{k+1})\lambda_\bfy^{k+1}\Big)\leq\epsilon_k,\label{pF-x}\\
& \dist\Big(0,\rho_k\big(\partial_z\tf(x^{k+1},z^{k+1})+\nabla_z\tg(x^{k+1},z^{k+1})\lambda_\bfz^{k+1}\big)\Big)\leq\epsilon_k,\label{pF-y} \\
& \|[\tg(x^{k+1},z^{k+1})]_+\|\leq2\mu_k^{-1}G^{-1}(\epsilon_0+\rho_kL_\tf)D_\bfy  , \label{y-feas}\\ 
&|\langle\lambda^{k+1}_\bfz,\tg(x^{k+1},z^{k+1})\rangle|\leq2\rho_k^{-1}\mu_k^{-1}G^{-1}(\epsilon_0+\rho_kL_\tf)D_\bfy \max\{\|\lambda^0\|,\ 2G^{-1}(\epsilon_0+\rho_kL_\tf)D_\bfy\}, \label{y-complim}\\
& \|[\tg(x^{k+1},y^{k+1})]_+\|\leq2\mu_k^{-1}G^{-1}(\epsilon_0+L_f+\rho_kL_\tf)D_\bfy  , \label{x-feas}\\ 
&|\langle\lambda^{k+1}_\bfy,\tg(x^{k+1},y^{k+1})\rangle|\leq2\mu_k^{-1}G^{-1}(\epsilon_0+L_f+\rho_kL_\tf)D_\bfy \max\{\|\lambda^0\|,\ 2G^{-1}(\epsilon_0+L_f+\rho_kL_\tf)D_\bfy\}. \label{x-complim}
\end{align}
\end{lemma}

\begin{proof}
Suppose that $(x^{k+1},y^{k+1},z^{k+1},\lambda^{k+1})$ is generated by Algorithm \ref{AL-alg} for  some $0\leq k\in\bbK-1$ satisfying \eqref{muk-1}.  Notice that $(x^{k+1}, y^{k+1},z^{k+1})$ is an $\epsilon_k$-primal-dual stationary point of \eqref{AL-sub}. It then follows from Definition \ref{def1} that
\begin{align}
&\dist\big(0,\partial_{(x,y)}\mcL(x^{k+1},y^{k+1},z^{k+1},\lambda^k;\rho_k,\mu_k)\big)\leq\epsilon_k,\label{x-stationary}\\
&\dist\big(0,\partial_z\mcL(x^{k+1},y^{k+1},z^{k+1},\lambda^k;\rho_k,\mu_k)\big)\leq\epsilon_k.\label{y-stationary}
\end{align}
In view of these, \eqref{Lag} and \eqref{def-tlx}, one has
\begin{align*}
&\partial_{(x,y)}\mcL(x^{k+1},y^{k+1},z^{k+1},\lambda^k;\rho_k,\mu_k) \\
&=\ \partial f(x^{k+1},y^{k+1})+\rho_k\partial\tf(x^{k+1},y^{k+1})+\nabla \tg(x^{k+1},y^{k+1})[\lambda^k+\mu_k\tg(x^{k+1},y^{k+1})]_+\\
&\quad\,\, -\big(\rho_k\nabla_x\tf(x^{k+1},z^{k+1})+\nabla_x\tg(x^{k+1},z^{k+1})[\lambda^k+\mu_k\tg(x^{k+1},z^{k+1})]_+;0\big)\\
&=\ \partial f(x^{k+1},y^{k+1})+\rho_k\partial\tf(x^{k+1},y^{k+1})-\rho_k\big(\nabla_x\tf(x^{k+1},z^{k+1})+\nabla_x \tg(x^{k+1},z^{k+1})\lambda^{k+1}_\bfz;0\big)\\
&\quad\,\,+\nabla\tg(x^{k+1},y^{k+1})\lambda^{k+1}_\bfy, \\
&\partial_z\mcL(x^{k+1},y^{k+1},z^{k+1},\lambda^k;\rho_k,\mu_k)= \ -\rho_k\partial_z \tf(x^{k+1},z^{k+1})-\nabla_z \tg(x^{k+1},z^{k+1})[\lambda^k+\mu_k\tg(x^{k+1},z^{k+1})]_+\\
&=\ -\rho_k\big(\partial_z\tf(x^{k+1},z^{k+1})+\nabla_z\tg(x^{k+1},z^{k+1})\lambda_\bfz^{k+1}\big).
\end{align*}
These relations together with  \eqref{x-stationary} and \eqref{y-stationary} imply that \eqref{pF-x} and \eqref{pF-y} hold.

Notice from Algorithm \ref{AL-alg} that $0<\tau<1$, which together with \eqref{muk-1} implies that \eqref{muk-bnd} holds for $\mu_k$ and $\rho_k$. It then follows that \eqref{y-cnstr} and \eqref{x-cnstr} hold, which immediately yields \eqref{y-feas}, \eqref{x-feas}, and 
\beq\label{lambda-bnd}
\|\lambda^{k+1}\|\leq2G^{-1}(\epsilon_0+\rho_kL_\tf)D_\bfy, \quad \|[\lambda^k+\mu_k\tg(x^{k+1},y^{k+1})]_+\|\leq2G^{-1}(\epsilon_0+L_f+\rho_kL_\tf)D_\bfy.
\eeq
Also, notice from \eqref{def-tlx} and $\lambda^{k+1}=[\lambda^k+\mu_k\tg(x^{k+1},z^{k+1})]_+$ that 
$\lambda_\bfz^{k+1}=\rho_k^{-1}\lambda^{k+1}$. By this, \eqref{def-tlx} and \eqref{lambda-bnd},  one has
\beq \label{lambday-bnd1}
\|\lambda_\bfz^{k+1}\| \leq2\rho_k^{-1}G^{-1}(\epsilon_0+\rho_kL_\tf)D_\bfy, \qquad  \|\lambda_\bfy^{k+1}\|\leq2G^{-1}(\epsilon_0+L_f+\rho_kL_\tf)D_\bfy.
\eeq
Observe from \eqref{def-tlx} that $\langle\lambda_\bfy^{k+1},\lambda^k+\mu_k\tg(x^{k+1},y^{k+1})\rangle=\|[\lambda^k+\mu_k\tg(x^{k+1},y^{k+1})]_+\|^2\geq0$, which implies that 
\beq \label{y-complim-ineq}
-\langle\lambda_\bfy^{k+1},\mu_k^{-1}\lambda^k\rangle \leq \langle\lambda_\bfy^{k+1},\tg(x^{k+1},y^{k+1})\rangle.
\eeq
In addition, we claim that 
\beq \label{lambday-bnd2}
\|\lambda^k\|\leq\max\{\|\lambda^0\|,\ 2G^{-1}(\epsilon_0+\rho_kL_\tf)D_\bfy\}.
\eeq
Indeed, \eqref{lambday-bnd2} clearly holds if $k=0$.  We now assume that $k>0$. Notice from Algorithm \ref{AL-alg} that $\mu_{k-1}=\tau^3\mu_k$ and $\rho_{k-1}=\tau\rho_k$, which along with \eqref{muk-1} imply that $\rho_{k-1}^{-1}\mu_{k-1}\geq8G^{-2}\vartheta$. By this and Lemma \ref{l-ycnstr} with $k$ replaced by $k-1$, one can conclude that $\|\lambda^k\|\leq2G^{-1}(\epsilon_0+\rho_{k-1} L_\tf)D_\bfy$.  This together with $\rho_{k-1}<\rho_k$ implies that \eqref{lambday-bnd2} holds as desired.

We next show that \eqref{y-complim} and \eqref{x-complim} hold. By $\lambda_\bfy^{k+1},\lambda_\bfz^{k+1}\geq 0$,  \eqref{y-feas}, \eqref{x-feas}, \eqref{lambday-bnd1}, \eqref{y-complim-ineq} and \eqref{lambday-bnd2}, one has 
\begin{align*}
\langle\lambda_\bfz^{k+1},\tg(x^{k+1},z^{k+1})\rangle & \ \ \leq \ \langle\lambda_\bfz^{k+1},[\tg(x^{k+1},z^{k+1})]_+\rangle \leq\|\lambda_\bfz^{k+1}\|\|[\tg(x^{k+1},z^{k+1})]_+\|\\ 
& \overset{\eqref{y-feas}\eqref{lambday-bnd1}}\leq4\rho_k^{-1}\mu_k^{-1}G^{-2}(\epsilon_0+\rho_kL_\tf)^2D_\bfy^2,\\
\langle\lambda^{k+1}_\bfz,\tg(x^{k+1},z^{k+1})\rangle & = \rho_k^{-1}\langle\lambda^{k+1},\tg(x^{k+1},z^{k+1})\rangle \overset{\eqref{complim-ineq}}{\geq}-\rho_k^{-1}\langle\lambda^{k+1},\mu_k^{-1}\lambda^k\rangle\geq-\rho_k^{-1}\mu_k^{-1}\|\lambda^{k+1}\|\|\lambda^k\|\\
&\ \overset{\eqref{lambda-bnd}\eqref{lambday-bnd2}}{\geq}-2\rho_k^{-1}\mu_k^{-1}G^{-1}(\epsilon_0+\rho_k L_\tf)D_\bfy \max\{\|\lambda^0\|,\ 2G^{-1}(\epsilon_0+\rho_kL_\tf)D_\bfy\},\\
\langle\lambda_\bfy^{k+1},\tg(x^{k+1},y^{k+1})\rangle & \ \ \leq \ \langle\lambda_\bfy^{k+1},[\tg(x^{k+1},y^{k+1})]_+\rangle \leq\|\lambda_\bfy^{k+1}\|\|[\tg(x^{k+1},y^{k+1})]_+\|\\ 
& \overset{\eqref{x-feas}\eqref{lambday-bnd1}}\leq4\mu_k^{-1}G^{-2}(\epsilon_0+L_f+\rho_kL_\tf)^2D_\bfy^2,\\
\langle\lambda_\bfy^{k+1},\tg(x^{k+1},y^{k+1})\rangle&\ \   \overset{\eqref{y-complim-ineq}}\geq\langle\lambda_\bfy^{k+1},-\mu_k^{-1}\lambda^k\rangle\geq-\mu_k^{-1}\|\lambda_\bfy^{k+1}\|\|\lambda^k\|\\
&\ \overset{\eqref{lambday-bnd1}\eqref{lambday-bnd2}}{\geq}-2\mu_k^{-1}G^{-1}(\epsilon_0+L_f+\rho_k L_\tf)D_\bfy \max\{\|\lambda^0\|,\ 2G^{-1}(\epsilon_0+\rho_kL_\tf)D_\bfy\}.
\end{align*}
These relations imply that \eqref{y-complim} and \eqref{x-complim} hold.
\end{proof}

The following lemma provides an estimate on operation complexity at step 3 of Algorithm~\ref{AL-alg} for  problem \eqref{prob} with $\sigma=0$, i.e., $\tf_1(x,\cdot)$ being convex but not strongly convex for any given $x\in\dom\,f_2$.

\begin{lemma}\label{l-subp}
Suppose that Assumption \ref{a1} holds with $\sigma=0$, i.e., $\tf_1(x,\cdot)$ being convex but not strongly convex for any given $x\in\dom\,f_2$.  Let $f^*$, $L_k$, $\tf^*_{\rm hi}$, $D_\bfx$, $D_\bfy$, $f_{\rm hi}$, $f_{\rm low}$ $\tf_{\rm low}$, $\tg_{\rm hi}$, $\bbK$ and $\vartheta$ be defined in \eqref{prob}, \eqref{Lk}, \eqref{def-tFx}, \eqref{DxDy}, \eqref{fhi}, \eqref{tfhi}, \eqref{def-K} and \eqref{ht}, $L_\tf$ be given in Assumption \ref{a1}, $\epsilon_k$, $\rho_k$ and $\mu_k$ be given in Algorithm \ref{AL-alg}, and
\begin{align}
\alpha_k=&\ \min\Big\{1,\sqrt{4\epsilon_k/(D_\bfy L_k)}\Big\},\label{mmax-omega}\\
\delta_k=&\ (2+\alpha_k^{-1})L_k (D_\bfx^2+D_\bfy^2)+\max\left\{\epsilon_k/D_\bfy,\alpha_k L_k/4\right\}D_\bfy^2,\label{mmax-zeta}\\
M_k=&\ \frac{16\max\left\{1/(2L_k),\min\left\{D_\bfy/\epsilon_k,4/(\alpha_k L_k)\right\}\right\}\mu_k}{\left[(3L_k+\epsilon_k/(2D_\bfy))^2/\min\{L_k,\epsilon_k/(2D_\bfy)\}+ 3L_k+\epsilon_k/(2D_\bfy)\right]^{-2}\epsilon_k^2}\times\Big(\delta_k+2\alpha_k^{-1}\nn\\
&\times\big(f^*-f_{\rm low}+\rho_k(\tf^*_{\rm hi}-\tf_{\rm low})+\rho_kL_\tf D_\bfy+3\rho_k\vartheta+\mu_k\tg_{\rm hi}^2+\epsilon_k D_\bfy/4+L_k (D_\bfx^2+D_\bfy^2)\big)\Big),\label{mmax-Mk}\\
T_k=&\ \left\lceil16\left(f_{\rm hi}-f_{\rm low}+\rho_k\epsilon_k+\epsilon_k D_\bfy/4\right) L_k\epsilon_k^{-2}+8(1+4D_\bfy^2L_k^2\epsilon_k^{-2})\mu_k^{-1}-1\right\rceil_+,\label{mmax-Tk}\\
N_k=&\ \left(\left\lceil96\sqrt{2}\left(1+\left(24L_k+4\epsilon_k/D_\bfy\right)L_k^{-1}\right)\right\rceil+2\right)\max\Big\{2,\sqrt{D_\bfy L_k\epsilon_k^{-1}}\Big\}\notag\\
&\ \times\left((T_k+1)(\log M_k)_++T_k+1+2T_k\log(T_k+1) \right).\label{mmax-N}
\end{align}
Then for all $0\leq k\in\bbK-1$, an $\epsilon_k$-primal-dual stationary point $(x^{k+1},y^{k+1},z^{k+1})$ of problem \eqref{AL-sub} is successfully found at step 3 of Algorithm~\ref{AL-alg} that satisfies 
\begin{align}
\max_z\mcL(x^{k+1},y^{k+1},z,\lambda^k;\rho_k,\mu_k)\leq&\  f_{\rm hi}+\rho_k\epsilon_k+\frac{\epsilon_k D_\bfy}{4}+\frac{1}{2\mu_k}\left(L_k^{-1}\epsilon_k^2+4D_\bfy^2L_k\right).\label{upperbnd-k}
\end{align}
Moreover, the total number of evaluations of $\nabla f_1$, $\nabla \tf_1$, $\nabla \tg$ and proximal operators of $f_2$ and $\tf_2$ performed at step 3 in iteration $k$ of Algorithm~\ref{AL-alg} is no more than $N_k$, respectively.
\end{lemma}

\begin{proof}
Observe from \eqref{Lag} and Assumption \ref{a1} that problem \eqref{AL-sub} can be viewed as
\[
\min_{x,y}\max_z\{h(x,y,z)+p(x,y)-q(z)\}
\]
with
\begin{align*}
&h(x,y,z)=f_1(x,y)+\rho_k\tf_1(x,y)+\frac{1}{2\mu_k}\|[\lambda^k+\mu_k \tg(x,y)]_+\|^2-\rho_k \tf_1(x,z)-\frac{1}{2\mu_k}\|[\lambda^k+\mu_k \tg(x,z)]_+\|^2,\nn\\
&p(x,y)=f_2(x)+\rho_k\tf_2(y),\quad q(z)=\rho_k\tf_2(z).
\end{align*}
By \eqref{tghi} and Assumption \ref{a1}, it can be verified that $\|[\lambda^k+\mu_k\tg (x,y)]_+\|^2/(2\mu_k)$ and $\|[\lambda^k+\mu_k\tg (x,z)]_+\|^2/(2\mu_k)$ are both $(\mu_kL_\tg^2+\mu_k\tg_{\rm hi}L_{\nabla\tg}+\|\lambda^k\| L_{\nabla\tg})$-smooth on $\mcX\times\mcY$. Using this and the fact that $f_1$ and $\tf_1$ are respectively $L_{\nabla f_1}$- and $L_{\nabla\tf_1}$-smooth on $\mcX\times\mcY$, we can see that $h(x,y,z)$ is $L_k$-smooth on $\mcX\times\mcY\times\mcY$ for all $0\leq k\in\bbK-1$, where $L_k$ is given in \eqref{Lk}. In addition, it follows from Assumption \ref{a1} and $\sigma=0$ that $h(x,y,\cdot)$ is concave but not strongly concave. Consequently, it follows from Theorem \ref{mmax-thm} (see Appendix \ref{appendix-B}) that an $\epsilon_k$-primal-dual stationary point $(x^{k+1},y^{k+1},z^{k+1})$ of problem \eqref{AL-sub} is successfully found by Algorithm \ref{mmax-alg2} at step 3 of Algorithm~\ref{AL-alg}.

In addition, by \eqref{Lag}, \eqref{tLag} and \eqref{fhi}, one has
\begin{align}
\min_{x,y}\max_z\mcL(x,y,z,\lambda^k;\rho_k,\mu_k)\overset{\eqref{Lag}\eqref{tLag}}{=}&\min_{x,y}\Big\{f(x,y)+\rho_k\tcL(x,y,\lambda^k;\rho_k,\mu_k)-\min_z\rho_k\tcL(x,z,\lambda^k;\rho_k,\mu_k)\Big\}\nn\\
\geq\ \ &\min_{(x,y)\in\mcX\times\mcY}f(x,y)\overset{\eqref{fhi}}{=} f_{\rm low}.\label{e1}
\end{align}
Let $(x^*,y^*)$ be an optimal solution of \eqref{prob}. It then follows that $f(x^*,y^*)=f^*$, $\tf(x^*,y^*)=\tf^*(x^*)$ and $\tg(x^*,y^*)\leq0$, where $f^*$ and $\tf^*$ are defined in \eqref{prob} and \eqref{tfstarx}, respectively. Using these, \eqref{Lag}, \eqref{tLag}, \eqref{def-tFx}, \eqref{tfhi} and \eqref{ly-cnstr}, we obtain that
\begin{align}
&\ \ \min_{x,y}\max_z\mcL(x,y,z,\lambda^k;\rho_k,\mu_k)\leq\max_z\mcL(x^*,y^*,z,\lambda^k;\rho_k,\mu_k)\nn\\
&\ \ \overset{\eqref{Lag}\eqref{tLag}}{=}f(x^*,y^*)+\rho_k\tf(x^*,y^*)+\frac{1}{2\mu_k}\|[\lambda^k+\mu_k\tg(x^*,y^*)]_+\|^2-\min_z\rho_k\tL(x^*,z,\lambda^k;\rho_k,\mu_k)\nn\\
&\quad\, \leq f^*+\rho_k\tf^*(x^*)+\frac{1}{2\mu_k}\|\lambda^k\|^2-\min_z\left\{\rho_k\tf(x^*,z)+\frac{1}{2\mu_k}\|[\lambda^k+\mu_k\tg(x^*,z)]_+\|^2\right\}\nn\\
&\ \overset{\eqref{def-tFx}\eqref{tfhi}}{\leq} f^*+\rho_k(\tf^*_{\rm hi}-\tf_{\rm low})+\frac{1}{2\mu_k}\|\lambda^k\|^2\overset{\eqref{ly-cnstr}}{\leq} f^*+\rho_k(\tf^*_{\rm hi}-\tf_{\rm low})+\rho_k\vartheta,\label{e2}
\end{align}
where the second inequality is due to $\tf(x^*,y^*)=\tf^*(x^*)$, $\tg(x^*,y^*)\leq0$, and \eqref{tLag}. Also, by \eqref{Lag}, \eqref{DxDy}, \eqref{fhi}, \eqref{tfhi} and \eqref{ly-cnstr}, one has
\begin{align}
&\min_{(x,y,z)\in\mcX\times\mcY\times\mcY}\mcL(x,y,z,\lambda^k;\rho_k,\mu_k)\nn\\
&\overset{\eqref{Lag}}{\geq}\min_{(x,y,z)\in\mcX\times\mcY\times\mcY}\left\{f(x,y)+\rho_k(\tf(x,y)-\tf(x,z))-\frac{1}{2\mu_k}\|[\lambda^k+\mu_k\tg(x,z)]_+\|^2\right\}\nn\\
&\ \geq\min_{(x,y,z)\in\mcX\times\mcY\times\mcY}\left\{f(x,y)-\rho_kL_\tf\|y-z\|-\frac{1}{2\mu_k}\left(\|\lambda^k\|+\mu_k\|[\tg(x,z)]_+\|\right)^2\right\}\nn\\
&\ \geq\min_{(x,y,z)\in\mcX\times\mcY\times\mcY}\left\{f(x,y)-\rho_kL_\tf\|y-z\|-\frac{1}{\mu_k}\|\lambda^k\|^2-\mu_k\|[\tg(x,z)]_+\|^2\right\}\nn\\
&\ \geq f_{\rm low}-\rho_kL_\tf D_\bfy-2\rho_k\vartheta-\mu_k\tg_{\rm hi}^2,\label{e3}
\end{align}
where the second inequality is due to $\lambda^k\in\bR_+^l$ and $L_\tf$-Lipschitz continuity of $\tf$ (see Assumption \ref{a1}(i)), and the last inequality is due to \eqref{DxDy}, \eqref{fhi}, \eqref{tfhi} and \eqref{ly-cnstr}. Notice from step 2 of Algorithm~\ref{AL-alg} that $y_{\rm init}^k$ is an approximate solution of $\min_z\tL(x^{k},z,\lambda^k;\rho_k,\mu_k)$ satisfying \eqref{y-nf}. It then follows from \eqref{Lag}, \eqref{tLag}, \eqref{y-nf} and \eqref{fhi} that
\begin{align}
\max_z\mcL(x^k,y_{\rm init}^k,z,\lambda^k;\rho_k,\mu_k)\overset{\eqref{Lag}\eqref{tLag}}=&f(x^k,y_{\rm init}^k)+\rho_k\Big(\tcL(x^k,y_{\rm init}^k,\lambda^k;\rho_k,\mu_k)-\min_z\tcL(x^k,z,\lambda^k;\rho_k,\mu_k)\Big)\nn\\
\overset{\eqref{y-nf}}{\leq}\ &f(x^k,y_{\rm init}^k)+\rho_k\epsilon_k\overset{\eqref{fhi}}{\leq} f_{\rm hi}+\rho_k\epsilon_k.\label{init-bnd}
\end{align}

To complete the rest of the proof, let
\begin{align}
&H(x,y,z)=\mcL(x,y,z,\lambda^k;\rho_k,\mu_k),\quad H^*=\min_{x,y}\max_z\mcL(x,y,z,\lambda^k;\rho_k,\mu_k),\label{Hk1}\\
&H_{\rm low}=\min\left\{\mcL(x,y,z,\lambda^k;\rho_k,\mu_k)|(x,y,z)\in\mcX\times\mcY\times\mcY\right\}.\label{Hk2}
\end{align}
In view of these, \eqref{e1}, \eqref{e2}, \eqref{e3} and \eqref{init-bnd}, we obtain that
\begin{align*}
&\max_zH(x^k,y_{\rm init}^k,z)\overset{\eqref{init-bnd}}{\leq} f_{\rm hi}+\rho_k\epsilon_k,\\
& f_{\rm low}\overset{\eqref{e1}}{\leq} H^*\overset{\eqref{e2}}{\leq} f^*+\rho_k(\tf^*_{\rm hi}-\tf_{\rm low})+\rho_k\vartheta,\\
& H_{\rm low}\overset{\eqref{e3}}{\geq} f_{\rm low}-\rho_kL_\tf D_\bfy-2\rho_k\vartheta-\mu_k\tg_{\rm hi}^2.
\end{align*}
Using these and Theorem \ref{mmax-thm} (see Appendix \ref{appendix-B}) with $\hat x^0=(x^k,y_{\rm init}^k)$, $\epsilon=\epsilon_k$, $\hat\epsilon_0=\epsilon_k/(2\sqrt{\mu_k})$, $L_{\nabla\bh}=L_k$, $\widehat L=3L_k+\epsilon_k/(2D_\bfy)$, $\sigma_y=0$, $\hat\sigma_y=\epsilon_k/(2D_\bfy)$, $\hat\alpha=\alpha_k$, $\hat\delta=\delta_k$, $D_p=\sqrt{D_\bfx^2+D_\bfy^2}$, $D_q=D_\bfy$, and $\bH$, $\bH^*$, $\bH_{\rm low}$ given in \eqref{Hk1} and \eqref{Hk2}, we can conclude that 
the $\epsilon_k$-primal-dual stationary point $(x^{k+1},y^{k+1},z^{k+1})$ of problem \eqref{AL-sub} found at step 3 of Algorithm~\ref{AL-alg}  satisfies \eqref{upperbnd-k}. Moreover, the total number of evaluations of $\nabla f_1$, $\nabla \tf_1$, $\nabla \tg$ and proximal operators of $f_2$ and $\tf_2$ performed by Algorithm \ref{mmax-alg2} at step 3 of Algorithm~\ref{AL-alg}  is no more than $N_k$, respectively.
\end{proof}

The next lemma presents an upper bound on the optimality violation of $y^{k+1}$ for the lower-level problem of \eqref{prob} when $\sigma=0$ and $x=x^{k+1}$.

\begin{lemma}\label{l-fgap}
Suppose that Assumptions \ref{a1} and \ref{a2} hold with $\sigma=0$, i.e., $\tf_1(x,\cdot)$ being convex but not strongly convex for any given $x\in\dom\,f_2$.  Let $\tf^*$, $L_k$, $D_\bfy$, $f_{\rm hi}$, $f_{\rm low}$ and $\bbK$ be defined in \eqref{tfstarx}, \eqref{Lk}, \eqref{DxDy}, \eqref{fhi} and \eqref{def-K}, $L_f$, $L_\tf$ and $G$ be given in Assumptions \ref{a1} and \ref{a2}, and $\epsilon_k$, $\rho_k$, $\mu_k$ and $\lambda^0$ be given in Algorithm \ref{AL-alg}. Suppose that $(x^{k+1},y^{k+1},\lambda^{k+1})$ is generated by Algorithm \ref{AL-alg} for some $0\leq k\in\bbK-1$ satisfying \eqref{muk-1}. Then we have
\begin{align*}
|\tf(x^{k+1},y^{k+1})-\tf^*(x^{k+1})|\leq\max\Bigg\{&2\mu_k^{-1}G^{-2}L_\tf(\epsilon_0+L_f+\rho_kL_\tf)D_\bfy^2,\nn\\
&\ \rho_k^{-1}\mu_k^{-1}\max\{\|\lambda^0\|,\ 2G^{-1}(\epsilon_0+\rho_kL_\tf)D_\bfy\}/2\nn\\
&+\rho_k^{-1 }\left(f_{\rm hi}-f_{\rm low}+\rho_k\epsilon_k+\frac{\epsilon_k D_\bfy}{4}+\frac{1}{2\mu_k}\left(L_k^{-1}\epsilon_k^2+4D_\bfy^2L_k\right)\right)\Bigg\}.
\end{align*}
\end{lemma}

\begin{proof}
Notice from \eqref{muk-1} and the proof of Lemma \ref{l-subdcnstr} that \eqref{lambday-bnd2} holds. Using this, \eqref{Lag}, \eqref{tLag}, \eqref{fhi} and \eqref{p-ineq}, we have
\begin{align*}
&\ \ \max_z\mcL(x^{k+1},y^{k+1},z,\lambda^k;\rho_k,\mu_k)\\
&\ \ \overset{\eqref{Lag}\eqref{tLag}}= f(x^{k+1},y^{k+1})+\rho_k \tf(x^{k+1},y^{k+1})+\frac{1}{2\mu_k}\|[\lambda^k+\mu_k\tg(x^{k+1},y^{k+1})]_+\|^2-\min_z\rho_k\tcL(x^{k+1},z,\lambda^k;\rho_k,\mu_k) \\ 
&\ \ \geq f(x^{k+1},y^{k+1})+\rho_k \tf(x^{k+1},y^{k+1})-\min_z\rho_k\tcL(x^{k+1},z,\lambda^k;\rho_k,\mu_k) \\ 
&\ \overset{\eqref{fhi}\eqref{p-ineq}}{\geq}f_{\rm low}+\rho_k\big(\tf(x^{k+1},y^{k+1})-\tf^*(x^{k+1})\big)-\frac{1}{2\mu_k}\|\lambda^k\|^2\\
&\quad \overset{\eqref{lambday-bnd2}}{\geq}f_{\rm low}+\rho_k\big(\tf(x^{k+1},y^{k+1})-\tf^*(x^{k+1})\big)-\mu_k^{-1}\max\{\|\lambda^0\|,\ 2G^{-1}(\epsilon_0+\rho_kL_\tf)D_\bfy\}/2.
\end{align*}
This together with \eqref{upperbnd-k} implies that
\begin{align}
\tf(x^{k+1},y^{k+1})-\tf^*(x^{k+1})\ \leq&\ \rho_k^{-1 }\Big(f_{\rm hi}-f_{\rm low}+\rho_k\epsilon_k+\frac{\epsilon_k D_\bfy}{4}+\frac{1}{2\mu_k}(L_k^{-1}\epsilon_k^2+4D_\bfy^2L_k)\Big)\nn\\
&+\rho_k^{-1}\mu_k^{-1}\max\{\|\lambda^0\|,\ 2G^{-1}(\epsilon_0+\rho_kL_\tf)D_\bfy\}/2.\label{f-gap}
\end{align}
On the other hand, let $\lambda^*\in\bR^l_+$ be an optimal Lagrangian multiplier of problem \eqref{tfstarx} with $x=x^{k+1}$. It then follows from Lemma \ref{dual-bnd}(i) that  $\|\lambda^*\|\leq G^{-1}L_\tf D_\bfy$. Using these, \eqref{tfstarx} and \eqref{x-feas}, we have
\begin{align*}
\tf^*(x^{k+1})=&\min_y\big\{\tf(x^{k+1},y)+\langle\lambda^*,\tg(x^{k+1},y)\rangle\big\}\leq\tf(x^{k+1},y^{k+1})+\langle\lambda^*,\tg(x^{k+1},y^{k+1})\rangle\\
\leq & \ \tf(x^{k+1},y^{k+1})+\|\lambda^*\|\|[\tg(x^{k+1},y^{k+1})]_+\|\leq\tf(x^{k+1},y^{k+1})+2\mu_k^{-1}G^{-2}L_\tf(\epsilon_0+L_f+\rho_kL_\tf)D_\bfy^2.
\end{align*}
The conclusion of this lemma then follows from this and \eqref{f-gap}.
\end{proof}

The following lemma provides an operation complexity of step 2 of Algorithm~\ref{AL-alg} for  problem \eqref{prob} with $\sigma=0$, i.e., $\tf_1(x,\cdot)$ being convex but not strongly convex for any given $x\in\dom\,f_2$.

\begin{lemma}\label{l-init}
Suppose that Assumption \ref{a1} holds with $\sigma=0$, i.e., $\tf_1(x,\cdot)$ being convex but not strongly convex for any given $x\in\dom\,f_2$. Let $\tL_k$, $D_\bfy$ and $\bbK$ be defined in \eqref{tLk}, \eqref{DxDy} and \eqref{def-K}, $\epsilon_k$ be given in Algorithm \ref{AL-alg}, and 
\beq\label{tNk}
N_k'=\left\lceil D_\bfy\sqrt{2\epsilon_k^{-1}\tL_k}\,\right\rceil.
\eeq
Then for all $0\leq k\in\bbK-1$, $y_{\rm init}^k$ satisfying \eqref{y-nf} is found at step 2 of  Algorithm \ref{AL-alg} by Algorithm \ref{opt} in no more than $N_k'$ evaluations of $\nabla \tf_1$, $\nabla\tg$ and the proximal operator of $\tf_2$, respectively.
\end{lemma}

\begin{proof}
 Notice from \eqref{tLag} and Algorithm \ref{AL-alg} that $y_{\rm init}^k$ satisfying \eqref{y-nf} is found by Algorithm \ref{opt} applied to the problem
\begin{equation*}
\min_y\left\{\tL(x^k,y,\lambda^k;\rho_k,\mu_k)=\phi(y)+P(y)\right\},
\end{equation*}
where $\phi(y)=\tf_1(x^k,y)+\|[\lambda^k+\mu_k\tg(x^k,y)]_+\|^2/(2\rho_k\mu_k)$ and $P(y)=\tf_2(y)$. By Assumption \ref{a1}, $\sigma=0$ and \eqref{tghi}, one can see that $\phi$ is convex but not strongly convex and $\tL_k$-smooth on $\dom\,P$, where $\tL_k$ is given in \eqref{tLk}. It then follows from this and Theorem \ref{thm-opt} (see Appendix \ref{appendix-A}) with $\tilde\epsilon=\epsilon_k$, $D_P=D_\bfy$ and  $L_{\nabla\phi}=\tL_k$  that Algorithm \ref{opt} finds $y_{\rm init}^k$ satisfying \eqref{y-nf} in no more than $N_k'$ iterations. Notice that each iteration of Algorithm \ref{opt} requires one evaluation of $\nabla\phi$ and the proximal operator of $P$, respectively. Hence, the conclusion of this lemma holds.
\end{proof}

We are now ready to prove Theorem \ref{complexity}.

\begin{proof}[\textbf{Proof of Theorem \ref{complexity}}]
(i) Observe from the definition of $K$ in \eqref{def-K} and $\epsilon_k=\epsilon_0\tau^k$ that $K$ is the smallest nonnegative integer such that $\epsilon_K\leq\varepsilon$. Hence, Algorithm \ref{AL-alg} terminates and outputs $(x^{K+1},y^{K+1})$ after $K+1$ outer iterations. Also, one can see from Algorithm \ref{AL-alg} that 
\beq\label{eps}
\rho_K=\epsilon_K^{-1},\quad\mu_K=\epsilon_K^{-3}, \quad  \eta_K=\epsilon_K.
\eeq 
Moreover, notice from the assumption of Theorem \ref{complexity} that $\varepsilon^{-2}-8\tau^{-2}G^{-2}\vartheta \geq 0$. It then follows from this and \eqref{eps} that
\begin{align*}
\rho_K^{-1}\mu_K=\epsilon_K^{-2} \geq \varepsilon^{-2} \geq  8\tau^{-2}G^{-2}\vartheta,
\end{align*}
which implies that \eqref{muk-1} holds for $k=K$. In addition, by \eqref{Lk}, \eqref{hL}, \eqref{ly-cnstr} and $\mu_k\geq\rho_k\geq1$, one has that for all $0\leq k\in\bbK-1$,
\begin{align}
& 2\mu_kL_\tg^2\leq L_k\overset{\eqref{Lk}}{=} L_{\nabla f_1}+2\rho_kL_{\nabla\tf_1}+2\mu_kL_\tg^2+2\mu_k\tg_{\rm hi}L_{\nabla \tg}+2\|\lambda^k\|L_{\nabla\tg}\nn\\
&\overset{\eqref{ly-cnstr}}{\leq} L_{\nabla f_1}+2\rho_kL_{\nabla\tf_1}+2\mu_kL_\tg^2+2\mu_k\tg_{\rm hi}L_{\nabla \tg}+2\sqrt{2\rho_k\mu_k\vartheta}L_{\nabla\tg}\leq\mu_kL.\label{L-ineq}
\end{align}
It then follows from $\epsilon_K\leq\varepsilon$, \eqref{eps} and Lemmas \ref{l-subdcnstr} and \ref{l-fgap} that \eqref{kkt1}-\eqref{kkt7} hold, which proves statement (i) of Theorem \ref{complexity}. 

(ii) Let $K$ and $N$ be given in \eqref{def-K} and \eqref{def-N}. Recall from Lemmas \ref{l-subp} and \ref{l-init} that the number of evaluations of $\nabla f_1$, $\nabla \tf_1$, $\nabla \tg$, proximal operators of $f_2$ and $\tf_2$ performed by Algorithms \ref{opt} and \ref{mmax-alg2} at iteration $k$ of Algorithm~\ref{AL-alg} is at most $N_k+N_k'$, where $N_k$ and $N_k'$ are given in \eqref{mmax-N} and \eqref{tNk}, respectively. By this and statement (i) of this theorem, one can observe that the total number of evaluations of $\nabla f_1$, $\nabla \tf_1$, $\nabla \tg$ and proximal operators of $f_2$ and $\tf_2$ performed in Algorithm~\ref{AL-alg} is no more than $\sum_{k=0}^K(N_k+N_k')$, respectively. As a result, to prove statement (ii) of this theorem, it suffices to show that $\sum_{k=0}^K(N_k+N_k')\leq N$. 

To this end, using $\mu_k\geq1\geq\epsilon_k$, \eqref{ho}, \eqref{hM}, \eqref{hT}, \eqref{mmax-omega}, \eqref{mmax-zeta}, \eqref{mmax-Mk}, \eqref{mmax-Tk} and \eqref{L-ineq}, we obtain that
\begin{align}
&1\geq\alpha_k\geq\min\left\{1,\sqrt{4\epsilon_k/(\mu_kD_\bfy L)}\right\}\geq\epsilon_k^{1/2}\mu_k^{-1/2}\alpha, \label{alpha-ineq}\\
&\delta_k\leq(2+\epsilon_k^{-1/2}\mu_k^{1/2}\alpha^{-1})\mu_kL (D_\bfx^2+D_\bfy^2)+\max\{1/D_\bfy,\mu_kL/4\}D_\bfy^2\leq\epsilon_k^{-1/2}\mu_k^{3/2}\delta, \label{delta-ineq}\\
&M_k\leq\frac{16\max\left\{1/(4\mu_kL_\tg^2),2/(\epsilon_k^{1/2}\mu_k^{-1/2}\alpha\mu_kL_\tg^2)\right\}\mu_k}{\left[(3\mu_kL+1/(2D_\bfy))^2/\min\{2\mu_kL_\tg^2,\epsilon_k/(2D_\bfy)\}+3\mu_kL+1/(2D_\bfy)\right]^{-2}\epsilon_k^2}\times\Bigg(\epsilon_k^{-1/2}\mu_k^{3/2}\delta \nn \\
&\qquad \ \ +2\epsilon_k^{-1/2}\mu_k^{1/2}\alpha^{-1}\Big(f^*-f_{\rm low}+\rho_k(\tf^*_{\rm hi}-\tf_{\rm low})+\rho_kL_\tf D_\bfy+3\rho_k\vartheta+\mu_k\tg_{\rm hi}^2+\frac{D_\bfy}{4}+\mu_kL(D_\bfx^2+D_\bfy^2)\Big)\Bigg) \label{M-ineq1} \\
&\leq \frac{16\epsilon_k^{-1/2}\mu_k^{-1/2}\max\left\{1/(4L_\tg^2),2/(\alpha L_\tg^2)\right\}\mu_k}{\epsilon_k^2\mu_k^{-4}\left[(3L+1/(2D_\bfy))^2/\min\{2L_\tg^2,1/(2D_\bfy)\}+3L+1/(2D_\bfy)\right]^{-2}\epsilon_k^2}\times(\epsilon_k^{-1/2}\mu_k^{3/2}) \nn \\
& \quad  \times \Bigg(\delta+2\alpha^{-1}\Big(f^*-f_{\rm low}+\tf^*_{\rm hi}-\tf_{\rm low}+L_\tf D_\bfy+3\vartheta+\tg_{\rm hi}^2+\frac{D_\bfy}{4}+L(D_\bfx^2+D_\bfy^2)\Big)\Bigg) \ = \ \epsilon_k^{-5}\mu_k^6M, \label{M-ineq2} \\
&T_k\leq\Bigg\lceil16\left(f_{\rm hi}-f_{\rm low}+\rho_k\epsilon_k+\frac{D_\bfy}{4}\right)\epsilon_k^{-2}\mu_kL+8(1+4D_\bfy^2\mu_k^2L^2\epsilon_k^{-2})\mu_k^{-1}-1\Bigg\rceil_+\leq\epsilon_k^{-2}\mu_kT, \label{T-ineq} 
\end{align}
where \eqref{alpha-ineq} follows from \eqref{ho}, \eqref{mmax-omega} and \eqref{L-ineq};  \eqref{delta-ineq} is due to \eqref{ho}, \eqref{mmax-zeta}, \eqref{alpha-ineq} and $\mu_k\geq1\geq\epsilon_k$; \eqref{M-ineq1} is due to \eqref{mmax-Mk}, \eqref{L-ineq}, \eqref{alpha-ineq}, \eqref{delta-ineq} and $\epsilon_k\in (0,1]$; \eqref{M-ineq2} follows from 
$\mu_k\geq\rho_k\geq1\geq\epsilon_k$ and \eqref{hM}; and \eqref{T-ineq} is due to \eqref{L-ineq}, \eqref{hT} and the fact that $\epsilon_k\in (0,1]$ and $\rho_k\epsilon_k=1$. By the above inequalities, \eqref{mmax-N}, \eqref{L-ineq}, $T>1$ and $\mu_k\geq1\geq\epsilon_k$, one has
\begin{align}
&\sum_{k=0}^KN_k\leq \sum_{k=0}^K \left(\left\lceil96\sqrt{2}\left(1+\left(24\mu_kL+4/D_\bfy\right)/(2\mu_kL_\tg^2)\right)\right\rceil+2\right)\max\left\{2,\sqrt{D_\bfy\mu_kL\epsilon_k^{-1}}\right\}\nn\\
&\qquad\qquad\, \times\left((\epsilon_k^{-2}\mu_kT+1)(\log (\epsilon_k^{-5}\mu_k^6M))_++\epsilon_k^{-2}\mu_kT+1+2\epsilon_k^{-2}\mu_kT\log(\epsilon_k^{-2}\mu_kT+1) \right)\nn\\
&\leq\sum_{k=0}^K\left(\left\lceil96\sqrt{2}\left(1+\left(12L+2/D_\bfy\right)/L_\tg^2\right)\right\rceil+2\right)\max\left\{2,\sqrt{D_\bfy L}\right\}\epsilon_k^{-1/2}\mu_k^{1/2}\nn\\
&\quad \times\epsilon_k^{-2}\mu_k\left((T+1)(\log (\epsilon_k^{-5}\mu_k^6M))_++T+1+2T\log(\epsilon_k^{-2}\mu_kT+1) \right)\nn\\
&\leq\sum_{k=0}^K\left(\left\lceil96\sqrt{2}\left(1+\left(12L+2/D_\bfy\right)/L_\tg^2\right)\right\rceil+2\right)\max\left\{2,\sqrt{D_\bfy L}\right\}\nn\\
&\quad \times\epsilon_k^{-5/2}\mu_k^{3/2}T\left(2(\log (\epsilon_k^{-5}\mu_k^6M))_++2+2\log(2\epsilon_k^{-2}\mu_kT) \right)\nn\\
&\leq\sum_{k=0}^K\left(\left\lceil96\sqrt{2}\left(1+\left(12L+2/D_\bfy\right)/L_\tg^2\right)\right\rceil+2\right)\max\left\{2,\sqrt{D_\bfy L}\right\}T\nn\\
&\quad\times\epsilon_k^{-5/2}\mu_k^{3/2}\left(14\log\mu_k-14\log\epsilon_k+2(\log M)_++2+2\log(2T) \right),\label{sum-N}
\end{align}
where the first inequality follows from $\epsilon_k\in (0,1]$,  \eqref{mmax-N}, \eqref{L-ineq}, \eqref{M-ineq2} and \eqref{T-ineq}, and the second and third inequalities are due to the fact that $\mu_k\geq1\geq\epsilon_k$ and $T>1$. 
By the definition of $K$ in \eqref{def-K}, one has $\tau^K\geq\tau\varepsilon/\epsilon_0$. Also, notice from Algorithm \ref{AL-alg} that $\mu_k=\epsilon_k^{-3}=(\epsilon_0\tau^k)^{-3}$. It then follows from these and \eqref{sum-N} that
\begin{align}
&\sum_{k=0}^KN_k\leq\sum_{k=0}^K\Big(\big\lceil96\sqrt{2}\left(1+\left(12L+2/D_\bfy \right)/L_\tg^2\right)\big\rceil+2\Big)\max\Big\{2,\sqrt{D_\bfy L}\Big\}T\nn\\
&\qquad\qquad\,\times\epsilon_k^{-7}\left(56\log(1/\epsilon_k)+2(\log M)_++2+2\log(2T) \right)\nn\\
&= \left(\big\lceil96\sqrt{2}\left(1+\left(12L+2/D_\bfy \right)/L_\tg^2\right)\big\rceil+2\right)\max\left\{2,\sqrt{D_\bfy L}\right\}T\nn\\
&\quad \times\sum_{k=0}^K\epsilon_0^{-7}\tau^{-7k}\left(56k\log(1/\tau)+56\log(1/\epsilon_0)+2(\log M)_++2+2\log(2T) \right)\nn\\
&\leq \left(\big\lceil96\sqrt{2}\left(1+\left(12L+2/D_\bfy \right)/L_\tg^2\right)\big\rceil+2\right)\max\left\{2,\sqrt{D_\bfy L}\right\}T\nn\\
&\quad\times\sum_{k=0}^K\epsilon_0^{-7}\tau^{-7k}\left(56K\log(1/\tau)+56\log(1/\epsilon_0)+2(\log M)_++2+2\log(2T) \right)\nn\\
&\leq \left(\big\lceil96\sqrt{2}\left(1+\left(12L+2/D_\bfy \right)/L_\tg^2\right)\big\rceil+2\right)\max\Big\{2,\sqrt{D_\bfy L}\Big\}T\epsilon_0^{-7}\nn\\
&\quad \times\tau^{-7K}(1-\tau^7)^{-1}\left(56K\log(1/\tau)+56\log(1/\epsilon_0)+2(\log M)_++2+2\log(2T) \right)\nn\\
&\leq \left(\big\lceil96\sqrt{2}\left(1+\left(12L+2/D_\bfy \right)/L_\tg^2\right)\big\rceil+2\right)\max\left\{2,\sqrt{D_\bfy L}\right\}T\epsilon_0^{-7}(1-\tau^7)^{-1}\nn\\
&\quad\times(\tau\varepsilon/\epsilon_0)^{-7}\left(56K\log(1/\tau)+56\log(1/\epsilon_0)+2(\log M)_++2+2\log(2T) \right),\label{part1}
\end{align}
where the second last inequality is due to $\sum_{k=0}^K\tau^{-7k}\leq \tau^{-7K}/(1-\tau^7)$, and the last inequality follows from $\tau^K\geq\tau\varepsilon/\epsilon_0$.

In addition, observe from \eqref{tLk}, \eqref{hL}, \eqref{ly-cnstr} and $\rho_k^{-1}\mu_k\geq1$, one has that for all $0\leq k\in\bbK-1$,
\begin{equation*}
\tL_k= L_{\nabla\tf_1}+\rho_k^{-1}(\mu_kL_\tg^2+\mu_k\tg_{\rm hi}L_{\nabla \tg}+\|\lambda^k\|L_{\nabla\tg})\leq L_{\nabla\tf_1}+\rho_k^{-1}(\mu_kL_\tg^2+\mu_k\tg_{\rm hi}L_{\nabla \tg}+\sqrt{2\rho_k\mu_k\vartheta}L_{\nabla\tg})\leq\rho_k^{-1}\mu_k\tL.
\end{equation*}
Using this, \eqref{tNk}, $\epsilon_k=\epsilon_0\tau^k$, $\rho_k=\epsilon_k^{-1}$, and $\mu_k=\epsilon_k^{-3}$, we have
\begin{align*}
&\sum_{k=1}^KN_k'\leq\sum_{k=1}^KD_\bfy\sqrt{2\mu_k(\rho_k\epsilon_k)^{-1}\tL}+K=\sum_{k=1}^K\epsilon_k^{-3/2}D_\bfy\sqrt{2\tL}+K=\sum_{k=1}^K\epsilon_0^{-3/2}\tau^{-3k/2}D_\bfy\sqrt{2\tL}+K\\
&\leq\epsilon_0^{-3/2}\tau^{-3K/2}(1-\tau^{3/2})^{-1}D_\bfy\sqrt{2\tL}+K\leq\epsilon_0^{-3/2}(\tau\varepsilon/\epsilon_0)^{-3/2}(1-\tau^{3/2})^{-1}D_\bfy\sqrt{2\tL}+K,
\end{align*}
where the second last inequality is due to $\sum_{k=0}^K\tau^{-3k/2}\leq \tau^{-3K/2}/(1-\tau^{3/2})$, and the last inequality follows from $\tau^K\geq\tau\varepsilon/\epsilon_0$.
This together with \eqref{def-N} and \eqref{part1} implies that $\sum_{k=1}^K(N_k+N_k')\leq N$. Hence, statement (ii) of Theorem \ref{complexity} holds.
\end{proof}

In the remainder of this section, we first establish several lemmas and then use them to prove Theorem \ref{complexity-str}. In particular, the following lemma provides an operation complexity of step 3 of Algorithm~\ref{AL-alg} for  problem \eqref{prob} with $\sigma>0$, i.e., $\tf_1(x,\cdot)$ being strongly convex with parameter $\sigma$ for any given $x\in\dom\,f_2$.

\begin{lemma}\label{l-subp-str}
Suppose that Assumption \ref{a1} holds with $\sigma>0$, i.e., $\tf_1(x,\cdot)$ being strongly convex with parameter $\sigma$ for any given $x\in\dom\,f_2$.  Let $f^*$, $L_k$, $\tf^*_{\rm hi}$, $D_\bfx$, $D_\bfy$, $f_{\rm hi}$, $f_{\rm low}$ $\tf_{\rm low}$, $\tg_{\rm hi}$, $\bbK$ and $\vartheta$ be defined in \eqref{prob}, \eqref{Lk}, \eqref{def-tFx}, \eqref{DxDy}, \eqref{fhi}, \eqref{tfhi}, \eqref{def-K} and \eqref{ht}, $L_\tf$ and $\sigma$ be given in Assumption \ref{a1}, $\epsilon_k$, $\rho_k$ and $\mu_k$ be given in Algorithm \ref{AL-alg}, and
\begin{align}
\tilde\alpha_k=&\ \min\left\{1,\sqrt{8\sigma\rho_k/L_k}\right\},\label{mmax-omega-str}\\
\tilde\delta_k=&\ (2+\tilde\alpha_k^{-1})(D_\bfx^2+D_\bfy^2)L_k +\max\left\{2\sigma\rho_k,\tilde\alpha_k L_k/4\right\}D_\bfy^2,\label{mmax-zeta-str}\\
\widetilde M_k=&\ \frac{16\max\left\{1/(2L_k),\min\left\{1/(2\sigma\rho_k),4/(\tilde\alpha_k L_k)\right\}\right\}}{\left[9L_k^2/\min\{L_k,\sigma\rho_k\}+ 3L_k\right]^{-2}\epsilon_k^2}\nn\\
&\times\left(\tilde\delta_k+2\tilde\alpha_k^{-1}\left(f^*-f_{\rm low}+\rho_k(\tf^*_{\rm hi}-\tf_{\rm low})+\rho_kL_\tf D_\bfy+3\rho_k\vartheta+\mu_k\tg_{\rm hi}^2+L_k(D_\bfx^2+D_\bfy^2)\right)\right),\label{mmax-Mk-str}\\
\widetilde T_k=&\ \left\lceil16\left(f_{\rm hi}-f_{\rm low}+\rho_k\epsilon_k\right) L_k\epsilon_k^{-2}+8\sigma^{-2}\rho_k^{-2}L_k^2+7\right\rceil_+,\label{mmax-Tk-str}\\
\widetilde N_k=&\ 3397\max\left\{2,\sqrt{L_k/(2\sigma\rho_k)}\right\}\left((\widetilde T_k+1)(\log \widetilde M_k)_++\widetilde T_k+1+2\widetilde T_k\log(\widetilde T_k+1) \right).\label{mmax-N-str}
\end{align}
Then for all $0\leq k\in\bbK-1$, an $\epsilon_k$-primal-dual stationary point $(x^{k+1},y^{k+1},z^{k+1})$ of problem \eqref{AL-sub} is successfully found at step 3 of Algorithm~\ref{AL-alg} that satisfies 
\begin{align}
\max_z\mcL(x^{k+1},y^{k+1},z,\lambda^k;\rho_k,\mu_k)\leq&\  f_{\rm hi}+\rho_k\epsilon_k+\frac{1}{2\epsilon_k^2}\left(L_k^{-1}+\sigma^{-2}\rho_k^{-2}L_k\right).\label{upperbnd-k-str}
\end{align}
Moreover, the total number of evaluations of $\nabla f_1$, $\nabla \tf_1$, $\nabla \tg$ and proximal operators of $f_2$ and $\tf_2$ performed at step 3 in iteration $k$ of Algorithm~\ref{AL-alg} is no more than $\widetilde N_k$, respectively.
\end{lemma}

\begin{proof}
Observe from \eqref{Lag} and Assumption \ref{a1} that problem \eqref{AL-sub} can be viewed as
\[
\min_{x,y}\max_z\{h(x,y,z)+p(x,y)-q(z)\}
\]
with
\begin{align*}
&h(x,y,z)=f_1(x,y)+\rho_k\tf_1(x,y)+\frac{1}{2\mu_k}\|[\lambda^k+\mu_k \tg(x,y)]_+\|^2-\rho_k \tf_1(x,z)-\frac{1}{2\mu_k}\|[\lambda^k+\mu_k \tg(x,z)]_+\|^2,\nn\\
&p(x,y)=f_2(x)+\rho_k\tf_2(y),\quad q(z)=\rho_k\tf_2(z).
\end{align*}
By \eqref{tghi} and Assumption \ref{a1}, it can be verified that $\|[\lambda^k+\mu_k\tg (x,y)]_+\|^2/(2\mu_k)$ and $\|[\lambda^k+\mu_k\tg (x,z)]_+\|^2/(2\mu_k)$ are both $(\mu_kL_\tg^2+\mu_k\tg_{\rm hi}L_{\nabla\tg}+\|\lambda^k\| L_{\nabla\tg})$-smooth on $\mcX\times\mcY$. Using this,  Assumption \ref{a1} with $\sigma>0$, and the fact that $f_1$ and $\tf_1$ are respectively $L_{\nabla f_1}$- and $L_{\nabla\tf_1}$-smooth on $\mcX\times\mcY$, we can see that $h(x,y,\cdot)$ is $\sigma\rho_k$-strongly-concave on $\mcY$, and  $h(x,y,z)$ is $L_k$-smooth on $\mcX\times\mcY\times\mcY$ for all $0\leq k\in\bbK-1$, where $L_k$ is given in \eqref{Lk}. Consequently, it follows from Theorem \ref{mmax-thm} (see Appendix \ref{appendix-B}) that an $\epsilon_k$-primal-dual stationary point $(x^{k+1},y^{k+1},z^{k+1})$ of problem \eqref{AL-sub} is successfully found by Algorithm \ref{mmax-alg2} at step 3 of Algorithm~\ref{AL-alg}.

In addition, by \eqref{Lag}, \eqref{tLag} and \eqref{fhi}, one has
\begin{align}
\min_{x,y}\max_z\mcL(x,y,z,\lambda^k;\rho_k,\mu_k)\overset{\eqref{Lag}\eqref{tLag}}{=}&\min_{x,y}\left\{f(x,y)+\rho_k\tL(x,y,\lambda^k;\rho_k,\mu_k)-\min_z\rho_k\tL(x,z,\lambda^k;\rho_k,\mu_k)\right\}\nn\\
\geq\ \ &\min_{(x,y)\in\mcX\times\mcY}f(x,y)\overset{\eqref{fhi}}{=} f_{\rm low}.\label{e1-str}
\end{align}
Let $(x^*,y^*)$ be an optimal solution of \eqref{prob}. It then follows that $f(x^*,y^*)=f^*$, $\tf(x^*,y^*)=\tf^*(x^*)$ and $\tg(x^*,y^*)\leq0$, where $f^*$ and $\tf^*$ are defined in \eqref{prob} and \eqref{tfstarx}, respectively. Using these, \eqref{Lag}, \eqref{tLag}, \eqref{def-tFx}, \eqref{tfhi} and \eqref{ly-cnstr}, we obtain that
\begin{align}
&\ \ \min_{x,y}\max_z\mcL(x,y,z,\lambda^k;\rho_k,\mu_k)\leq\max_z\mcL(x^*,y^*,z,\lambda^k;\rho_k,\mu_k)\nn\\
&\ \ \overset{\eqref{Lag}\eqref{tLag}}{=}f(x^*,y^*)+\rho_k\tf(x^*,y^*)+\frac{1}{2\mu_k}\|[\lambda^k+\mu_k\tg(x^*,y^*)]_+\|^2-\min_z\rho_k\tL(x^*,z,\lambda^k;\rho_k,\mu_k)\nn\\
&\quad\, \leq f^*+\rho_k\tf^*(x^*)+\frac{1}{2\mu_k}\|\lambda^k\|^2-\min_z\left\{\rho_k\tf(x^*,z)+\frac{1}{2\mu_k}\|[\lambda^k+\mu_k\tg(x^*,z)]_+\|^2\right\}\nn\\
&\ \overset{\eqref{def-tFx}\eqref{tfhi}}{\leq} f^*+\rho_k(\tf^*_{\rm hi}-\tf_{\rm low})+\frac{1}{2\mu_k}\|\lambda^k\|^2\overset{\eqref{ly-cnstr}}{\leq} f^*+\rho_k(\tf^*_{\rm hi}-\tf_{\rm low})+\rho_k\vartheta,\label{e2-str}
\end{align}
where the second inequality is due to $\tf(x^*,y^*)=\tf^*(x^*)$, $\tg(x^*,y^*)\leq0$ and \eqref{tLag}. Also, by \eqref{Lag}, \eqref{DxDy}, \eqref{fhi}, \eqref{tfhi} and \eqref{ly-cnstr}, one has
\begin{align}
&\min_{(x,y,z)\in\mcX\times\mcY\times\mcY}\mcL(x,y,z,\lambda^k;\rho_k,\mu_k)\nn\\
&\overset{\eqref{Lag}}{\geq}\min_{(x,y,z)\in\mcX\times\mcY\times\mcY}\left\{f(x,y)+\rho_k(\tf(x,y)-\tf(x,z))-\frac{1}{2\mu_k}\|[\lambda^k+\mu_k\tg(x,z)]_+\|^2\right\}\nn\\
&\ \geq\min_{(x,y,z)\in\mcX\times\mcY\times\mcY}\left\{f(x,y)-\rho_kL_\tf\|y-z\|-\frac{1}{2\mu_k}\left(\|\lambda^k\|+\mu_k\|[\tg(x,z)]_+\|\right)^2\right\}\nn\\
&\ \geq\min_{(x,y,z)\in\mcX\times\mcY\times\mcY}\left\{f(x,y)-\rho_kL_\tf\|y-z\|-\frac{1}{\mu_k}\|\lambda^k\|^2-\mu_k\|[\tg(x,z)]_+\|^2\right\}\nn\\
&\ \geq f_{\rm low}-\rho_kL_\tf D_\bfy-2\rho_k\vartheta-\mu_k\tg_{\rm hi}^2,\label{e3-str}
\end{align}
where the second inequality is due to $\lambda^k\in\bR_+^l$ and $L_\tf$-Lipschitz continuity of $\tf$ (see Assumption \ref{a1}(i)), and the last inequality is due to \eqref{DxDy}, \eqref{fhi}, \eqref{tfhi} and \eqref{ly-cnstr}. Notice from step 2 of Algorithm~\ref{AL-alg} that $y_{\rm init}^k$ is an approximate solution of $\min_z\tL(x^{k},z,\lambda^k;\rho_k,\mu_k)$ satisfying \eqref{y-nf}. It then follows from \eqref{Lag}, \eqref{tLag}, \eqref{y-nf} and \eqref{fhi} that
\begin{align}
\max_z\mcL(x^k,y_{\rm init}^k,z,\lambda^k;\rho_k,\mu_k)\overset{\eqref{Lag}\eqref{tLag}}=&f(x^k,y_{\rm init}^k)+\rho_k\left(\tcL(x^k,y_{\rm init}^k,\lambda^k;\rho_k,\mu_k)-\min_z\tcL(x^k,z,\lambda^k;\rho_k,\mu_k)\right)\nn\\
\overset{\eqref{y-nf}}{\leq}\ &f(x^k,y_{\rm init}^k)+\rho_k\epsilon_k\overset{\eqref{fhi}}{\leq} f_{\rm hi}+\rho_k\epsilon_k.\label{init-bnd-str}
\end{align}

To complete the rest of the proof, let
\begin{align}
&H(x,y,z)=\mcL(x,y,z,\lambda^k;\rho_k,\mu_k),\quad H^*=\min_{x,y}\max_z\mcL(x,y,z,\lambda^k;\rho_k,\mu_k),\label{Hk1-str}\\
&H_{\rm low}=\min\left\{\mcL(x,y,z,\lambda^k;\rho_k,\mu_k)|(x,y,z)\in\mcX\times\mcY\times\mcY\right\}.\label{Hk2-str}
\end{align}
In view of these, \eqref{e1-str}, \eqref{e2-str}, \eqref{e3-str} and \eqref{init-bnd-str}, we obtain that
\begin{align*}
&\max_zH(x^k,y_{\rm init}^k,z)\overset{\eqref{init-bnd-str}}{\leq} f_{\rm hi}+\rho_k\epsilon_k,\\
& f_{\rm low}\overset{\eqref{e1-str}}{\leq} H^*\overset{\eqref{e2}}{\leq} f^*+\rho_k(\tf^*_{\rm hi}-\tf_{\rm low})+\rho_k\vartheta,\\
& H_{\rm low}\overset{\eqref{e3-str}}{\geq} f_{\rm low}-\rho_kL_\tf D_\bfy-2\rho_k\vartheta-\mu_k\tg_{\rm hi}^2.
\end{align*}
Using these and Theorem \ref{mmax-thm} (see Appendix \ref{appendix-B}) with $\hat x^0=(x^k,y_{\rm init}^k)$, $\epsilon=\epsilon_k$, $\hat\epsilon_0=\epsilon_k/2$, $\hat\sigma_y=\sigma_y=\sigma\rho_k$, $L_{\nabla\bh}=L_k$, $\widehat L=3L_k$, $\hat\alpha=\alpha_k$, $\hat\delta=\delta_k$, $D_p=\sqrt{D_\bfx^2+D_\bfy^2}$, $D_q=D_\bfy$, and $\bH$, $\bH^*$, $\bH_{\rm low}$ given in \eqref{Hk1-str} and \eqref{Hk2-str}, we can conclude that 
the $\epsilon_k$-primal-dual stationary point $(x^{k+1},y^{k+1},z^{k+1})$ of problem \eqref{AL-sub} found at step 3 of Algorithm~\ref{AL-alg}  satisfies \eqref{upperbnd-k}. Moreover, the total number of evaluations of $\nabla f_1$, $\nabla \tf_1$, $\nabla \tg$ and proximal operators of $f_2$ and $\tf_2$ performed by Algorithm \ref{mmax-alg2} at step 3 of Algorithm~\ref{AL-alg}  is no more than $\widetilde N_k$, respectively.
\end{proof}

The next lemma presents an upper bound on the optimality violation of $y^{k+1}$ for the lower-level problem of \eqref{prob} when $\sigma>0$ and $x=x^{k+1}$.

\begin{lemma}\label{l-fgap-str}
Suppose that Assumptions \ref{a1} and \ref{a2} hold with $\sigma>0$, i.e., $\tf_1(x,\cdot)$ is strongly convex with parameter $\sigma$ for any given $x\in\dom\ f_2$. Let $\tf^*$, $L_k$, $D_\bfy$, $f_{\rm hi}$, $f_{\rm low}$ and $\bbK$ be defined in \eqref{tfstarx}, \eqref{Lk}, \eqref{DxDy}, \eqref{fhi} and \eqref{def-K}, $L_f$, $L_\tf$, $\sigma$ and $G$ be given in Assumptions \ref{a1} and \ref{a2}, and $\epsilon_k$, $\rho_k$, $\mu_k$ and $\lambda^0$ be given in Algorithm \ref{AL-alg}. Suppose that $(x^{k+1},y^{k+1},\lambda^{k+1})$ is generated by Algorithm \ref{AL-alg} for some $0\leq k\in\bbK-1$ satisfying \eqref{muk-1}. Then we have
\begin{align*}
|\tf(x^{k+1},y^{k+1})-\tf^*(x^{k+1})|\leq\max\Bigg\{&2\mu_k^{-1}G^{-2}L_\tf(\epsilon_0+L_f+\rho_kL_\tf)D_\bfy^2,\nn\\
&\ \rho_k^{-1}\mu_k^{-1}\max\{\|\lambda^0\|,\ 2G^{-1}(\epsilon_0+\rho_kL_\tf)D_\bfy\}/2\nn\\
&+\rho_k^{-1 }\left(f_{\rm hi}-f_{\rm low}+\rho_k\epsilon_k+\frac{1}{2\epsilon_k^2}\left(L_k^{-1}+\sigma^{-2}\rho_k^{-2}L_k\right)\right)\Bigg\}.
\end{align*}
\end{lemma}

\begin{proof}
Using \eqref{Lag}, \eqref{tLag}, \eqref{fhi}, \eqref{p-ineq}, and \eqref{lambday-bnd2}, we have
\begin{align*}
&\ \ \max_z\mcL(x^{k+1},y^{k+1},z,\lambda^k;\rho_k,\mu_k)\\
&\ \ \overset{\eqref{Lag}\eqref{tLag}}= f(x^{k+1},y^{k+1})+\rho_k \tf(x^{k+1},y^{k+1})+\frac{1}{2\mu_k}\|[\lambda^k+\mu_k\tg(x^{k+1},y^{k+1})]_+\|^2-\min_z\rho_k\tcL(x^{k+1},z,\lambda^k;\rho_k,\mu_k) \\ 
&\ \ \geq f(x^{k+1},y^{k+1})+\rho_k \tf(x^{k+1},y^{k+1})-\min_z\rho_k\tcL(x^{k+1},z,\lambda^k;\rho_k,\mu_k) \\ 
&\ \overset{\eqref{fhi}\eqref{p-ineq}}{\geq}f_{\rm low}+\rho_k\big(\tf(x^{k+1},y^{k+1})-\tf^*(x^{k+1})\big)-\frac{1}{2\mu_k}\|\lambda^k\|^2\\
&\quad \overset{\eqref{lambday-bnd2}}{\geq}f_{\rm low}+\rho_k\big(\tf(x^{k+1},y^{k+1})-\tf^*(x^{k+1})\big)-\mu_k^{-1}\max\{\|\lambda^0\|,\ 2G^{-1}(\epsilon_0+\rho_kL_\tf)D_\bfy\}/2.
\end{align*}
This together with \eqref{upperbnd-k-str} implies that
\begin{align}
\tf(x^{k+1},y^{k+1})-\tf^*(x^{k+1})\ \leq&\ \rho_k^{-1 }\left(f_{\rm hi}-f_{\rm low}+\rho_k\epsilon_k+\frac{1}{2\epsilon_k^2}\left(L_k^{-1}+\sigma^{-2}\rho_k^{-2}L_k\right)\right)\nn\\
&+\rho_k^{-1}\mu_k^{-1}\max\{\|\lambda^0\|,\ 2G^{-1}(\epsilon_0+\rho_kL_\tf)D_\bfy\}/2.\label{f-gap-str}
\end{align}
On the other hand, let $\lambda^*\in\bR^l_+$ be an optimal Lagrangian multiplier of problem \eqref{tfstarx} with $x=x^{k+1}$. It then follows from Lemma \ref{dual-bnd}(i) that  $\|\lambda^*\|\leq G^{-1}L_\tf D_\bfy$. Using these, \eqref{tfstarx} and \eqref{x-feas}, we have
\begin{align*}
\tf^*(x^{k+1})=&\min_y\left\{\tf(x^{k+1},y)+\langle\lambda^*,\tg(x^{k+1},y)\rangle\right\}\leq\tf(x^{k+1},y^{k+1})+\langle\lambda^*,\tg(x^{k+1},y^{k+1})\rangle\\
\leq & \ \tf(x^{k+1},y^{k+1})+\|\lambda^*\|\|[\tg(x^{k+1},y^{k+1})]_+\|\leq\tf(x^{k+1},y^{k+1})+2\mu_k^{-1}G^{-2}L_\tf(\epsilon_0+L_f+\rho_kL_\tf)D_\bfy^2.
\end{align*}
The conclusion of this lemma then follows from this and \eqref{f-gap-str}.
\end{proof}

The following lemma provides an estimate on operation complexity at step 2 of Algorithm~\ref{AL-alg} for  problem \eqref{prob} with $\sigma>0$, i.e., $\tf_1(x,\cdot)$ being strongly convex with parameter $\sigma$  for any given $x\in\dom\,f_2$.

\begin{lemma}\label{l-init-str}
Suppose that Assumptions \ref{a1} and \ref{a2} hold with $\sigma>0$, i.e., $\tf_1(x,\cdot)$ being strongly convex with parameter $\sigma$  for any given $x\in\dom\,f_2$.  Let $\tL_k$, $D_\bfy$ and $\bbK$ be defined in \eqref{tLk}, \eqref{DxDy} and \eqref{def-K}, $\sigma$ be given in Assumption \ref{a1}, $\epsilon_k$ be given in Algorithm \ref{AL-alg}, and 
\beq\label{tNk-str}
\widetilde N_k'=2\left\lceil\sqrt{\tL_k\sigma^{-1}}\,\right\rceil\max\left\{1,\left\lceil2\log(2\epsilon_k^{-1}\tL_kD_\bfy^2)\right\rceil\right\}+1.
\eeq
Then for all $0\leq k\in\bbK-1$, $y_{\rm init}^k$ satisfying \eqref{y-nf} is found at step 2 of  Algorithm \ref{AL-alg} by Algorithm \ref{opt} in no more than $\widetilde N_k'$ evaluations of $\nabla \tf_1$, $\nabla\tg$ and the proximal operator of $\tf_2$, respectively.
\end{lemma}

\begin{proof}
 Notice from \eqref{tLag} and Algorithm \ref{AL-alg} that $y_{\rm init}^k$ satisfying \eqref{y-nf} is found by Algorithm \ref{opt} applied to the problem
\begin{equation*}
\min_y\left\{\tL(x^k,y,\lambda^k;\rho_k,\mu_k)=\phi(y)+P(y)\right\},
\end{equation*}
where $\phi(y)=\tf_1(x^k,y)+\|[\lambda^k+\mu_k\tg(x^k,y)]_+\|^2/(2\rho_k\mu_k)$ and $P(y)=\tf_2(y)$. By Assumption \ref{a1} with $\sigma>0$ and \eqref{tghi}, one can see that $\phi$ is $\sigma$-strongly-convex and $\tL_k$-smooth on $\dom\,P$ with $\tL_k$ given in \eqref{tLk}. It then follows from this, Theorem \ref{thm-opt} (see Appendix \ref{appendix-A}) and \eqref{opt-gap} with $\tilde\epsilon=\epsilon_k$, $D_P=D_\bfy$, $\sigma_\phi=\sigma$ and $L_{\nabla\phi}=\tL_k$ that Algorithm \ref{opt-str} finds $y_{\rm init}^k$ satisfying \eqref{y-nf} in no more than $\widetilde T_k'$ iterations, where
\begin{equation*}
\widetilde T_k'=\left\lceil\sqrt{\tL_k\sigma^{-1}}\,\right\rceil\max\left\{1,\left\lceil2\log(2\epsilon_k^{-1}\tL_kD_\bfy^2)\right\rceil\right\}.
\end{equation*}
Notice that the first step of Algorithm \ref{opt} requires one evaluation of $\nabla\phi$ and the proximal operator of $P$, respectively, and each iteration of Algorithm \ref{opt} requires two evaluation of $\nabla\phi$ and the proximal operator of $P$, respectively. Hence, the conclusion of this lemma holds.
\end{proof}

We are now ready to prove Theorem \ref{complexity-str}.

\begin{proof}[\textbf{Proof of Theorem \ref{complexity-str}}]
(i) Recall from the proof of Theorem \ref{complexity} that \eqref{L-ineq} holds. It then follows from this, $\epsilon_K\leq\varepsilon$, \eqref{eps} and Lemmas \ref{l-subdcnstr} and \ref{l-fgap-str} that \eqref{kkt1}-\eqref{kkt6} and \eqref{kkt7-str} hold, which proves statement (i) of Theorem \ref{complexity-str}. 

(ii) Let $K$, $\tilde\alpha$, $\tilde\delta$, $\widetilde M$, $\widetilde T$ and $\widetilde N$ be given in \eqref{def-K}, \eqref{ho-str}, \eqref{hM-str}, \eqref{hT-str} and \eqref{def-N-str}, respectively. Recall from Lemmas \ref{l-subp-str} and \ref{l-init-str} that the number of evaluations of $\nabla f_1$, $\nabla \tf_1$, $\nabla \tg$, proximal operators of $f_2$ and $\tf_2$ performed by Algorithms \ref{opt-str} and \ref{mmax-alg2} at iteration $k$ of Algorithm~\ref{AL-alg} is at most $\widetilde N_k+\widetilde N_k'$, where $\widetilde N_k$ and $\widetilde N_k'$ are given in \eqref{mmax-N-str} and \eqref{tNk-str}, respectively. By this and statement (i) of this theorem, one can observe that the total number of evaluations of $\nabla f_1$, $\nabla \tf_1$, $\nabla \tg$ and proximal operators of $f_2$ and $\tf_2$ performed in Algorithm~\ref{AL-alg} is no more than $\sum_{k=0}^K(\widetilde N_k+\widetilde N_k')$, respectively. As a result, to prove statement (ii) of this theorem, it suffices to show that $\sum_{k=0}^K(\widetilde N_k+\widetilde N_k')\leq \widetilde N$. 

To this end,  using $\mu_k\geq\rho_k\geq1\geq\epsilon_k$, \eqref{ho-str}, \eqref{hM-str}, \eqref{hT-str}, \eqref{L-ineq} \eqref{mmax-omega-str}, \eqref{mmax-zeta-str}, \eqref{mmax-Mk-str} and \eqref{mmax-Tk-str}, we obtain that
\begin{align}
&1\geq\tilde\alpha_k\geq\min\left\{1,\sqrt{8\sigma\rho_k/(\mu_k L)}\right\}\geq\rho_k^{1/2}\mu_k^{-1/2}\tilde\alpha, \label{alpha-ineq-str}\\
&\tilde\delta_k\leq(2+\rho_k^{-1/2}\mu_k^{1/2}\tilde\alpha^{-1})(D_\bfx^2+D_\bfy^2)\mu_kL+\max\{2\sigma\rho_k,\mu_kL/4\}D_\bfy^2\leq\rho_k^{-1/2}\mu_k^{3/2}\tilde\delta, \label{delta-ineq-str}\\
&\widetilde M_k\leq\frac{16\max\left\{1/(4\mu_kL_\tg^2),2/(\rho_k^{1/2}\mu_k^{-1/2}\tilde\alpha\mu_kL_\tg^2)\right\}}{\left[9\mu_k^2L^2/\min\{2\mu_kL_\tg^2,\sigma\rho_k\}+3\mu_kL\right]^{-2}\epsilon_k^2}\times\Big(\rho_k^{-1/2}\mu_k^{3/2}\tilde\delta \nn \\
&\qquad\quad+2\rho_k^{-1/2}\mu_k^{1/2}\tilde\alpha^{-1}\big(f^*-f_{\rm low}+\rho_k(\tf^*_{\rm hi}-\tf_{\rm low})+\rho_kL_\tf D_\bfy+3\rho_k\vartheta+\mu_k\tg_{\rm hi}^2+\mu_kL (D_\bfx^2+D_\bfy^2)\big)\Big) \label{M-ineq1-str} \\
&\leq \frac{16\rho_k^{-1/2}\mu_k^{-1/2}\max\left\{1/(4L_\tg^2),2/(\tilde\alpha L_\tg^2)\right\}}{\rho_k^2\mu_k^{-4}\left[9L^2/\min\{2L_\tg^2,\sigma\}+3L\right]^{-2}\epsilon_k^2}\times\rho_k^{-1/2}\mu_k^{3/2} \nn \\
& \quad  \times \Big(\tilde\delta+2\tilde\alpha^{-1}\big(f^*-f_{\rm low}+\tf^*_{\rm hi}-\tf_{\rm low}+L_\tf D_\bfy+3\vartheta+\tg_{\rm hi}^2+L(D_\bfx^2+D_\bfy^2)\big)\Big) \ = \epsilon_k^{-2}\rho_k^{-3}\mu_k^5\widetilde M, \label{M-ineq2-str} \\
&\widetilde T_k\leq\Big\lceil16\left(f_{\rm hi}-f_{\rm low}+\rho_k\epsilon_k\right)\epsilon_k^{-2}\mu_kL+8\sigma^{-2}\rho_k^{-2}\mu_k^2L^2+7\Big\rceil_+\leq\epsilon_k^{-2}\mu_k\widetilde T, \label{T-ineq-str} 
\end{align}
where \eqref{alpha-ineq-str} follows from \eqref{ho-str}, \eqref{L-ineq} and \eqref{mmax-omega-str};  \eqref{delta-ineq-str} is due to \eqref{ho-str}, \eqref{mmax-zeta-str}, \eqref{alpha-ineq-str} and $\mu_k\geq1\geq\epsilon_k$; \eqref{M-ineq1-str} is due to \eqref{L-ineq}, \eqref{mmax-Mk-str}, \eqref{alpha-ineq-str}, \eqref{delta-ineq-str} and $\epsilon_k\in (0,1]$; \eqref{M-ineq2-str} follows from 
$\mu_k\geq\rho_k\geq1\geq\epsilon_k$ and \eqref{hM-str}; and \eqref{T-ineq-str} is due to \eqref{hT-str}, \eqref{L-ineq} and the fact that $\epsilon_k\in (0,1]$ and $\rho_k\epsilon_k=1$. By the above inequalities, \eqref{L-ineq}, \eqref{mmax-N-str}, $\widetilde T>1$ and $\mu_k\geq1\geq\epsilon_k$, one has
\begin{align}
&\sum_{k=0}^K\widetilde N_k\leq \sum_{k=0}^K 3397\max\left\{2,\sqrt{\mu_kL/(2\sigma\rho_k)}\right\}\nn\\
&\qquad\qquad\,\times\left((\epsilon_k^{-2}\mu_k\widetilde T+1)(\log (\epsilon_k^{-2}\rho_k^{-3}\mu_k^5\widetilde M))_++\epsilon_k^{-2}\mu_k\widetilde T+1+2\epsilon_k^{-2}\mu_k\widetilde T\log(\epsilon_k^{-2}\mu_k\widetilde T+1) \right)\nn\\
&\leq\sum_{k=0}^K3397\max\left\{2,\sqrt{L/(2\sigma)}\right\}\rho_k^{-1/2}\mu_k^{1/2}\times\epsilon_k^{-2}\mu_k\left((\widetilde T+1)(\log (\epsilon_k^{-2}\rho_k^{-3}\mu_k^5\widetilde M))_++\widetilde T+1+2\widetilde T\log(\epsilon_k^{-2}\mu_k\widetilde T+1) \right)\nn\\
&\leq\sum_{k=0}^K3397\max\left\{2,\sqrt{L/(2\sigma)}\right\}\epsilon_k^{-2}\rho_k^{-1/2}\mu_k^{3/2}\widetilde T\left(2(\log (\epsilon_k^{-2}\rho_k^{-3}\mu_k^5\widetilde M))_++2+2\log(2\epsilon_k^{-2}\mu_k\widetilde T) \right)\nn\\
&\leq\sum_{k=0}^K3397\max\left\{2,\sqrt{L/(2\sigma)}\right\}\widetilde T\epsilon_k^{-2}\rho_k^{-1/2}\mu_k^{3/2}\left(12\log\mu_k-6\log\rho_k-8\log\epsilon_k+2(\log \widetilde M)_++2+2\log(2\widetilde T) \right),\label{sum-N-str}
\end{align}
where the first inequality follows from $\epsilon_k\in (0,1]$, \eqref{L-ineq}, \eqref{mmax-N-str}, \eqref{M-ineq2-str} and \eqref{T-ineq-str}, and the second and third inequalities are due to the fact that $\mu_k\geq1\geq\epsilon_k$ and $\widetilde T>1$. 
By the definition of $K$ in \eqref{def-K}, one has $\tau^K\geq\tau\varepsilon/\epsilon_0$. Also, notice from Algorithm \ref{AL-alg} that $\rho_k=\epsilon_k^{-1}=(\epsilon_0\tau^k)^{-1}$ and $\mu_k=\epsilon_k^{-3}=(\epsilon_0\tau^k)^{-3}$. It then follows from these and \eqref{sum-N-str} that
\begin{align}
&\sum_{k=0}^K\widetilde N_k\leq\sum_{k=0}^K3397\max\left\{2,\sqrt{L/(2\sigma)}\right\}\widetilde T \epsilon_k^{-6}\left(38\log(1/\epsilon_k)+2(\log \widetilde M)_++2+2\log(2\widetilde T) \right)\nn\\
&= 3397\max\left\{2,\sqrt{L/(2\sigma)}\right\}\widetilde T\sum_{k=0}^K\epsilon_0^{-6}\tau^{-6k}\left(38k\log(1/\tau)+38\log(1/\epsilon_0)+2(\log \widetilde M)_++2+2\log(2\widetilde T) \right)\nn\\
&\leq 3397\max\left\{2,\sqrt{L/(2\sigma)}\right\}\widetilde T\sum_{k=0}^K\epsilon_0^{-6}\tau^{-6k}\left(38K\log(1/\tau)+38\log(1/\epsilon_0)+2(\log \widetilde M)_++2+2\log(2\widetilde T) \right)\nn\\
&\leq 3397\max\left\{2,\sqrt{L/(2\sigma)}\right\}\widetilde T\epsilon_0^{-6}\tau^{-6K}(1-\tau^6)^{-1}\left(38K\log(1/\tau)+38\log(1/\epsilon_0)+2(\log \widetilde M)_++2+2\log(2\widetilde T) \right)\nn\\
&\leq3397\max\left\{2,\sqrt{L/(2\sigma)}\right\}\widetilde T\epsilon_0^{-6}(1-\tau^6)^{-1}\nn\\
&\ \ \ \times(\tau\varepsilon/\epsilon_0)^{-6}\left(38K\log(1/\tau)+38\log(1/\epsilon_0)+2(\log \widetilde M)_++2+2\log(2\widetilde T) \right),\label{part1-str}
\end{align}
where the second last inequality is due to $\sum_{k=0}^K\tau^{-6k}\leq \tau^{-6K}/(1-\tau^6)$, and the last inequality follows from $\tau^K\geq\tau\varepsilon/\epsilon_0$.

In addition, observe from \eqref{tLk}, \eqref{hL}, \eqref{ly-cnstr} and $\rho_k^{-1}\mu_k\geq1$, one has that for all $0\leq k\in\bbK-1$,
\begin{equation*}
\tL_k= L_{\nabla\tf_1}+\rho_k^{-1}(\mu_kL_\tg^2+\mu_k\tg_{\rm hi}L_{\nabla \tg}+\|\lambda^k\|L_{\nabla\tg})\leq L_{\nabla\tf_1}+\rho_k^{-1}(\mu_kL_\tg^2+\mu_k\tg_{\rm hi}L_{\nabla \tg}+\sqrt{2\rho_k\mu_k\vartheta}L_{\nabla\tg})\leq\rho_k^{-1}\mu_k\tL.
\end{equation*}
Using this, \eqref{tNk-str}, $\epsilon_k=\epsilon_0\tau^k$, $\rho_k=\epsilon_k^{-1}$, and $\mu_k=\epsilon_k^{-3}$, we have
\begin{align*}
&\sum_{k=1}^K\widetilde N_k'\leq\sum_{k=1}^K\left(2\left\lceil\sqrt{\frac{\rho_k^{-1}\mu_k\tL}{\sigma}}\,\right\rceil\max\left\{1,\left\lceil2\log\left(2\epsilon_k^{-1}\rho_k^{-1}\mu_k\tL D_\bfy^2\right)\right\rceil\right\}+1\right)\nn\\
&=\sum_{k=1}^K2\left\lceil(\epsilon_0\tau^k)^{-1}\sqrt{\frac{\tL}{\sigma}}\,\right\rceil\max\left\{1,\left\lceil2\log(2\tL D_\bfy^2)+6k\log(1/\tau)-6\log\epsilon_0\right\rceil\right\}+K\nn\\
&\leq\sum_{k=1}^K2(\epsilon_0\tau^k)^{-1}\left\lceil\sqrt{\frac{\tL}{\sigma}}+1\right\rceil\max\left\{1,\left\lceil2\log(2\tL D_\bfy^2)+6K\log(1/\tau)-6\log\epsilon_0\right\rceil\right\}+K\nn\\
&\leq2\epsilon_0^{-1}\tau^{-K}(1-\tau)\left\lceil\sqrt{\frac{\tL}{\sigma}}+1\right\rceil\max\left\{1,\left\lceil2\log(2\tL D_\bfy^2)+6K\log(1/\tau)-6\log\epsilon_0\right\rceil\right\}+K\nn\\
&\leq2(\tau\varepsilon)^{-1}(1-\tau)\left\lceil\sqrt{\frac{\tL}{\sigma}}+1\right\rceil\max\left\{1,\left\lceil2\log(2\tL D_\bfy^2)+6K\log(1/\tau)-6\log\epsilon_0\right\rceil\right\}+K
\end{align*}
where the second last inequality is due to $\sum_{k=0}^K\tau^{-k}\leq \tau^{-K}/(1-\tau)$, and the last inequality follows from $\tau^K\geq\tau\varepsilon/\epsilon_0$. 
This together with \eqref{def-N-str} and \eqref{part1-str} implies that $\sum_{k=1}^K(\widetilde N_k+\widetilde N_k')\leq \widetilde N$. Hence, statement (ii) of Theorem \ref{complexity-str} holds.
\end{proof}

\appendix
\section{Optimal first-order methods for unconstrained convex optimization problems}\label{appendix-A}

In this part we review optimal first-order methods for solving convex optimization problem
\beq\label{cvx-prob}
\Psi^*=\min_x\{\Psi(x)\coloneqq\phi(x)+P(x)\},
\eeq
where $P:\bR^m\to(-\infty,\infty]$ are closed convex functions, $\phi:\bR^m\to(-\infty,\infty]$ is a $\sigma_\phi$-strongly-convex function with $\sigma_\phi\geq 0$, and $\nabla\phi$ is $L_{\nabla\phi}$-Lipschitz continuous on $\dom\;P$. In addition, we assume that $\dom\;P$ is compact and let $D_P:=\max_{x,y\in\dom\;P}\|x-y\|$.

We first present an optimal first-order method in Algorithm \ref{opt} for solving problem \eqref{cvx-prob} with $\sigma_\phi=0$, i.e., $\phi$ being convex but not strongly convex.  It is a variant of Nesterov's optimal first-order method \cite{Nest04-1} and has been studied in, for example,  \cite[Section 3]{Tse08}.
\begin{algorithm}[H]
\caption{An optimal first-order method for problem \eqref{cvx-prob} with $\sigma_\phi=0$}\label{opt}
\begin{algorithmic}[1]
\REQUIRE $\tilde\epsilon>0$, $\tilde x^0\in\dom\;P$ and $x^0=z^0=\tilde x^0$.
\FOR{$k=0,1,\dots$}
\STATE Set $y^k=(kx^k+2z^k)/(k+2)$.
\STATE Compute $z^{k+1}$ as
\begin{equation*}
z^{k+1}=\argmin_z\left\{\ell(z;y^k)+\frac{L_{\nabla\phi}}{k+2}\|z-z^k\|^2\right\},
\end{equation*}
where
\beq\label{def-ell}
\ell(x;y):=\phi(y)+\langle\nabla\phi(y),x-y\rangle+P(x).
\eeq
\STATE Set $x^{k+1}=(kx^k+2z^{k+1})/(k+2)$.
\STATE Terminate the algorithm and output $x^{k+1}$ if 
\[
\Psi(x^{k+1})-\underline\Psi_{k+1}\leq\tilde\epsilon,\qquad\mbox{where}\quad\underline\Psi_{k+1}=\frac{4}{(k+1)(k+3)}\min\left\{\sum_{i=0}^k\frac{i+2}{2}\ell(x;y^i)\right\}.
\]
\ENDFOR
\end{algorithmic}
\end{algorithm}

The following result provides an iteration complexity of Algorithm \ref{opt} for finding an $\tilde\epsilon$-optimal solution\footnote{An $\tilde\epsilon$-optimal solution of 
problem \eqref{cvx-prob} is a point $x$ satisfying $\Psi(x) - \Psi^* \leq \tilde\epsilon$.} 
of \eqref{cvx-prob}. It is an immediate consequence of \cite[Corrolary]{Tse08} and its proof is thus omitted.
\begin{thm}\label{thm-opt}
Let $\{(x^k,y^k)\}$ be generated by Algorithm \ref{opt} and $\ell(\cdot;\cdot)$ be defined in \eqref{def-ell}. Then, $\Psi(x^k)-\Psi^*\leq\Psi(x^k)-\underline\Psi_k$ for all $k\geq1$. Moreover, for any given $\tilde\epsilon>0$, Algorithm \ref{opt} finds an approximate solution $x^{k+1}$ of problem \eqref{cvx-prob} such that $\Psi(x^{k+1})-\Psi^*\leq\Psi(x^{k+1})-\underline\Psi_{k+1}\leq\tilde\epsilon$ in no more than $\bar K$ iterations, where
\begin{equation*}
\bar K=\left\lceil D_P\sqrt{2L_{\nabla\phi} {\tilde\epsilon}^{-1}}\right\rceil.
\end{equation*}
\end{thm}

We next present an optimal first-order method \cite[Algorithm 4]{lu2018iteration} for solving problem \eqref{cvx-prob} with $\sigma_\phi>0$, i.e., $\phi$ being strongly convex with parameter $\sigma_\phi$, which is a slight variant of Nesterov's optimal first-order methods \cite{lin2015accelerated,Nest04-1}.

\begin{algorithm}[H]
\caption{An optimal first-order method for problem \eqref{cvx-prob} with $\sigma_\phi>0$}\label{opt-str}
\begin{algorithmic}[1]
\REQUIRE $\tilde\epsilon>0$ and $\tilde x^0\in\dom\;P$.
\STATE Compute
\begin{equation*}
x^0=\prox_{P/L_{\nabla\phi}}\left(\tilde x^0-L^{-1}_{\nabla\phi}\nabla\phi(\tilde x^0)\right).
\end{equation*}
\STATE Set $z^0=x^0$ and $\alpha=\sqrt{\sigma_\phi/L_{\nabla\phi}}$.
\FOR{$k=0,1,\dots$}
\STATE Set $y^k=(x^k+\alpha z^k)/(1+\alpha)$.
\STATE Compute $z^{k+1}$ as
\begin{equation*}
z^{k+1}=\argmin_z\left\{\ell(z;y^k)+\frac{\alpha L_{\nabla\phi}}{2}\|z-\alpha y^k-(1-\alpha)z^k\|^2\right\},
\end{equation*}
where $\ell(x;y)$ is defined in \eqref{def-ell}.
\STATE Set $x^{k+1}=(1-\alpha)x^k+\alpha z^{k+1}$.
\STATE Compute
\begin{equation*}
\tilde x^{k+1}=\prox_{P/L_{\nabla\phi}}\left(x^{k+1}-L^{-1}_{\nabla\phi}\nabla\phi(x^{k+1})\right).
\end{equation*}
\STATE Terminate the algorithm and output $\tilde x^{k+1}$ if 
\begin{equation*}
\|\tilde x^{k+1}-x^k\|\leq\frac{\tilde\epsilon}{2L_{\nabla\phi}D_P}.
\end{equation*}
\ENDFOR
\end{algorithmic}
\end{algorithm}

The following result provides an iteration complexity of Algorithm \ref{opt-str} for finding an approximate optimal solution of problem \eqref{cvx-prob}, which was established in \cite[Proposition 4]{lu2018iteration}.

\begin{thm}\label{thm-opt-str}
Let $\{\tilde x^k\}$ be the sequence generated by Algorithm \ref{opt-str}. Then for any given $\tilde\epsilon>0$, an approximate solution $\tilde x^{k+1}$ of problem \eqref{cvx-prob} satisfying $\dist(0,\partial\Psi(\tilde x^{k+1}))\leq2L_{\nabla\phi}\|\tilde x^{k+1}-x^{k+1}\|\leq\tilde\epsilon/D_P$ is generated by running Algorithm \ref{opt} for at most $\widetilde K$ iterations, where
\begin{equation*}
\widetilde K=\left\lceil\sqrt{\frac{L_{\nabla\phi}}{\sigma_\phi}}\,\right\rceil\max\left\{1,\left\lceil2\log\frac{2L_{\nabla\phi} D_P^2}{\tilde\epsilon}\right\rceil\right\}.
\end{equation*}
\end{thm}

\begin{rem}
By the convexity of $\Psi$, $D_P=\max_{x,y\in\dom\;P}\|x-y\|$, and Theorem \ref{thm-opt-str}, it is not hard to show that the output $\tilde x^{k+1}$ of Algorithm \ref{opt-str} satisfies 
\beq\label{opt-gap}
\Psi(\tilde x^{k+1})-\Psi^*\leq \dist(0,\partial\Psi(\tilde x^{k+1})) D_P  \leq\tilde\epsilon.
\eeq
\end{rem}

\section{A first-order method for nonconvex-concave minimax problem} \label{appendix-B}

In this part, we present a first-order method for finding an $\epsilon$-primal-dual stationary point of the nonconvex-concave minimax problem
\begin{equation}\label{ap-prob}
\bH^* = \min_x\max_y\left\{\bH(x,y)\coloneqq \bh(x,y)+p(x)-q(y)\right\},
\end{equation}
which has at least one optimal solution and satisfies the following assumptions.
\begin{assumption}\label{mmax-a}
\begin{enumerate}[label=(\roman*)]
\item $p:\bR^n\to\bR\cup\{\infty\}$ and $q:\bR^m\to\bR\cup\{\infty\}$ are proper convex functions and continuous on $\dom\,p$ and $\dom\,q$, respectively, and moreover, $\dom\,p$ and $\dom\,q$ are compact.
\item The proximal operators associated with $p$ and $q$ can be exactly evaluated.
\item $\bh$ is $L_{\nabla\bh}$-smooth on $\dom\,p\times\dom\,q$, and moreover, $\bh(x,\cdot)$ is $\sigma_y$-strongly-concave with $\sigma_y\geq 0$ for any given $x\in\dom\,p$. 
\end{enumerate}
\end{assumption}

For ease of presentation, we define
\begin{align}
&D_p=\max\{\|u-v\|\big|u,v\in\dom\,p\},\quad D_q=\max\{\|u-v\|\big|u,v\in\dom\,q\},\label{ap-D}\\
&H_{\rm low}=\min\{H(x,y)|(x,y)\in\dom\, p\times\dom\, q\}.\label{ap-H}
\end{align}

Recently, a first-order method was proposed in \cite[Algorithm 2]{lu2024first} for finding an $\epsilon$-primal-dual stationary point of problem \eqref{ap-prob} with $\sigma_y=0$, while another first-order method was proposed in \cite[Algorithm 1]{lu2024strongminimax} for finding an $\epsilon$-primal-dual stationary point of problem \eqref{ap-prob} with $\sigma_y>0$. We will present a unified first-order method in Algorithm \ref{mmax-alg2} below by combining these two methods. Specifically, given an iterate $(x^k,y^k)$, this unified first-order method finds the next iterate $(x^{k+1},y^{k+1})$ by applying  \cite[Algorithm 1]{lu2024first}, which is a slight modification of a novel optimal first-order method \cite[Algorithm 4]{kovalev2022first} by incorporating a forward-backward splitting scheme and a verifiable termination criterion (see steps 23-25 in Algorithm \ref{mmax-alg1}), to the strongly-convex-strongly-concave minimax problem
\[
\min_x\max_y\{\bh_k(x,y)+p(x)-q(y)\},
\]
where
\beq\label{hk}
\bh_k(x,y)=\left\{\begin{aligned}
&\bh(x,y)-\epsilon\|y-y^0\|^2/(4D_q)+L_{\nabla \bh}\|x-x^k\|^2,\quad&\mbox{if}\quad\sigma_y=0,\\
&\bh(x,y)+L_{\nabla \bh}\|x-x^k\|^2,\quad&\mbox{if}\quad\sigma_y>0.
\end{aligned}\right.
\eeq
This minimax problem arises from applying a proximal point method to the minimization problem $\min_x \{\max_y \bh(x,y)-q(y)-\epsilon\|y-y^0\|^2/(4D_q)+p(x)\}$ if $\sigma_y=0$ or the minimization problem $\min_x\{\max_y \bh(x,y)-q(y)+p(x)\}$ if $\sigma_y>0$. One can easily observe that $h_k$ is $L_{\nabla h}$-strongly-convex-$\hat\sigma_y$-strongly-concave and $\widehat L$-smooth on $\dom\,p\times\dom\, q$, where
\begin{align}
&\hat\sigma_y=\left\{\begin{aligned}
&\epsilon/(2D_q),\quad&\mbox{if}\quad\sigma_y=0,\\
&\sigma_y,\quad&\mbox{if}\quad\sigma_y>0,
\end{aligned}\right.
\qquad
\widehat L=\left\{\begin{aligned}
&3L_{\nabla h}+\epsilon/(2D_q),\quad&\mbox{if}\quad\sigma_y=0,\\
&3L_{\nabla h},\quad&\mbox{if}\quad\sigma_y>0.
\end{aligned}\right.
\label{hsigmaL}
\end{align}

Before presenting a unified first-order method for problem \eqref{ap-prob}, we first present the modified optimal first-order method \cite[Algorithm 1]{lu2024first} in Algorithm \ref{mmax-alg1} below for solving a general strongly-convex-strongly-concave minimax problem
\begin{equation}\label{ea-prob}
\min_{x}\max_{y}\left\{ \h(x,y)+p(x)-q(y)\right\},
\end{equation}
where 
$\h(x,y)$ is $\bar\sigma_x$-strongly-convex-$\bar\sigma_y$-strongly-concave and $L_{\nabla\h}$-smooth on $\dom\,p\times\dom\,q$ for some $\bar\sigma_x,\bar\sigma_y>0$.  The functions $\hat h$, $a^k_x$ and $a^k_y$ arising in Algorithm \ref{mmax-alg1} are defined as follows:
\begin{align*}
&\hat h(x,y)=\h(x,y)-\bar\sigma_x\|x\|^2/2+\bar\sigma_y\|y\|^2/2,\\
&a^k_x(x,y)=\nabla_x\hat h(x,y)+\bar\sigma_x(x-\bar\sigma_x^{-1}z^k_g)/2,\quad a^k_y(x,y)=-\nabla_y\hat h(x,y)+\bar\sigma_y y+\bar\sigma_x(y-y^k_g)/8,
\end{align*}
where $y^k_g$ and $z^k_g$ are generated at iteration $k$ of Algorithm \ref{mmax-alg1} below. 

\begin{algorithm}[H]
\caption{A modified optimal first-order method for problem \eqref{ea-prob}}
\label{mmax-alg1}
\begin{algorithmic}[1]
\REQUIRE $\tau>0$, $\bar z^0=z^0_f\in-\bar\sigma_x\dom\,p$,\footnote{} $\bar y^0=y^0_f\in\dom\,q$, $(z^0,y^0)=(\bar z^0, \bar y^0)$,  $\bar \alpha=\min\left\{1,\sqrt{8\bar\sigma_y/\bar\sigma_x}\right\}$, $\eta_z=\bar\sigma_x/2$, $\eta_y=\min\left\{1/(2\bar\sigma_y),4/(\bar \alpha\bar\sigma_x)\right\}$, $\beta_t=2/(t+3)$, $\zeta=\left(2\sqrt{5}(1+8L_{\nabla\h}/\bar\sigma_x)\right)^{-1}$, $\gamma_x=\gamma_y=8\bar\sigma_x^{-1}$, and $\hat\zeta=\min\{\bar\sigma_x,\bar\sigma_y\}/L_{\nabla \h}^2$.
\FOR{$k=0,1,2,\ldots$}
\STATE $(z^k_g,y^k_g)=\bar \alpha(z^k,y^k)+(1-\bar \alpha)(z^k_f,y^k_f)$.
\STATE $(x^{k,-1},y^{k,-1})=(-\bar\sigma_x^{-1}z^k_g,y^k_g)$.
\STATE $x^{k,0}=\prox_{\zeta\gamma_xp}(x^{k,-1}-\zeta\gamma_x a^k_x(x^{k,-1},y^{k,-1}))$.
\STATE $y^{k,0}=\prox_{\zeta\gamma_y q}(y^{k,-1}-\zeta\gamma_y a^k_y(x^{k,-1},y^{k,-1}))$.
\STATE $b^{k,0}_x=\frac{1}{\zeta\gamma_x}(x^{k,-1}-\zeta\gamma_x a^k_x(x^{k,-1},y^{k,-1})-x^{k,0})$.
\STATE $b^{k,0}_y=\frac{1}{\zeta\gamma_y}(y^{k,-1}-\zeta\gamma_y a^k_y(x^{k,-1},y^{k,-1})-y^{k,0})$.
\STATE $t=0$.
\WHILE{\\ $\gamma_x\|a^k_x(x^{k,t},y^{k,t})+b^{k,t}_x\|^2+\gamma_y\|a^k_y(x^{k,t},y^{k,t})+b^{k,t}_y\|^2>\gamma_x^{-1}\|x^{k,t}-x^{k,-1}\|^2+\gamma_y^{-1}\|y^{k,t}-y^{k,-1}\|^2$\\~~}
\STATE $x^{k,t+1/2}=x^{k,t}+\beta_t(x^{k,0}-x^{k,t})-\zeta\gamma_x(a^k_x(x^{k,t},y^{k,t})+b^{k,t}_x)$.
\STATE $y^{k,t+1/2}=y^{k,t}+\beta_t(y^{k,0}-y^{k,t})-\zeta\gamma_y(a^k_y(x^{k,t},y^{k,t})+b^{k,t}_y)$.
\STATE $x^{k,t+1}=\prox_{\zeta\gamma_x p}(x^{k,t}+\beta_t(x^{k,0}-x^{k,t})-\zeta\gamma_x a^k_x(x^{k,t+1/2},y^{k,t+1/2}))$.
\STATE $y^{k,t+1}=\prox_{\zeta\gamma_y q}(y^{k,t}+\beta_t(y^{k,0}-y^{k,t})-\zeta\gamma_y a^k_y(x^{k,t+1/2},y^{k,t+1/2}))$.
\STATE $b^{k,t+1}_x=\frac{1}{\zeta\gamma_x}(x^{k,t}+\beta_t(x^{k,0}-x^{k,t})-\zeta\gamma_x a^k_x(x^{k,t+1/2},y^{k,t+1/2})-x^{k,t+1})$.
\STATE $b^{k,t+1}_y=\frac{1}{\zeta\gamma_y}(y^{k,t}+\beta_t(y^{k,0}-y^{k,t})-\zeta\gamma_y a^k_y(x^{k,t+1/2},y^{k,t+1/2})-y^{k,t+1})$.
\STATE $t \leftarrow t+1$.
\ENDWHILE
\STATE $(x^{k+1}_f,y^{k+1}_f)=(x^{k,t},y^{k,t})$.
\STATE $(z^{k+1}_f,w^{k+1}_f)=(\nabla_x\hat h(x^{k+1}_f,y^{k+1}_f)+b^{k,t}_x,-\nabla_y\hat h(x^{k+1}_f,y^{k+1}_f)+b^{k,t}_y)$.
\STATE $z^{k+1}=z^k+\eta_z\bar\sigma_x^{-1}(z^{k+1}_f-z^k)-\eta_z(x^{k+1}_f+\bar\sigma_x^{-1}z^{k+1}_f)$.
\STATE $y^{k+1}=y^k+\eta_y\bar\sigma_y(y^{k+1}_f-y^k)-\eta_y(w^{k+1}_f+\bar\sigma_yy^{k+1}_f)$.
\STATE $x^{k+1}=-\bar\sigma_x^{-1}z^{k+1}$.
\STATE $\tx^{k+1}=\prox_{\hat\zeta p}(x^{k+1}-\hat\zeta\nabla_x\h(x^{k+1},y^{k+1}))$.
\STATE $\ty^{k+1}=\prox_{\hat\zeta q}(y^{k+1}+\hat\zeta\nabla_y\h(x^{k+1},y^{k+1}))$.
\STATE Terminate the algorithm and output $(\tx^{k+1},\ty^{k+1})$ if
\begin{equation*}
\|\hat\zeta^{-1}(x^{k+1}-\tx^{k+1},\ty^{k+1}-y^{k+1})-(\nabla \h(x^{k+1},y^{k+1})-\nabla \h(\tx^{k+1},\ty^{k+1}))\|\leq\tau.
\end{equation*}
\ENDFOR
\end{algorithmic}							
\end{algorithm}
\footnotetext{For convenience, $-\bar\sigma_x\dom\,p$ stands for the set $\{-\bar\sigma_x u|u\in\dom\,p\}$.}

We now present a first-order method for finding an $\epsilon$-primal-dual stationary point of problem \eqref{ap-prob} in Algorithm \ref{mmax-alg2} below by unifying  \cite[Algorithm 2]{lu2024first} and \cite[Algorithm 1]{lu2024strongminimax}.

\begin{algorithm}[H]
\caption{A first-order method for problem~\eqref{ap-prob}}
\label{mmax-alg2}
\begin{algorithmic}[1]
\REQUIRE $\epsilon>0$, $\hat\epsilon_0\in(0,\epsilon/2]$, $(\hat x^0,\hat y^0)\in\dom\,p\times\dom\,q$, $(x^0,y^0)=(\hat x^0,\hat y^0)$,  and $\hat\epsilon_k=\hat\epsilon_0/(k+1)$.
\FOR{$k=0,1,2,\ldots$}
\STATE Call Algorithm~\ref{mmax-alg1} with $\h\leftarrow \bh_k$, $\tau \leftarrow \hat\epsilon_k$, $\bar\sigma_x\leftarrow L_{\nabla \bh}$, $\bar\sigma_y\leftarrow \hat\sigma_y$, $L_{\nabla \h}\leftarrow \widehat L$, $\bar z^0=z^0_f\leftarrow-\bar\sigma_x x^k$, $\bar y^0=y^0_f\leftarrow y^k$, and denote its output by $(x^{k+1},y^{k+1})$, where $h_k$ is given in \eqref{hk}, $\hat\sigma_y$ and $\widehat L$ are given in \eqref{hsigmaL}.
\STATE Terminate the algorithm and output $(\xe,\ye)=(x^{k+1},y^{k+1})$ if
\begin{equation*}
\|x^{k+1}-x^k\|\leq\epsilon/(4L_{\nabla \bh}).
\end{equation*}
\ENDFOR
\end{algorithmic}
\end{algorithm}

The following theorem presents complexity results for Algorithm \ref{mmax-alg2}, which is a combination of  \cite[Theorem 2]{lu2024first} for $\sigma_y=0$ and \cite[Theorem 1]{lu2024strongminimax} for $\sigma_y>0$.

\begin{thm}[{\bf Complexity of Algorithm \ref{mmax-alg2}}]\label{mmax-thm}
Suppose that Assumption~\ref{mmax-a} holds. Let $\bH^*$, $H$, $D_p$, $D_q$, $\bH_{\rm low}$, $\hat\sigma_y$ and $\widehat L$ be defined in \eqref{ap-prob}, 
\eqref{ap-D}, \eqref{ap-H} and \eqref{hsigmaL}, $L_{\nabla \bh}$ and $\sigma_y$ be given in Assumption \ref{mmax-a}, $\epsilon$, $\hat\epsilon_0$ and $\hat x^0$ be given in Algorithm~\ref{mmax-alg2}, and 
\begin{align*}
\hat\alpha=&\ \min\left\{1,\sqrt{8\hat\sigma_y/L_{\nabla \bh}}\right\},\\
\hat\delta=&\ (2+\hat\alpha^{-1})L_{\nabla \bh} D_p^2+\max\left\{2\hat\sigma_y,\hat\alpha L_{\nabla \bh}/4\right\}D_q^2,\\
\widehat T=&\ \left\lceil16(\max_y\bH(\hat x^0,y)-\bH^*+(\hat\sigma_y-\sigma_y)D_q^2/2)L_{\nabla \bh}\epsilon^{-2}+32\hat\epsilon_0^2(1+\hat\sigma_y^{-2}L_{\nabla \bh}^2)\epsilon^{-2}-1\right\rceil_+,\\
\widehat N=&\ \left(\left\lceil96\sqrt{2}\left(1+8\widehat LL_{\nabla \bh}^{-1}\right)\right\rceil+2\right)\max\Big\{2,\sqrt{L_{\nabla \bh}/(2\hat\sigma_y)}\Big\}\notag\\
&\ \times\Bigg((\widehat T+1)\Bigg(\log\frac{4\max\left\{\frac{1}{2L_{\nabla \bh}},\min\left\{\frac{1}{2\hat\sigma_y},\frac{4}{\hat\alpha L_{\nabla \bh}}\right\}\right\}\left(\hat\delta+2\hat\alpha^{-1}(\bH^*-\bH_{\rm low}+(\hat\sigma_y-\sigma_y)D_q^2/2+L_{\nabla \bh} D_p^2)\right)}{\left[\widehat L^2/\min\{L_{\nabla \bh},\hat\sigma_y\}+ \widehat L\right]^{-2}\hat\epsilon_0^2}\Bigg)_+\notag\\
&\ +\widehat T+1+2\widehat T\log(\widehat T+1) \Bigg).
\end{align*}
Then Algorithm~\ref{mmax-alg2} terminates and outputs an $\epsilon$-primal-dual stationary point $(\xe,\ye)$ of \eqref{ap-prob} in at most $\widehat T+1$ outer iterations that satisfies 
\begin{equation*}
\max_y\bH(\xe,y)\leq \max_y\bH(\hat x^0,y)+(\hat\sigma_y-\sigma_y)D_q^2/2+2\hat\epsilon_0^2\left(L_{\nabla \bh}^{-1}+\hat\sigma_y^{-2}L_{\nabla \bh}\right).
\end{equation*}
Moreover, the total number of evaluations of $\nabla \bh$ and proximal operators of $p$ and $q$ performed in Algorithm~\ref{mmax-alg2} is no more than $\widehat N$, respectively.
\end{thm}

\end{document}